\RequirePackage[displaymath, mathlines]{lineno}
\documentclass[12pt]{article}
\usepackage{Sweave}

\usepackage{graphicx}        
\usepackage{multicol}        

\usepackage{amsfonts}
\usepackage[latin1]{inputenc}
\usepackage[T1]{fontenc}

\usepackage{amssymb}
\usepackage{amsthm}

\usepackage{bm}
\usepackage{bbold}
\usepackage{graphicx}
\usepackage{csquotes}
\usepackage{verbatim} 
\usepackage{eurosym}
\usepackage{siunitx}
\usepackage{hhline}
\usepackage{colortbl}
\usepackage[table]{xcolor}
\usepackage{multirow}
\usepackage{tikz-cd} 
\usepackage{csquotes}
\usepackage{caption}
\usepackage{subcaption}
\usepackage[framemethod=tikz]{mdframed}
\usepackage{array}
\usepackage{amsmath, amscd}
\usepackage{amsthm}

\newtheorem{theorem}{Theorem}[section]

\newtheorem{proposition}[theorem]{Proposition}

\newcommand{\clr}{\mbox{clr}}

\title{Changing reference measure in Bayes spaces with applications to functional data analysis}

\author{R. Talsk\'a$^{a\ast}$, A.~Menafoglio$^b$, K. Hron$^a$, J.J.~Egozcue$^c$, \\ J.~Palarea-Albaladejo$^d$\\
$^{a}$Department of Mathematical Analysis and Applications of Mathematics, \\ Faculty of Science, Palack\'y University Olomouc, Olomouc, Czech Republic \\ 
$^{b}$MOX -- Department of Mathematics, Politecnico di Milano, Milano, Italy \\
$^{c}$Department of Civil and Environmental Engineering, \\ Universitat Polit\`ecnica de Catalunya, Barcelona, Spain \\
$^{d}$Biomathematics and Statistics Scotland, Edinnburgh, Scotland, UK\\
$^\ast$Corresponding author. Email: talskarenata@seznam.cz
}

\begin{document} 

\maketitle 

\begin{abstract}
Probability density functions (PDFs) can be understood as continuous compositions by the theory of Bayes spaces. The origin of a Bayes space is determined by a given reference measure. This can be easily changed through the well-known chain rule which has an impact on the geometry of the Bayes space.
This work provides a mathematical framework for setting a reference measure. It is used to develop a weighting scheme on the bounded domain of distributional data. The impact on statistical analysis is shown from the perspective of simplicial functional principal component analysis. Moreover, a novel centered log-ratio transformation is proposed to map a weighted Bayes spaces into an unweighted $L^2$ space, enabling to use most tools developed in functional data analysis (e.g. clustering, regression analysis, etc.) while accounting for the weighting strategy.
The potential of our proposal is shown through simulation and on a real case study using Italian income data.
\end{abstract}

\section{Introduction}
Data collected through surveys occur frequently in the form of distributional data. These typically result from the discretization of probability density functions (PDFs), with a particular case being histogram data for continuous random variables.
This has motivated an increasing interest in the development of statistical methods for the analysis of distribution or density data \cite{bivand17,hron16,menafoglio16,menafoglio16b,petersen16}.
Although Functional Data Analysis (FDA) \cite{ramsay05} may potentially provide a wide range of methodological tools for this,
FDA methods are typically designed for data embedded in the $L^2$ space of square-integrable functions. As such, they cannot be applied directly to densities, since the metric of $L^2$ does not account for their peculiar properties. The need of developing novel methodological frameworks able to correctly characterize the data through the use of non-standard geometries is nowadays becoming widely recognized in FDA, not only for density data, but also for other kinds of constrained functional data (see, e.g., \cite{BongiornoGoia2018,CanaleVantini2016,RossiniCanale2018}).

A PDF is a non-negative Borel measurable function constrained to integrate to a constant, conventionally set to one. Several authors \cite{egozcue13,egozcue06,boogaart10,boogaart14} noted that PDFs have a \emph{relative} nature, in the sense that the meaningful information is embedded in the relative contribution of the probability of any (Borel) subset of the domain of the random variable generating the data to the overall probability, i.e. the measure of the whole set (so-called \emph{total}).
Changing the value of the total by multiplying the PDF by a positive real constant results in a scaled density conveying the same \emph{relative} information (which is known as the \textit{scale invariance} property). As a consequence, the actual total is in fact irrelevant for the purpose of the analysis, as e.g. assumed in Bayesian statistics \cite{gelman13}. The total used simply determines a representative of the equivalence class of proportional density functions.

The relative nature of PDFs motivates the use of the so-called log-ratio approach -- a well-established methodology for the analysis of compositional data. These are vectors describing quantitatively the parts of some whole, and are frequently represented as constrained data (e.g. proportions, percentages) carrying relative information \cite{aitchison86, pawlowsky15}. PDFs can be then interpreted as the continuous counterparts of compositions, i.e., as compositions with infinitely many parts. This has recently motivated the construction of the so-called Bayes Hilbert spaces \cite{boogaart14}, whose geometry results from the generalization of the Aitchison geometry for compositional data \cite{egozcue03} to the infinite-dimensional case. Although Bayes spaces allow to deal with both unbounded and bounded domains for the PDFs,  the latter case has been mainly considered so far in practice, and it will be the main focus in this work. 

Bayes Hilbert spaces can be defined only if a reference measure has been set. 
This measure can be arbitrarily chosen for statistical analysis. Nonetheless, it should be remarked that this choice has a direct impact on the geometry of the Bayes space, and it plays the role of origin of the space \cite{boogaart14}. So far, the Lebesgue measure has been a default choice for several real-world applications and most literature revolves around it. However, adopting a different reference measure is indeed needed to deal with PDFs on possibly unbounded supports.

As discussed in the multivariate case \cite{egozcue16}, changing the uniform reference measure can be interpreted as introducing a (non-uniform) weighting of parts in compositional analysis. This is a key point in practice, as rarely all parts of the composition have the same importance. Such weighting can be indeed relevant to consider a relative scale on the domain of a distributional variable \cite{mateu08}. For instance, when analyzing income distributions, changes in the low income stratum (e.g. an increase of 100 {\euro} for an income of 1000 {\euro} per month) are typically of greater importance than the same absolute differences for higher earners  (e.g. increase of 100 {\euro} for an income of 10,000 {\euro} per month). Accordingly, a sensible weighting strategy may be aimed at emphasizing the variability in the bottom of the domain. Weighting scheme can also be considered to account for imprecise values near the detection limit of a measurement device, or for combined effects. The choice of the reference measure should be thus driven by the purpose of the analysis, in order to down-weight (or up-weight) ``parts'' of the domain -- i.e. subdomains -- that may have lesser (or greater) importance for the analysis.

The aim of this work is to frame the previous considerations into a rigorous mathematical setting for Bayes spaces built upon general reference measures, and to develop it further in order to facilitate its practical use. We shall provide solid theoretical basis as well as clear guidelines for the use of non-uniform reference measures, and explore the consequences of changing the reference measure from uniform to non-uniform. We shall particularly illustrate and emphasize the consequences for simplicial functional principal component analysis (SFPCA), which is a dimensionality reduction method for PDFs in (unweighted) Bayes spaces proposed in \cite{hron16}.

The rest of this article is organized as follows. Section \ref{B2} introduces Bayes Hilbert spaces built upon a general reference measure and discusses in detail their properties. We then present our key contributions concerning the mathematical framework for statistical analysis in weighted Bayes spaces. Particular attention will be paid to the centered log-ratio (clr) transformation for general reference measures, and to a novel \emph{unweighting} clr which allows for the extension of several FDA methods to the weighted Bayes space setting. Amongst these, functional principal component analysis in weighted Bayes spaces is discussed in Section \ref{subsec:wSFPCA}. The consequences of changing the reference measure are further explored in Section \ref{sec:simu} through simulated data. In Section \ref{sec:casestudy}, the  developments are illustrated in a real-world application using Italian income distributions in different regions. Finally, Section \ref{sec:conclusion} concludes with some final remarks.


\section{The geometry of Bayes spaces of measures}\label{B2}
\subsection{Bayes spaces with a reference measure}\label{subsec:BayesHilbert}

The Bayes space methodology provides a Hilbert space representation for density functions. The aim of the following section is to summarize the basics of the Bayes space methodology with respect to general reference measures. Similar to the case of compositional data \cite{egozcue16}, the choice of a reference measure other than the standard uniform one induces weighting effects on the domain of densities. The reason for that might be relative scale of the domain itself as motivated by the income data application mentioned in the previous section, but for example also to highlight information about the dispersion of distributions around their mean when the mean distribution is taken as the reference measure.

Following \cite{boogaart14}, consider two $\sigma$-finite positive real-valued measures $\mu$ and $\nu$ on a measurable space $(\Omega, \mathcal{A})$, with $\Omega\subset\mathbb{R}$ and $\mathcal{A}$ being a sigma-algebra. The measures $\mu$ and $\nu$ are $\mathcal{B}$-equivalent (denoted by $\nu =_{\mathcal{B}} \mu$) if they are proportional. That is, if there exists a positive real constant $c$ such that,  for any subset $\mbox{A} \in \mathcal{A}$, $\mu(\mbox{A}) = c \cdot \nu(\mbox{A})$.
By considering a measure $\mu$ on $(\Omega, \mathcal{A})$ such that $\mu(\Omega)=1$ (i.e., a probability measure), we single out a particular representative within a $\mathcal{B}$-equivalence class of proportional measures which provides the same \emph{relative} information. Indeed, this is typically quantified through the (log-)ratios $\mu(\mbox{A}_1)/\mu(\mbox{A}_2)$, with $\mbox{A}_1$ and $\mbox{A}_2$ in $\mathcal{A}$, which are clearly invariant within the $\mathcal{B}$-equivalence class (i.e. \emph{scale invariance} is satisfied).
In the following, the representative of an equivalence class will always be considered to be a probability measure for interpretative purposes. For this reason, measures on a bounded support will be considered in the following, although the theoretical background in \cite{boogaart14} is of general validity.

When a reference measure $\mathsf{P}$ is fixed, any measure $\mu$ can be identified with its density w.r.t. the reference measure. A typical choice for $\mathsf{P}$ is the Lebesgue measure, restricted here to a bounded support. In the following, the notation $\lambda$ will be kept for this measure, although it is in fact a uniform measure that is $\mathcal{B}$-equivalent to the Lebesgue measure restricted to the bounded support. 
Given a measure $\mu$ and its density $f = d \mu / d \mathsf{P}$ with respect to $\mathsf{P}$, the probability measure of any event
$\mbox{A} \in \mathcal{A}$ is
$$\mu(\mbox{A})=\int_{\mbox{A}} f \, d \mathsf{P} = \int_{\mbox{A}} \frac{d \mu}{d \mathsf{P}} \,
d \mathsf{P}.$$

Note that the choice of the reference measure is not scale invariant, because it reflects on the scale of the entire Bayes space. For instance, the Lebesgue measure on a domain $\Omega = [a,b]$ belongs to the $\mathcal{B}$-equivalence class of the uniform measure $\mathsf{P}_0$ on $\Omega$. Clearly, $\lambda$ has density $d \lambda / d \lambda = 1$ with respect to itself, whereas it has density $d \lambda / d \mathsf{P}_0 = b-a$ w.r.t. $\mathsf{P}_0$. Thus, a rescaling of the reference measure determines a rescaling of the \emph{total}. For example, when $\lambda$ is considered, the total is set to $\lambda(\Omega)=b-a$, whereas $\mathsf{P}_0$ is associated with a total equal to $\mathsf{P}_0(\Omega)=1$.
On the other hand, once the scale of the reference measure is fixed, the corresponding densities satisfy the scale invariance property.
For instance, having set the reference measure on $\Omega = [a,b]$ to $\lambda$, the Lebesgue density $d\lambda / d \lambda$ and the uniform density $d\mathsf{P}_0 / d \lambda = \frac{1}{b-a}$ are equivalent.
As such, it will always be necessary to distinguish which kind of information (relative or absolute) is conveyed by the reference measures as this matters for the analysis.

To change the reference measure from $\lambda$ to a measure $\mathsf{P}$ with strictly positive $\lambda$-density $p = d \mathsf{P}/ d\lambda$, we can use the well-known chain rule. That is, for a generic measure $\mu$ we have that
$$
\mu(\mbox{A})= \int_{\mbox{A}}\frac{d \mu}{d \lambda} \, d \lambda = \int_{\mbox{A}} \frac{d \mu}{d \lambda} \cdot \frac{d \lambda}{d \mathsf{P}}\, d \mathsf{P} = \int_{\mbox{A}} \frac{d \mu}{d \lambda} \cdot \frac{1}{p}\, d \mathsf{P}.$$

Given a $\sigma$-finite measure $\mathsf{P}$, the Bayes space $\mathcal{B}^2(\mathsf{P})$ is a space of $\mathcal{B}$-equivalence classes of $\sigma$-finite positive measures $\mu$ with square-integrable log-density w.r.t. $\mathsf{P}$:
$$
\mathcal{B}^2(\mathsf{P}) = \left\{ \mu \in \mathcal{B}^2(\mathsf{P}) : \int \left|\ln \frac{d \mu}{d \mathsf{P}} \right|^2  d \, \mathsf{P}< +\infty\right\}.
$$

\subsection{Hilbert geometry in weighted Bayes spaces}\label{subsec:weightedBayes}
In this subsection, we introduce the Hilbert space geometry of weighted Bayes spaces. Like in the standard $L^2$ space, Bayes spaces operations analogous to addition of two functions and multiplication of a function by a scalar together with definition of the inner product are expected. They should now, however, respect the scale invariance of densities. While both operations, called perturbation and powering in the following, remain formally unchanged when changing the reference measure, the weighting affects the inner product. Here also the absolute scale of reference measure $\mathsf{P}$ matters, which corresponds to volume of the space $\Omega$. It is possible to express densities from the Bayes space in the $L^2$ space (with respect to reference measure $\mathsf{P}$) using clr transformation. This, however, still leaves open the problem of how to express the weighted densities in an \textit{unweighted} $L^2$ space. A possible solution will be presented in Section \ref{subsec:clrs}.

Using a reference measure $\mathsf{P}$, Van den Boogaart et al. \cite{boogaart10} define the operations of \textit{perturbation} and \textit{powering} as
\begin{equation} \label{eq}
(\mu \oplus_\mathsf{P} \nu)(\mbox{A}) =_{\mathcal{B}(\mathsf{P})} \int_{\mbox{A}} \frac{d \mu}{d \mathsf{P}}(t) \cdot \frac{d \nu}{d \mathsf{P}}(t) \, d \mathsf{P}(t), \quad \mbox{A} \in \mathcal{A}
\end{equation}
and
\begin{equation} \label{eq1}
(\alpha \odot_\mathsf{P} \mu)(\mbox{A}) =_{\mathcal{B}(\mathsf{P})} \int_{\mbox{A}} \left(\frac{d \mu}{d \mathsf{P}}(t)\right)^{\alpha} \, d \mathsf{P}(t), \quad \mbox{A} \in \mathcal{A},
\end{equation}
where $\mu$ and $\nu$ are measures in $\mathcal{B}^2(\mathsf{P})$ and  $\alpha$ is a real number. Moreover, all the measures $\mu,\,\nu,\,\lambda$ and $\mathsf{P}$ are assumed to be Radon-Nikodym derivatives of each other. These operations define a vector space structure on
$\mathcal{B}^2(\mathsf{P})$ \cite{boogaart10}.
The operations (\ref{eq}) and (\ref{eq1}) can be equivalently expressed using the densities with respect to $\mathsf{P}$. Denoting them by $f_\mathsf{P} = \frac{d \mu}{d \mathsf{P}}$ and $g_\mathsf{P} = \frac{d \nu}{d \mathsf{P}}$ respectively, we have that
\begin{equation*}
(f_\mathsf{P} \oplus_\mathsf{P} g_\mathsf{P})(t) =_{\mathcal{B}(\mathsf{P})} f_{\mathsf{P}}(t) \cdot g_{\mathsf{P}}(t) \quad \text{and} \quad (\alpha \odot_\mathsf{P} f_\mathsf{P})(t) = f_\mathsf{P}(t)^{\alpha}.
\end{equation*}

Both operations can be represented with respect to the Lebesgue measure $\lambda$, which is preferable from a practical perspective. Indeed, it holds that
\begin{eqnarray*} \label{eq-b}
(\mu \oplus_\mathsf{P} \nu)(\mbox{A}) &=_{\mathcal{B}(\lambda)}& \int_{\mbox{A}} \left(\frac{d \mu}{d \lambda} \cdot \frac{d \nu}{d \lambda}\right) \cdot \left(\frac{d \mathsf{P}}{d \lambda}\right)^{-1} \, d \lambda,
\quad \mbox{A} \in \mathcal{A}
\end{eqnarray*}
and
\begin{eqnarray*} 
(\alpha \odot_\mathsf{P} \mu)(\mbox{A}) &=_{\mathcal{B}(\lambda)}& \int_{\mbox{A}} \left(\frac{d \mu}{d \lambda}\right)^{\alpha} \cdot \left(\frac{d \mathsf{P}}{d \lambda}\right)^{-\alpha+1} \, d \lambda,
\quad \mbox{A} \in \mathcal{A};
\end{eqnarray*}
and in terms of densities
\begin{eqnarray*}
(f_\mathsf{P} \oplus_\mathsf{P} g_\mathsf{P})(t) &=_{\mathcal{B}(\lambda)} & \left[ f(t) \oplus g(t) \right]\ominus p(t) =_{\mathcal{B}(\lambda)} \frac{f(t) \cdot g(t)}{p(t)}, \quad t \in \Omega
\end{eqnarray*}
and
\begin{eqnarray*}
(\alpha \odot_\mathsf{P} f_\mathsf{P})(t)&=_{\mathcal{B}(\lambda)}& \left[ \alpha \odot f(t)\right] \ominus p(t) =_{\mathcal{B}(\lambda)} \frac{f(t)^\alpha}{  p(t)^{\alpha-1} }, \quad t \in \Omega,
\end{eqnarray*}
where $\oplus, \ominus, \odot$ are operations in $\mathcal{B}^2(\lambda)$.

It is easy to verify that scale invariance of the reference density $p$ holds for these operations.
On the other hand, the scale of $p$ is crucial for the definition of the inner product \cite{boogaart14}:
\begin{equation} \label{eq: inner}
\begin{split}
\left\langle f_\mathsf{P}, g_\mathsf{P}\right\rangle_{\mathcal{B}(\mathsf{P})} & = \frac{1}{2\mathsf{P}(\Omega)} \int_\Omega \int_\Omega \ln \frac{f_\mathsf{P}(t)}{f_\mathsf{P}(u)} \ln \frac{g_\mathsf{P}(t)}{g_\mathsf{P}(u)} \, d \mathsf{P}(t) d \mathsf{P}(u) \\
                                                    & = \frac{1}{2\mathsf{P}(\Omega)} \int_\Omega \int_\Omega \ln \frac{f(t)}{f(u)} \ln \frac{g(t)}{g(u)} \cdot p(t) \cdot p(u) \, d \lambda(t) d \lambda(u),
\end{split}
\end{equation}
which endows the Bayes space $\mathcal{B}^2(\mathsf{P})$ with a separable Hilbert space structure. As a consequence, the distance between two densities $f_{\mathsf{P}},g_{\mathsf{P}}\in\mathcal{B}^2(\mathsf{P})$ is obtained as
\begin{equation} \label{eq: dist}
d_{\mathcal{B}(\mathsf{P})}(f_\mathsf{P}, g_\mathsf{P}) = \sqrt{\frac{1}{2\mathsf{P}(\Omega)} \int_\Omega \int_\Omega 
\left(\ln \frac{f_\mathsf{P}(t)}{f_\mathsf{P}(u)} - \ln \frac{g_\mathsf{P}(t)}{g_\mathsf{P}(u)}\right)^2 \, d \mathsf{P}(t) d \mathsf{P}(u).}                                                    
\end{equation}

Note that the above definitions of inner product (\ref{eq: inner}) and distance (\ref{eq: dist}) generalize the approach presented in \cite{egozcue16} in order to keep dominance under change of reference measure. Specifically, let $p_0$ be a uniform density of a measure $\mathsf{P}_0$, not necessarily normalized to $\mathsf{P}_0(\Omega)=1$, supported in an interval (or compact set) $I$ in $\mathbb{R}$ (or $\mathbb{R}^m$), such that
$$\mathsf{P}_0(I)=\int_I p_0(t)\,dt< +\infty.$$
Let $p,q$ be densities in $\mathcal{B}^2(\mathsf{P}_0)$ corresponding to measures $\mathsf{P},\mathsf{Q}$ such that $\mathsf{P}$ dominates $\mathsf{Q}$, 
$\mathsf{P}\succ \mathsf{Q}$, that is 
$$\mathsf{P}_0(t\in I:p(t)\geq q(t))=\mathsf{P}_0(I).$$
Then, for $f_{\mathsf{P}_0},g_{\mathsf{P}_0}\in\mathcal{B}^2(\mathsf{P}_0)$,
\begin{equation}
\label{eq: domin}
d_{\mathcal{B}(\mathsf{P})}(f_\mathsf{P}, g_\mathsf{P})\geq d_{\mathcal{B}(\mathsf{Q})}(f_\mathsf{Q}, g_\mathsf{Q}),
\end{equation}
where $f_\mathsf{P}=f_{\mathsf{P}_0}\cdot d\mathsf{P}_0/d\mathsf{P}=_{\mathcal{B}(\mathsf{P})}f_{\mathsf{P}_0}\ominus p_{\mathsf{P}_0}$ and
$g_\mathsf{P}=g_{\mathsf{P}_0}\cdot d\mathsf{P}_0/d\mathsf{P}=_{\mathcal{B}(\mathsf{P})}g_{\mathsf{P}_0}\ominus p_{\mathsf{P}_0}$ \cite{gelman13}. The property \eqref{eq: domin} represents indeed the continuous counterpart to the subcompositional dominance in compositions \cite{pawlowsky15}. That is, if the volume of the space $\mathsf{P}(I)$ is greater than or equal to $\mathsf{Q}(I)$ uniformly for any subinterval of $I$, then distances in
$\mathcal{B}(\mathsf{P})$ dominate distances in $\mathcal{B}(\mathsf{Q})$. A limiting case is comparing distances in
a subinterval $I_1\subseteq I$ with those in $I$, which corresponds to subcompositions in the simplex \cite{egozcue16}.

Let's denote by $L_0^2(\mathsf{P})$ the closed subspace of $L^2(\mathsf{P})$ whose elements $f_0$ have zero integral $\int_{\Omega}f_0\, d \mathsf{P}=0$. Since the Bayes space $\mathcal{B}^2(\mathsf{P})$ is Hilbertian, we can define an isometric isomorphism (i.e. a bijective map preserving distances) between $\mathcal{B}^2(\mathsf{P})$ and $L_0^2(\mathsf{P})$. Such a map is provided by the \textit{centred log-ratio (clr)} transformation with respect to $\mathsf{P}$, which is denoted by $\clr_{\mathsf{P}}$ and is defined for $f_\mathsf{P} \in \mathcal{B}^2(\mathsf{P})$ by \cite{boogaart14} as
\begin{equation} \label{clrp}
f^c_\mathsf{P}(t) = \clr_\mathsf{P}(f_\mathsf{P})(t) = \ln f_\mathsf{P}(t) - \frac{1}{\mathsf{P}(\Omega)} \int_{\Omega} \ln f_\mathsf{P}(u)  \, d \mathsf{P}(u), \quad t \in \Omega.
\end{equation} 
Its inverse mapping to $\mathcal{B}^2(\mathsf{P})$ is obtained by using the exponential transformation, $\exp[f^c_\mathsf{P}](t) = \exp[\clr_\mathsf{P}(f_\mathsf{P})](t)$, as shown in \cite{boogaart14}. The clr representation allows to use the ordinary geometry of $L^2(\mathsf{P})$ to conduct operations of perturbation, powering, and inner product for the elements of $\mathcal{B}^2(\mathsf{P})$, while accounting for the specific features captured by the Bayes space. Indeed,
\begin{eqnarray*} \label{}
\clr_\mathsf{P}(f_\mathsf{P} \oplus_\mathsf{P} g_\mathsf{P}) = \clr_\mathsf{P}(f_\mathsf{P}) + \clr_\mathsf{P}(g_\mathsf{P}), \quad \clr_\mathsf{P}(\alpha \odot f_\mathsf{P}) = \alpha \cdot \clr_\mathsf{P}(f_\mathsf{P})(t)
\end{eqnarray*}
and
\begin{eqnarray} \label{innprod-clrP}
\left\langle f_\mathsf{P}, g_\mathsf{P} \right\rangle_{\mathcal{B}^2(\mathsf{P})} = \left\langle \clr_\mathsf{P}(f_\mathsf{P}), \clr_\mathsf{P}(f_\mathsf{P})\right\rangle_{L^2(\mathsf{P})}.
\end{eqnarray}
In order to prove the relationship (\ref{innprod-clrP}) (recall: $f_\mathsf{P}, g_\mathsf{P}$ are elements of $\mathcal{B}^2(\mathsf{P})$), we develop the right-hand side of (\ref{eq: inner}),
\begin{equation*} 
\begin{split}
\left\langle f_\mathsf{P}, g_\mathsf{P}\right\rangle_{\mathcal{B}(\mathsf{P})} = & 
\frac{1}{\mathsf{2P}(\Omega)} \int_{\Omega} \int_{\Omega} \left[\ln f_\mathsf{P}(t) -\ln f_\mathsf{P}(u) \right] \cdot \left[\ln g_\mathsf{P}(t) -\ln g_\mathsf{P}(u) \right] \, d\mathsf{P} (t) \, d\mathsf{P} (u) 
\\
= &\frac{1}{2\mathsf{P}(\Omega)}
\left[
           2\int_{\Omega} \int_{\Omega} \ln f_\mathsf{P}(t) \cdot \ln g_\mathsf{P}(t) \, d\mathsf{P} (t) \, d\mathsf{P} (u) \right. \\
& \left. - 2\int_{\Omega} \int_{\Omega} \ln f_\mathsf{P}(t) \cdot \ln g_\mathsf{P}(u) \, d\mathsf{P} (t) \, d\mathsf{P} (u) 
\right] \\
= &
\int_{\Omega} \ln f_\mathsf{P}(t) \cdot \ln g_\mathsf{P}(t) \, d\mathsf{P} (t) - 
\frac{1}{\mathsf{P}(\Omega)} \int_{\Omega} \ln f_\mathsf{P}(t) \, d\mathsf{P} (t) \cdot \int_{\Omega} \ln g_\mathsf{P}(u) \, d \mathsf{P} (u),
\end{split}
\end{equation*}
which truly equals the right-hand side of (\ref{innprod-clrP}),
\begin{equation*} 
\begin{split}
& \left\langle \clr_\mathsf{P}(f_\mathsf{P}), \clr_\mathsf{P}(f_\mathsf{P})\right\rangle_{L^2(\mathsf{P})} = 
\int_{\Omega} \clr_\mathsf{P}(f_\mathsf{P})(t) \cdot \clr_\mathsf{P}(g_\mathsf{P})(t) \, d\mathsf{P} (t) \\
 = &
\int_{\Omega} \left[
\ln f_\mathsf{P}(t) - \frac{1}{\mathsf{P}(\Omega)} \int_{\Omega} \ln f_\mathsf{P}(u) d\mathsf{P} (u)
\right] 
\cdot
\left[ 
\ln g_\mathsf{P}(t) - \frac{1}{\mathsf{P}(\Omega)} \int_{\Omega} \ln g_\mathsf{P}(u) d\mathsf{P} (u)
\right] d\mathsf{P} (t) \\
 = &
\int_{\Omega} \ln f_\mathsf{P}(t) \cdot \ln g_\mathsf{P}(t) \, d\mathsf{P} (t)  d \mathsf{P} (t) - \frac{2}{\mathsf{P}(\Omega)}  \int_{\Omega} \ln f_\mathsf{P}(t) \, d\mathsf{P} (t) \cdot \int_{\Omega} \ln g_\mathsf{P}(u) \, d \mathsf{P} (u) \\
 & +  \frac{1}{\mathsf{P}^2(\Omega)} \int_{\Omega}
\left[
\int_{\Omega} \ln f_\mathsf{P}(u) \, d \mathsf{P} (u) \int_{\Omega} \ln g_\mathsf{P}(u) \, d \mathsf{P} (u)
\right] d \mathsf{P} (t) \\
 = &
\int_{\Omega} \ln f_\mathsf{P}(t) \cdot \ln g_\mathsf{P}(t) \, d\mathsf{P} (t) - 
\frac{1}{\mathsf{P}(\Omega)} \int_{\Omega} \ln f_\mathsf{P}(t) \, d\mathsf{P} (t) \cdot \int_{\Omega} \ln g_\mathsf{P}(u) \, d \mathsf{P} (u),
\end{split}
\end{equation*}
where $t,u \in \Omega$.
As noted by \cite{hron16,Mach,talska18}, the zero integral constraint of clr transformed $\mathsf{P}$-densities ($\int_\Omega \clr_\mathsf{P}(f_\mathsf{P}) d \mathsf{P} = 0$) should be taken into account for any subsequent statistical analysis.

Unlike the case of \cite[Sect. 4]{boogaart14}, in this work the reference measure $\mbox{P}$ in $L_0^2(\mathsf{P})$ is not necessarily a probability measure, as its normalization may lead to incoherent results when restricting the analysis to a subdomain of the original domain $\Omega$ (as was shown in the discrete case \cite{egozcue16}). Moreover, it should be noted that, when considering density functions defined on an unbounded domain, the transformation $\clr_\mathsf{P}$ is no longer valid with an (unrestricted) Lebesgue reference measure, as $\clr_\mathsf{P}$ involves to the total mass $\mathsf{P}(\Omega)$ in the denominator. In those cases, a reference measure which is finite on $\Omega$ has to be chosen. Nonetheless, once a proper reference measure is set, the statistical analysis can be performed equivalently to the case with finite support.

\subsection{Unweighting Bayes spaces}\label{subsec:clrs}
Most methods developed for FDA rely on the assumption that functional data are embedded in the \emph{unweighted} $L^2$ space.
However, the clr transformation (\ref{clrp}) maps measures in (a subspace of) a weighted space $L^2$ space, i.e. $L^2_0(\mathsf{P})$. A transformation mapping $\mathsf{P}$-densities from $\mathcal{B}^2(\mathsf{P})$ to an unweighted counterpart of $L_0(\mathsf{P})$ would have the advantage of allowing the use of most FDA methods, while accounting for the weighted Bayes structure of the data. In this subsection, we derive an unweighting scheme to represent the weighted Bayes space geometry in an unweighted Bayes space, as well as in an unweighted $L^2$ space.

We thus aim to define three mappings. Firstly, we define $\omega$ from $\mathcal{B}^2(\lambda)$ to $\mathcal{B}^2(\mathsf{P})$ as a \emph{weighting} map associating an unweighted $\lambda$-density to a weighted $\mathsf{P}$-density. Inversely, $\omega^{-1}$ is interpreted as an \emph{unweighting} map. Similarly, we define $\omega_2$ and its inverse $\omega_2^{-1}$ which play the same role between the unweighted and weighted $L^2$ spaces, i.e. $L^2(\lambda)$ and $L^2(\mathsf{P})$ respectively. Finally, we define $\clr_{u}$ (\emph{unweighting clr}) such that, for $f_\mathsf{P} \in \mathcal{B}^2(\mathsf{P})$,
\begin{eqnarray*} \label{clr2-prop}
\clr_{u}(f_\mathsf{P} \oplus_\mathsf{P} g_\mathsf{P}) = \clr_{u}(f_\mathsf{P}) + \clr_{u}(g_\mathsf{P}), \quad \clr_{u}(\alpha \odot f_\mathsf{P}) = \alpha \cdot \clr_{u}(f_\mathsf{P})(t)
\end{eqnarray*}
and
\begin{eqnarray}
\left\langle f_\mathsf{P}, g_\mathsf{P} \right\rangle_{\mathcal{B}^2(\mathsf{P})} = \left\langle \clr_\mathsf{u}(f_\mathsf{P}),\clr_\mathsf{u}(f_\mathsf{P})\right\rangle_{L^2(\lambda)}.
\end{eqnarray}
To support this construction and study the properties of these maps, we shall use an auxiliary measure $\sqrt{\mathsf{P}}$ defined as
$$
\sqrt{\mathsf{P}}(\mbox{A}) = \int_{\mbox{A}} \sqrt{p}\, d\lambda, \quad \mbox{A}\in \mathcal{A}.
$$
This measure plays the role of \emph{unweighting} measure, in the sense that it allows to consistently map the \emph{weighted} Bayes space $\mathcal{B}^2(\mathsf{P})$ into a subset of the \emph{unweighted} $L^2$ space. We refer the reader to the scheme in Figure \ref{fig:scheme-spaces} as a concise representation of these relationships.

\begin{mdframed}
\edef\myindent{\the\parindent}
\begin{minipage}{0.9\textwidth}\setlength{\parindent}{\myindent}
\vspace*{0.4cm}
\begin{equation*}
\begin{tikzcd}
L^{2}_0(\lambda)
\arrow[r, shift left, leftarrow ,"\clr_{\lambda}"]
\arrow[r, shift right, swap, "\exp"]
& \mathcal{B}^{2}(\lambda)
\arrow[r, shift left, "\ominus p"]
\arrow[r, shift right, leftarrow, swap, "\oplus p"]
& \mathcal{B}^{2}(\mathsf{P})
\arrow[r, shift left, "\clr_\mathsf{P}"] \arrow[r, shift right, leftarrow, swap, "\exp"]
\arrow[rd, "\clr_{u}"]
\arrow[d, shift left,"\omega^{-1}"]
&[3.5em]
L^{2}_0(\mathsf{P})
\arrow[d,shift left, "\omega_2^{-1}"]\\[3.5em]
& & \mathcal{B}^{2}(\lambda)
\arrow[r, shift left, "\clr_{\sqrt{\mathsf P}}"]
\arrow[r, shift right, leftarrow, swap, "\exp"]
\arrow[u, shift left,"\omega"]
& L^{2}_{0,\mathsf{\sqrt{P}}}(\lambda)
\arrow[u, shift left,"\omega_2"]
\end{tikzcd}
\end{equation*}
\end{minipage}
\end{mdframed}

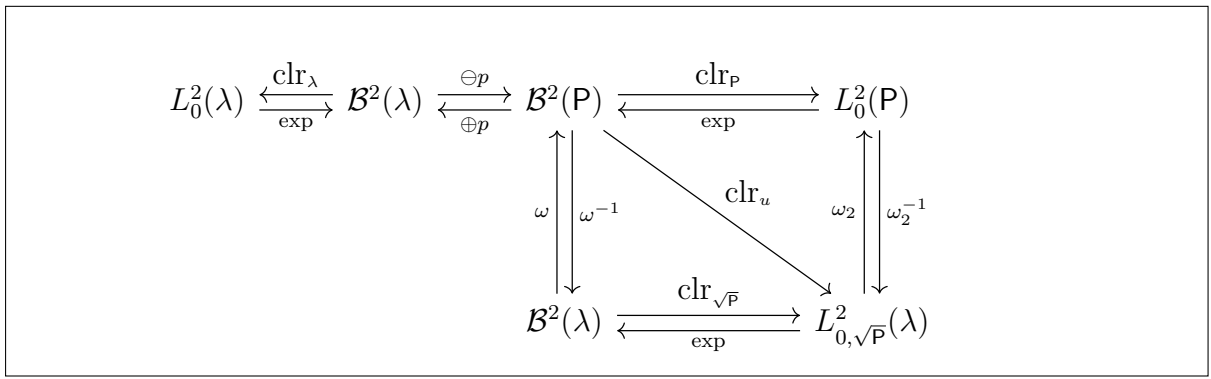
\captionof{figure}{Relationships among weighted and unweighted Bayes spaces, $\mathcal{B}^2(\mathsf{P})$ and $\mathcal{B}^2(\lambda)$, and weighted and unweighted $L^2(\mathsf{P})$ and $L^2(\lambda)$ spaces.\label{fig:scheme-spaces}}

\vspace*{1cm}
We define the $\mathcal{B}^2$-\emph{weighting} map $\omega$ as
\begin{equation}\label{eq:omega}
\begin{split}
\omega:\mathcal{B}^2(\lambda) & \rightarrow \mathcal{B}^2(\mathsf P) \\
\varphi & \mapsto \omega(\varphi) = \varphi^{1/\sqrt{p}},
\end{split}
\end{equation}
where $p = \frac{d\mathsf{P}}{d\lambda}$ (recall: $p$ is assumed to be strictly positive in $\Omega$).
The map $\omega$ defines a bijection between $\mathcal{B}^2(\lambda)$ and $\mathcal{B}^2(\mathsf P)$, as proved in the following proposition.

\begin{proposition}\label{lemma:omega}
The map $\omega$ defined in \eqref{eq:omega} is one-to-one and onto.
\end{proposition}
\begin{proof}
For $\varphi\in \mathcal{B}^{2}(\lambda)$, let be $\omega(\varphi) = \varphi^{1/\sqrt{p}}$. Clearly, $\omega(\varphi)$ is uniquely defined. Then $\omega(\varphi)\in \mathcal{B}^{2}(\mathsf P)$ due to
\[
\int_{\Omega} \ln^2(\varphi)\, d\lambda = \int_{\Omega} \ln^2[\omega(\varphi)^{\sqrt{p}}]\, d\lambda = \int_{\Omega} \ln^2[\omega(\varphi)]{p}\, d\lambda = 
\int_{\Omega} \ln^2[\omega(\varphi)]\, d\mathsf P.\]
Inversely, for $\omega(\varphi)\in \mathcal{B}^{2}(\mathsf P)$, let be $\varphi = \omega(\varphi)^{\sqrt{p}}$. Then, based on the same arguments, it results that $\varphi\in \mathcal{B}^{2}(\lambda)$.
\end{proof}
The inverse $\omega^{-1}$ is defined as $\omega^{-1}(\psi) = \psi^{\sqrt{p}}$ and it is interpreted as a $\mathcal{B}^2$-\emph{unweighting} map. It is represented in the bottom left part of the scheme in Figure \ref{fig:scheme-spaces}. Obviously, both $\omega$ and $\omega^{-1}$ depend on the scale of $\mathsf{P}$.

We define the $L^2$-\emph{weighting} map $\omega_2$ as
\begin{eqnarray*}
\omega_2:L^2(\lambda)&\rightarrow&L^2(\mathsf P)\\
\eta &\mapsto& \omega(\eta) = \eta/ \sqrt{p}.
\end{eqnarray*}
Using the same rationale as for Proposition \ref{lemma:omega}, it can be proved that $\omega_2$ defines a bijection between $L^2(\lambda)$ and $L^2(\mathsf P)$.
Its inverse $\omega_2^{-1}$ is defined as $\omega^{-1}_2(\xi) = \xi{\sqrt{p}}$ and it is interpreted as a $L^2$-\emph{unweighting} map. It is represented in the bottom right part of the scheme in Figure \ref{fig:scheme-spaces}.
Note that $\omega$ is non-linear with respect to the Bayes space geometry, as well as $\omega_2$ is non-linear in $L^2$.

Using \eqref{eq: inner}, the map $\clr_{u}: B^2(\mathsf P) \rightarrow L^2(\lambda)$ can be then defined as
\begin{equation} \label{clr2}
\clr_{u}(f_\mathsf{P}) = \omega_2^{-1}[\clr_\mathsf{P}(f_\mathsf{P})].
\end{equation}
It can be proven that \eqref{clr2} fulfills all the properties detailed in \eqref{clr2-prop}. Note that the scale of $\clr_{u}$ depends on the scale of $\sqrt{p}$, hence on the scale of $\sqrt{\mathsf{P}}$, because of the non-linearity of $\omega_2$ (see \cite{egozcue11} for the case of finite-dimensional compositions). As such, similarly to the multivariate case \cite{egozcue16}, the scale of the reference measure is relevant in the geometry of both weighted and unweighted spaces.

It is worth noticing that $\clr_{u}$ is closely related to a different centered log-ratio transformation. This is defined on the unweighted space $\mathcal{B}^{2}(\lambda)$ and induced by the unweighting measure $\sqrt{\mathsf{P}}$. Indeed, let $L^{2}_{0,\sqrt{\mathsf{P}}}(\lambda)$  be the subspace of $L^{2}(\lambda)$ such that $\int_\Omega f\, d \sqrt{\mathsf{P}} = 0$ for $f \in L^{2}(\lambda)$. Let's define on $\mathcal{B}^{2}(\lambda)$ the map $\clr_{\sqrt{\mathsf{P}}}$ as
\begin{equation} \label{clrsqrtp}
\clr_{\sqrt{\mathsf{P}}}(\varphi)(t) = \ln \varphi(t) - \frac{1}{\mathsf{\sqrt{P}}(\Omega)} \int_{\Omega} \ln [\varphi(u)]  \, d\sqrt{\mathsf{P}} (u), \quad t \in \Omega, \quad \varphi\in \mathcal{B}^{2}(\lambda).
\end{equation}
In light of Proposition \ref{lemma:omega}, it is easy to see that the map \eqref{clrsqrtp} is well defined. For any $\varphi\in \mathcal{B}^{2}(\lambda)$, we can set $f_{\mathsf P} \in \mathcal{B}^2(\mathsf{P})$ to $f_{\mathsf P} = \omega(\varphi) = \varphi^{1/\sqrt{p}}$. Then, it holds that
\[
\int_{\Omega} \ln [\varphi(u)]  \, d\sqrt{\mathsf{P}} (u) = \int_{\Omega} \ln [f_{\mathsf P}(u)] p(u) \,d\lambda (u) < +\infty.\]
Moreover, for any $\varphi$ in $\mathcal{B}^{2}(\lambda)$, we have that $\clr_{\sqrt{\mathsf{P}}}(\varphi)\in L^2_{0,\sqrt{\mathsf{P}}}(\lambda)$.
The following proposition establishes the close relationship between $\clr_{u}$ and $\clr_\mathsf{\sqrt{P}}$, thus completing the scheme in Figure \ref{fig:scheme-spaces}.

\begin{proposition}
The following statements hold true.
\begin{enumerate}
\item[(i)] The image of the space $\mathcal{B}^{2}(\mathsf{P})$ under the map $\clr_{u}$ defined in \eqref{clr2} is $L^{2}_{0,\sqrt{\mathsf{P}}}(\lambda)$.
\item[(ii)] The map $\clr_{u}$ coincides with the composed function $\clr_{\sqrt{\mathsf P}}\circ\omega^{-1}$, i.e.
\[
\clr_{u}(f_{\mathsf P}) = \clr_{\sqrt{\mathsf P}}(\omega^{-1}(f_{\mathsf P})) \quad \text{and} \quad f_{\mathsf P}\in \mathcal{B}^{2}(\mathsf{P}).\]
\item[(iii)] The inverse of the map $\clr_\mathsf{\sqrt{P}}$ is $\clr_\mathsf{\sqrt{P}}^{-1}: L^{2}_{0,\sqrt{\mathsf{P}}}(\lambda)\rightarrow \mathcal{B}^{2}(\lambda)$  and is given by
\[
\clr_\mathsf{\sqrt{P}}^{-1}(\psi) =_{\mathcal{B}^{2}(\lambda)} \exp(\psi),\]
for any $\psi$ in $L^{2}_{0,\sqrt{\mathsf{P}}}$.
\item[(iv)] The inverse of the map $\clr_{u}$ is $\clr_{u}^{-1}: L^{2}_{0,\sqrt{\mathsf{P}}}(\lambda)\rightarrow \mathcal{B}^{2}(\mathsf P)$  and is given by
\[
\clr_{u}^{-1}(\psi) =_{\mathcal{B}^{2}(\mathsf{P})} \exp[\omega_2(\psi)] =_{\mathcal{B}^{2}(\mathsf{P})} \omega[\exp(\psi)],\]
for any $\psi$ in $L^{2}_{0,\sqrt{\mathsf{P}}}$.
\end{enumerate}
\end{proposition}

\begin{proof}
\emph{Statement (i)}. Let's denote by $f_\mathsf{P}$ a density in $\mathcal{B}^2(\mathsf{P})$. Then
\[
\int_{\Omega} \clr_{u}(f_\mathsf{P})\, d\sqrt{\mathsf{P}} = \int_{\Omega} \clr_{u}(f_\mathsf{P})\,\sqrt{p}\, d\lambda = \int_{\Omega} \clr_\mathsf{P}(f_\mathsf{P})\,d\mathsf{P} = 0,\]
proving the first statement.

\emph{Statement (ii)}. Consider $f_\mathsf{P}\in\mathcal{B}^2(\mathsf{P})$. Then
\begin{align}\label{eq:proof-prop-eq1}
\nonumber\clr_{\sqrt{\mathsf{P}}}(\omega^{-1}(f_\mathsf{P})) &= \ln (f_\mathsf{P}^{\sqrt{p}}) - \frac{1}{\sqrt{\mathsf P}(\Omega)}\int_{\Omega} \ln (f_\mathsf{P}^{\sqrt{p}})\, d \sqrt{\mathsf P} = \\
& = {\sqrt{p}} \ln (f_\mathsf{P}) - \frac{1}{\sqrt{\mathsf P}(\Omega)}\int_{\Omega} \ln (f_\mathsf{P}) p\, d \lambda.
\end{align}
Let's call $\xi\in L^{2}_{0}(\mathsf{P})$ the element $\xi = \clr_\mathsf{P}(f_{\mathsf P})$. Since $\clr_\mathsf{P}$ is one-to-one and onto between $\mathcal{B}^2(\mathsf{P})$ and $L^{2}_{0}(\mathsf{P})$, it holds that $f_{\mathsf P} =_{\mathcal{B}^2(\mathsf{P})} \exp{\xi}$ and we can rewrite
\[
\clr_{\sqrt{\mathsf{P}}}(\omega^{-1}(f_\mathsf{P})) = \xi\sqrt{p},
\]
where the last term of \eqref{eq:proof-prop-eq1} cancels because $\xi \in L^{2}_{0}(\mathsf{P})$.
Considering $\clr_{u}$, using the same notation as before, it results that
\[
\clr_{u}(f_\mathsf{P}) = \sqrt{p}\cdot\left[\ln (f_\mathsf{P}) - \frac{1}{\mathsf P(\Omega)}\int_{\Omega} \ln (f_\mathsf{P})\, d \mathsf P \right] = \xi{\sqrt{p}}.
\]

\emph{Statement (iii)}. For $\psi \in L^2_{0,\sqrt{\mathsf{P}}}$, it holds that
\[
\clr_\mathsf{\sqrt{P}}[\exp(\psi)](u) = \ln[\exp(\psi)] - \frac{1}{\sqrt{\mathsf{P}}(\Omega)}\int_{\Omega}\ln[\exp(\psi(u))]d\sqrt{\mathsf{P}}(u) = \psi(u),\]
for any $u \in \Omega$.

\emph{Statement (iv)}. This is an obvious consequence of the previous point (iii).
\end{proof}

\vspace*{0.5cm}
Note that taking the $\mathcal{B}^2$-\emph{unweighting} transformation $\omega^{-1}$ is indeed different from simply changing the reference measure from $\mathsf{P}$ to $\lambda$. The former transformation is indeed used to represent the weighted Bayes space through an unweighted one, while preserving its weighted Hilbert geometry. In fact, as further highlighted in Sections \ref{sec:simu} and \ref{sec:casestudy}, this auxiliary space may serve to enhance interpretation of the weighted structure: the $\mathcal{B}^2$-\emph{unweighting} $\mathsf P$-densities can be interpreted in the same way as we interpret unweighted PDFs. It is also clear that, as long as the Lebesgue reference measure is concerned ($\mathsf{P}(\Omega) = \lambda([a,b])$),  the transformations $\clr_{u}$ and $\clr_\mathsf{P}$ coincide, and they reduce to the clr transformation used in \cite{hron16,menafoglio14} -- here denoted by $\clr_{\lambda}$. Note, however, that this would not be true for \textit{any} uniform measure because the scale of the reference does have an impact on the Hilbert geometry.

\medskip
The above considerations have a direct impact on applications. For a sample of densities $f_1, \ldots, f_N$ to be analyzed with respect to a reference measure $\mathsf{P}$, the following strategy can be adopted:
\begin{enumerate}
\item Set the reference measure $\mathsf P$.
\item If the PDFs were given w.r.t. the Lebesgue measure, change the reference measure from $\lambda$ to $\mathsf{P}$. That is, set $f_{\mathsf{P},i} = f_i \ominus p$, for $i=1, \ldots, N$, with $f_{\mathsf{P},i}\in \mathcal{B}^2(\mathsf{P})$.
	\item Map $f_{\mathsf{P},i}$, for $i=1, \ldots, N$, onto $L_{0, \sqrt{\mathsf{P}}}^2(\lambda)$ by using the $\clr_{u}$ transformation. Set $y_i = \clr_{u}(f_{\mathsf{P},i})$, for $i=1, \ldots, N$.
	\item Perform the statistical analysis on $y_i$, $i=1, \ldots, N$, using \emph{unweighted} $L^2_0$ ($L_{0, \sqrt{\mathsf{P}}}^2(\lambda)$) methods.
	\item If the results needs to be given in terms of densities, use the inverse transformation
	$\exp[\clr_{u}(f_\mathsf{P})]$ to express the results in the unweighted space $\mathcal{B}^2(\lambda)$, where they can be easily interpreted.
\end{enumerate}
This strategy is further illustrated in the Section \ref{subsec:wSFPCA}, which presents a dimensionality reduction method in weighted Bayes spaces.

\section{Statistical methods in weighted Bayes spaces: weighted SFPCA}\label{subsec:wSFPCA}
Simplicial functional principal component analysis (SFPCA, \cite{hron16}) was recently introduced to adapt the well-known functional principal component analysis \cite{ramsay05} to density functions. It is grounded on the theory of Bayes spaces and assumes that the Lebesgue measure is set as reference measure. SFPCA aims to explore the main modes of \emph{relative} variability in a sample of density data and can be used to suggest a possible dimensionality reduction of a dataset of PDFs.
In this section, we extend the SFPCA to its weighted version, named hereafter wSFPCA. Besides its relevance in applications, this extension serves as an illustrative example of the strategy detailed in Subsection \ref{subsec:clrs}.

Let's denote by $f_{1}, \ldots, f_{N}$ an i.i.d. sample in $\mathcal{B}^2(\lambda)$. After selecting the reference measure $\mathsf{P}$ with $\lambda$-density $p$, a sample ${f}_{\mathsf{P},i} = f_i \ominus p$, for $i=1, \ldots, N$, in $\mathcal{B}^2(\mathsf{P})$ is obtained.
We assume without loss of generality this sample is mean-centered. If this is not the case, it is enough to consider $\tilde{f}_{\mathsf{P},i} = {f}_{\mathsf{P},i} \ominus \bar{f}_{\mathsf{P}}$, where $\bar{f}_{\mathsf{P}}$ stands for the (weighted) sample mean of the observed (weighted) densities \[
\bar{f}_{\mathsf{P}} = \frac{1}{N} \odot_{\mathsf{P}}
{\bigoplus}_{i=1}^N {f}_{\mathsf{P},i}.\]
Note that the centering operation shifts the center of the sample to the neutral element of the (weighted) perturbation operation. 
That is, the uniform density on $\mathcal{B}^2(\mathsf{P})$. 

The aim of wSFPCA is to identify a collection of orthogonal and normalized $\mathsf{P}$-density functions $\left\{\xi_{\mathsf{P},j}\right\}_{j\geq 1}$ in $\mathcal{B}^2(\mathsf{P})$  corresponding to the directions in $\mathcal{B}^2(\mathsf{P})$ along which the dataset displays its main modes of variability. These directions are called weighted simplicial functional principal components (wSFPCs), and they are obtained by maximizing the following objective function
\begin{equation} \label{eq_max}
\sum_{i=1}^N \left\langle f_{\mathsf{P},i}, \xi_{\mathsf{P}} \right\rangle_{\mathcal{B}(\mathsf{P})}^2 \: \text{subject to }
\left\| \xi_{\mathsf{P}}\right\|_{\mathcal{B}(\mathsf{P})} = 1; \: \text{with} \qquad \left\langle \xi_{\mathsf{P}}, \xi_{\mathsf{P},k} \right\rangle_{\mathcal{B}(\mathsf{P})} =0, \: k<j,
\end{equation}
over $\xi_{\mathsf{P}}$ in $\mathcal{B}^2(\mathsf{P})$, where $\left\langle f_{\mathsf{P},i}, \xi_{\mathsf{P}} \right\rangle_{\mathcal{B}(\mathsf{P})}$ is the projection of $f_{\mathsf{P},i}$ along the direction in $\mathcal{B}^2(\mathsf{P})$ identified by $\xi_{\mathsf{P}}$, i.e., coordinate of $f_{\mathsf{P}}$ (Fourier coefficient). The orthogonality condition has only to be fulfilled for $j \geq 2$, and guarantees that the $j$th wSFPC $\xi_{\mathsf{P},j}$ is orthogonal to the first $j-1$ wSFPCs.

Since $\mathcal{B}^2(\mathsf{P})$ is a Hilbert space, the solution of the maximization problem (\ref{eq_max}) exists and is unique for all $j \in \left\{1, 2, \ldots, N-1\right\}$. It coincides with the set of eigenfunctions associated with the ordered eigenvalues of the sample covariance operator $V: \mathcal{B}^2(\mathsf{P}) \rightarrow \mathcal{B}^2(\mathsf{P})$, defined for $\xi_{\mathsf{P}} \in \mathcal{B}^2(\mathsf{P})$ as
\begin{equation} \label{V_oper}
V \xi_\mathsf{P} = \frac{1}{N} \odot_\mathsf{P}\bigoplus_{i=1}^N \left\langle f_{\mathsf{P},i}, \xi_{\mathsf{P}} \right\rangle_{\mathcal{B}(\mathsf{P})} \odot_\mathsf{P} f_{\mathsf{P},i}.
\end{equation}
The $j$th wSFPC $\xi_{\mathsf{P},j}$ is thus obtained by solving the eigenequation $V \xi_{\mathsf{P},j} = \rho_j \odot_{\mathsf{P}} \xi_{\mathsf{P},j}$. The $N-1$ eigenvalues $\rho_1 \geq \ldots \geq \rho_{N-1}$ represent the variability of the dataset along the directions of the associated eigenfunctions $\xi_{\mathsf{P},1}, \ldots, \xi_{\mathsf{P},N-1}$.

From the practical viewpoint, it is desirable to restate the problem of finding the eigenpairs $(\xi_{\mathsf{P},j}, \rho_j), j=1, \ldots, N-1$, in $\mathcal{B}^2(\mathsf{P})$ in terms of the unweighted $L^2$ spaces, i.e. $L^2_{0,\sqrt{\mathsf{P}}}(\lambda)$, where well-established computational methods are available.
To this end, consider the $\clr_{u}$ transformation of the data, i.e. $\clr_{u}(f_{\mathsf{P},1}),\ldots,
\clr_{u}(f_{\mathsf{P},N})$. Following the same arguments of \cite{hron16}, one can easily prove that performing a functional principal component analysis of the transformed dataset in $L^2_{0,\sqrt{\mathsf{P}}}(\lambda)$ yields the eigenpairs $\left(\clr_{u}(\xi_{\mathsf{P},j}), \rho_j\right), j=1, \ldots, N-1$. The resulting eigenfunctions
$\clr_{u}(\xi_{\mathsf{P},j})$ can be eventually transformed back into $\mathcal{B}^2(\mathsf{P})$, or into the unweighted $\mathcal{B}^2({\lambda})$, by using the corresponding inverse clr transformation (i.e. $\clr_{u}^{-1}$ or $\clr_{\sqrt{\mathsf{P}}}^{-1}$ respectively) to proceed with interpretation in the original space.

The results of wSFPCA can be interpreted, e.g. by analyzing the principal component scores, which are useful to inspect the relationships among observations. Note that the score $f_{ij}$ is a projection of the (centered) observation $f_{\mathsf{P},i}$ along the direction $\xi_{\mathsf{P},j}$, i.e. $f_{ij} = \left\langle f_{\mathsf{P},i}, \xi_{\mathsf{P},j} \right\rangle_{\mathcal{B}(\mathsf{P})} = \left\langle \clr_{u}(f_{\mathsf{P},i}), \clr_{u}(\xi_{\mathsf{P},j}) \right\rangle_{L^2(\lambda)}$, and thus the scores coincide in $\mathcal{B}^2(\mathsf{P})$ and $L^2(\lambda)$.
It is useful to visualize the mean density perturbed by the $j$th wSFPC $\xi_{\mathsf{P},j}$ powered by a suitable coefficient. This represents the variability around the mean function  along the direction of a given wSFPC, and can support the analyst in the definition of a weighting strategy for the dataset at hand. Indeed, in the context of general reference measures, the wSFPCs can be plotted and interpreted to see the effect of weighting the domain of the distributional variable according to alternative reference measures.
Finally, for the purpose of dimensionality reduction, the number of wSFPCs to be retained can be set by the commonly used scree plot. Particularly, searching for an elbow shape or setting a threshold on the portion of variance explained by wSFPCs as usually.

\section{Illustration using simulated densities from exponential families}
\label{sec:simu}

\subsection{Change of the reference measure: the consequences for density data}
\label{subsec:sim_log_wei}

In order to examine the effects of changing the reference measure, we simulate densities from two exponential families and analyze them with respect to different reference measures -- Lebesgue, uniform (as its normalized counterpart) and exponential measures. While the first two reference measures represent equal weighting on the respective domains, the last one is an example of down-weighting the right-hand side of domain, possibly to stress the relative scale along the domain of the data.

Inspired by the case study presented in Section \ref{sec:casestudy}, we consider a set of (truncated) log-normal densities with means $\mu_i = 0.6+0.25\cdot(i-1)$ and standard deviations $\sigma_j = 0.5+0.07\cdot(j-1)$ for $i, j=1, \ldots,9$, on the interval $I=\left[1,10\right]$. They are represented with respect to the Lebesgue measure and displayed in Figure \ref{fig:log-leb}, where the color scale follows the index $\kappa = j+9(i-1)$, $i,j = 1,...,9$ (i.e. equal mean values are represented with similar colors).
In this case, the transformations $\clr_\mathsf{P}$ and $\clr_{u}$ coincide (Figure \ref{fig:log-leb-clr}), and they are obtained as
\begin{equation} \label{clrden-lognorm}
\begin{split}
f^c_\lambda(t; \mu_i,\sigma_j) = &
 -\frac{\ln^2 t} {2\sigma_j^2} +  \left(-1+\frac{\mu_i}{\sigma_j^2}\right)  \left(\ln t - \frac{10}{9} \cdot \ln 10 + 9\right) + \\
& +\frac{1}{\sigma_j^2} \left(1+\frac{5}{9} \cdot \ln^2 10 -\frac{9}{10} \ln 10 \right), \: t \in I.
\end{split}
\end{equation}

\begin{figure}[h!]
    \centering
    \begin{subfigure}[t]{0.4\textwidth}
        \centering
        \includegraphics[width=\textwidth]{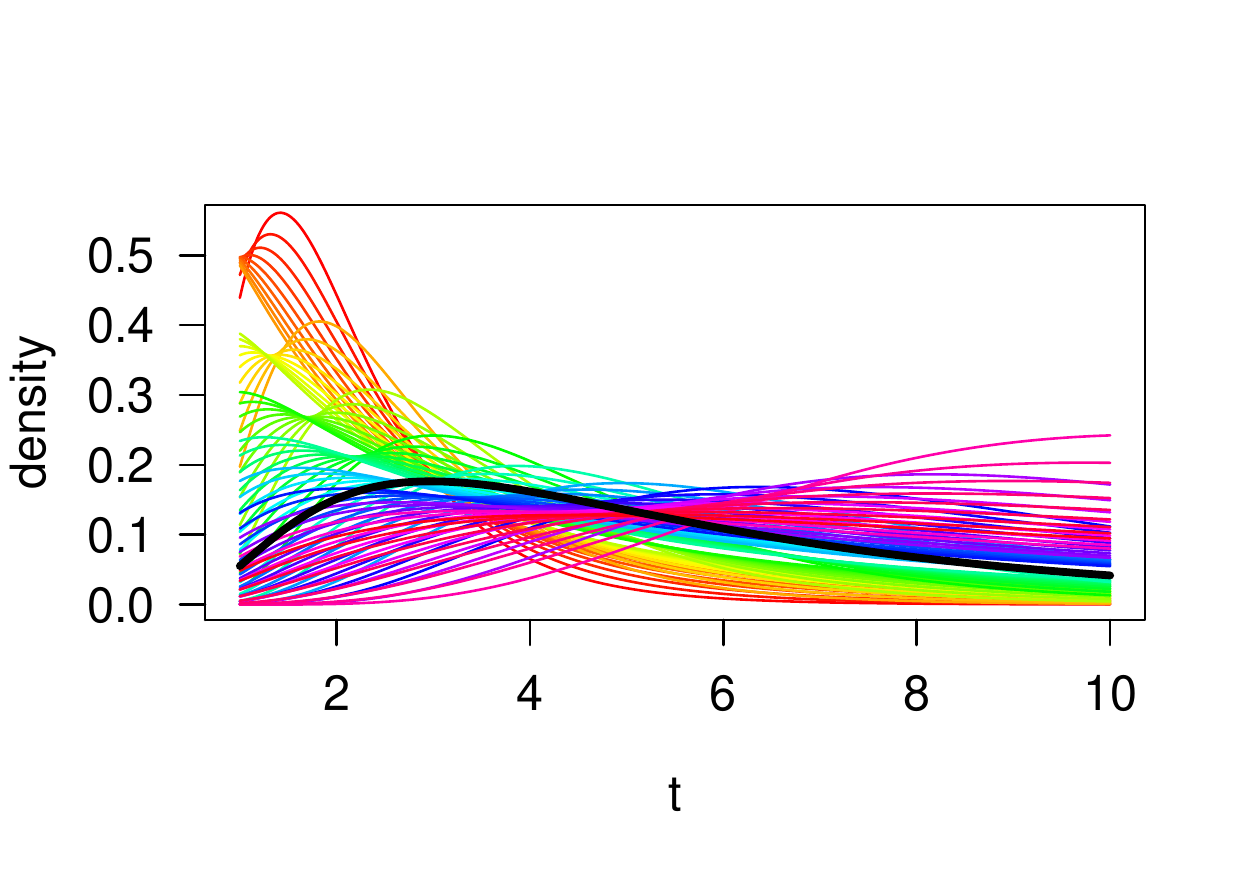}
        \caption{$\lambda$-density functions on $\mathcal{B}^2(\lambda)$.}
				\label{fig:log-leb}
    \end{subfigure}
		    \begin{subfigure}[t]{0.4\textwidth}
        \centering
        \includegraphics[width=\textwidth]{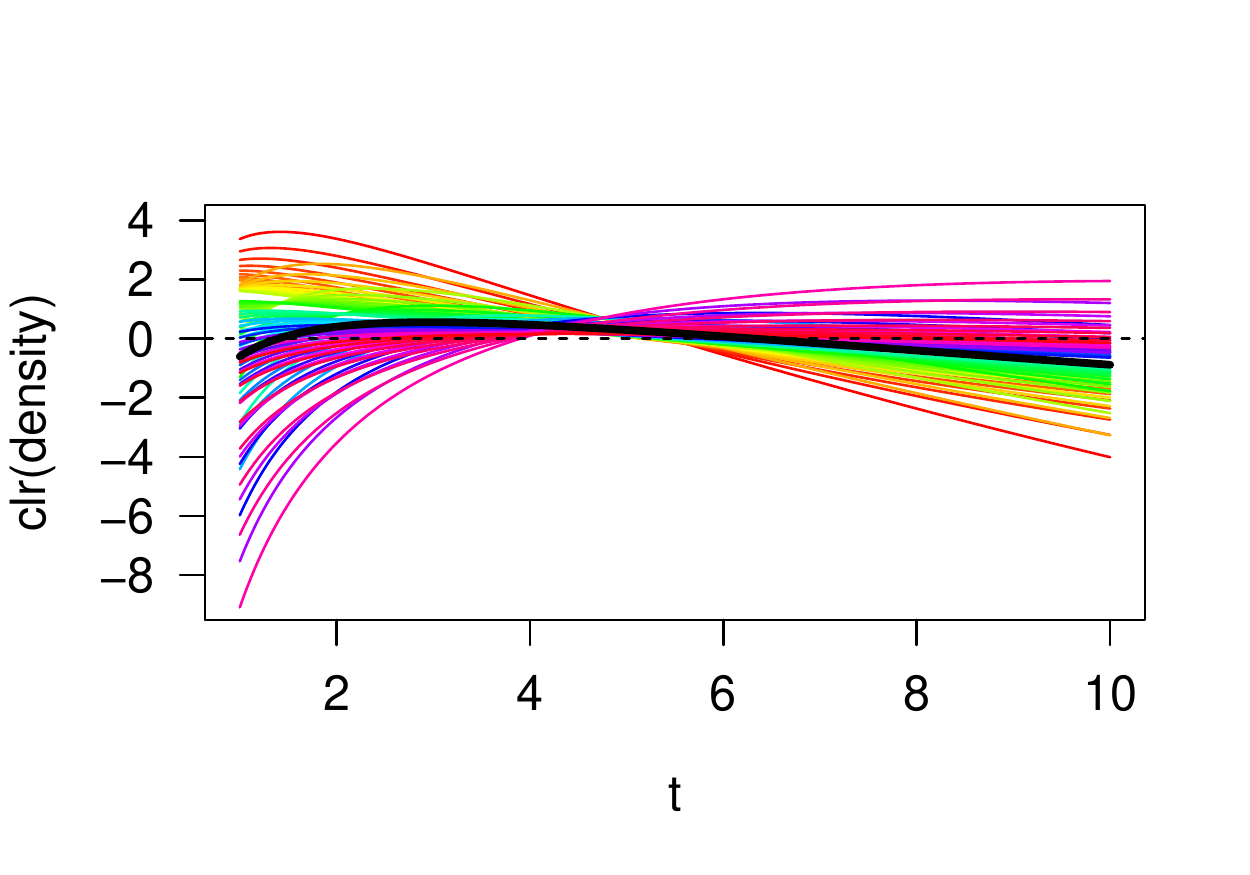}
        \caption{$\lambda$-density functions on $L^2_0(\lambda)$ (after $\clr_\lambda$ transformation).}
				\label{fig:log-leb-clr}
		\end{subfigure}  \hfill
		    \begin{subfigure}[t]{0.4\textwidth}
        \centering
        \includegraphics[width=\textwidth]{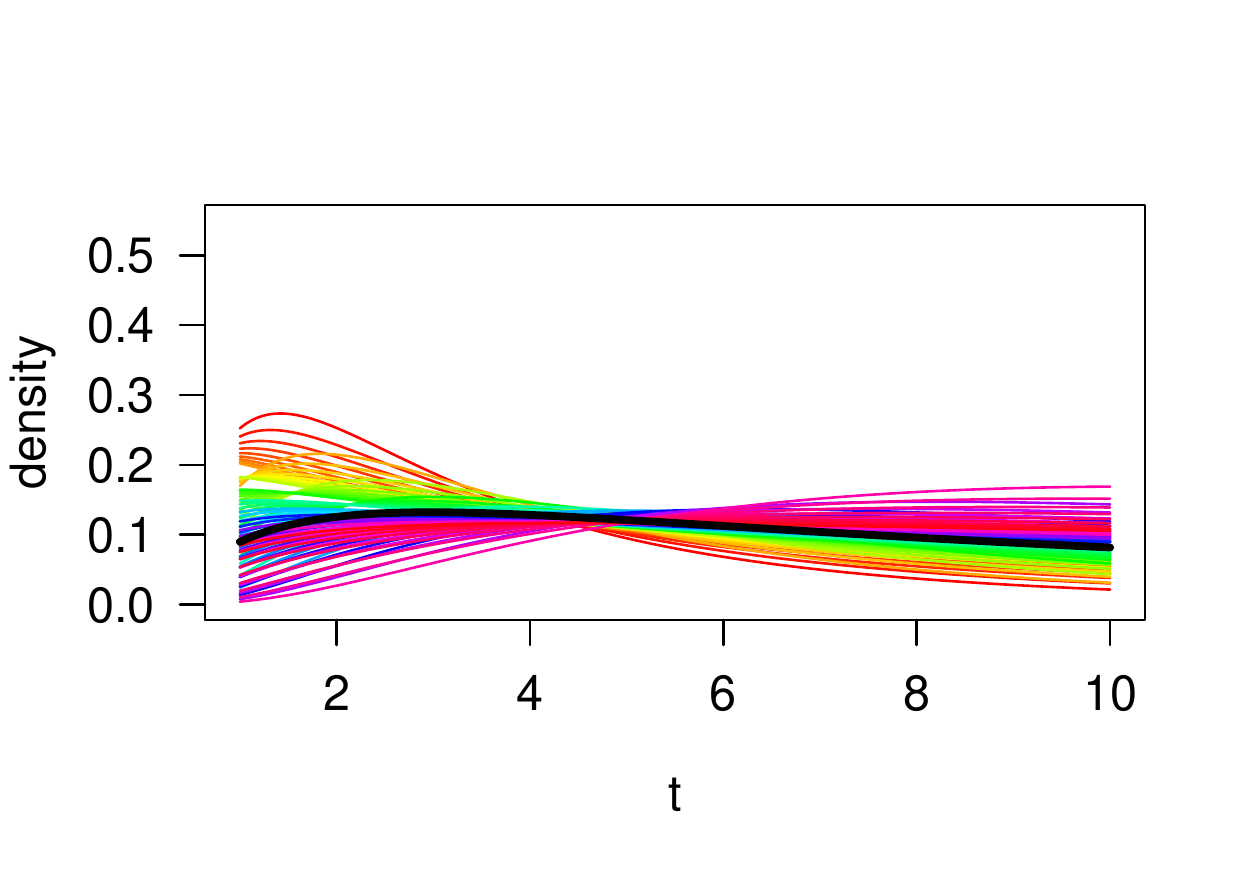}
        \caption{$\mathcal{B}^2$-unweighted $\mathsf{P}_0$-density functions on $\mathcal{B}^2(\lambda)$.}
				\label{fig:log-leb2}
    \end{subfigure}
		    \begin{subfigure}[t]{0.4\textwidth}
        \centering
        \includegraphics[width=\textwidth]{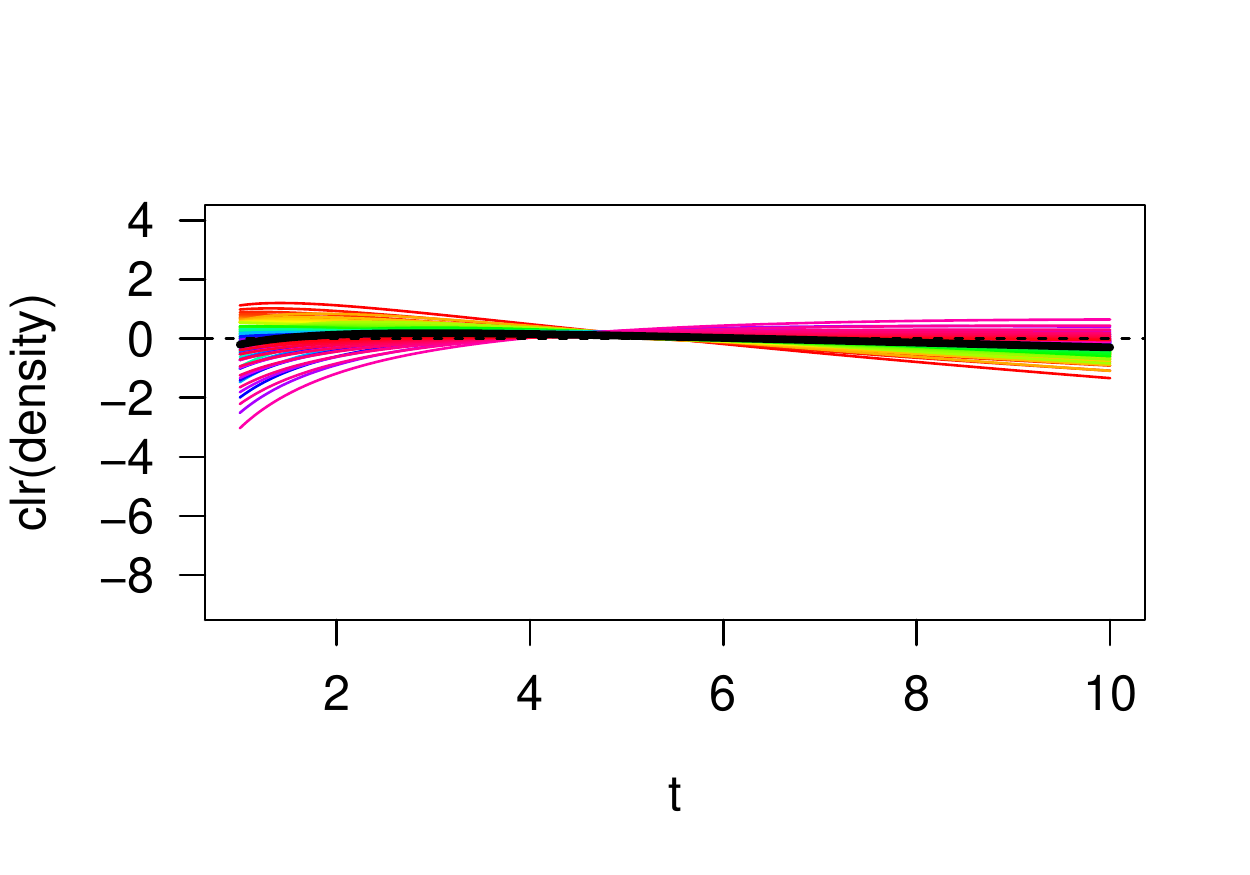}
        \caption{$\mathsf{P}_0$-density functions on $L^2_{0,\sqrt{\mathsf{P}_0}}(\lambda)$ (after $\clr_{u}$ transformation).}
				\label{fig:log-leb-clr2}
		\end{subfigure}
    \caption{Log-normal density functions w.r.t. the Lebesgue measure (panels (a)-(b)) and w.r.t. the uniform measure $\mathsf{P}_0$ (panels (c)-(d)), with parameters $\mu_i = 0.6+0.25\cdot(i-1)$ and $\sigma_j = 0.5+0.07\cdot(j-1)$ for $i, j=1, \ldots,9$, $I=\left[1,10\right]$. Black curves indicate the corresponding mean functions.
}
		\label{fig:Log-norm}
\end{figure}

To appreciate the influence of changing the scale of the reference measure, we set $\mathsf{P}_0$ to be the uniform measure on $I$, $\mathsf{P}_0 = \lambda/9$ (with density $p_0(t)=1/9$, for $t\in I$). The log-normal densities w.r.t. $\mathsf{P}_0$ are proportional to those in Figure \ref{fig:log-leb}, which is precisely the scaling effect induced by the reference measure. The $\clr_{\mathsf{P}_0}$ representations of the $\mathsf{P}_0$-densities coincide with those in Figure \ref{fig:log-leb-clr}; however, the former are embedded in $L^2(\mathsf{P}_0)$, whereas the latter do so in $L^2(\lambda)$. As such, a different scale is actually characterizing the two Bayes spaces.
The $\clr_{u}$ transformed densities, i.e. $y_i = (1/{3}) \cdot \clr_{\mathsf{P}_0}(f_{\mathsf{P}_0,i})$ -- which is an element of
$L^2_{0,\sqrt{\mathsf{P}_0}}$ --  are displayed in Figure \ref{fig:log-leb-clr2}. Here, the different scales of the two spaces are apparent. Finally, Figure \ref{fig:log-leb2} displays the $\mathcal{B}^2$-unweighted densities, i.e., $\omega^{-1}(f_{\mathsf{P}_0,i}) = (f_{\mathsf{P}_0,i})^{3}$, which are now elements of $\mathcal{B}^2(\lambda)$. A graphical representation like in Figure \ref{fig:log-leb2} may be very convenient in applications, as it allows to visually neglect the weighting of the domain when observing the figure.

Visual inspection of Figure \ref{fig:log-leb2} suggests that the scaling of the reference measure by $\alpha>1$ (or $\alpha<1$) results in a shrinkage (or expansion) of the corresponding Bayes space. The shrinkage of the Bayes space can be readily observed by comparing Figures \ref{fig:log-leb-clr} and \ref{fig:log-leb-clr2} (note that these representations are comparable because they are referred to the same reference $\lambda$). This is also well reflected in the covariance functions (Figure \ref{fig:Cov-log-norm}); indeed, the covariance structure is preserved but it differs in the scale. Here, the variability of the data, when these are embedded in $\mathcal{B}^2(\lambda)$ (resp. $\mathcal{B}^2(\mathsf{P}_0)$), is concentrated on the boundaries of the domain $I$. Particularly being more dominant in its left-hand side, where the densities display larger relative differences. Analogous conclusion can be derived from Figures \ref{fig:log-leb-clr} and \ref{fig:log-leb-clr2} respectively, but note that these graphs are interpreted in terms of absolute differences among curves in agreement with the $L^2$ geometry considered therein.
\begin{figure}
    \centering
        \includegraphics[width=0.35\textwidth]{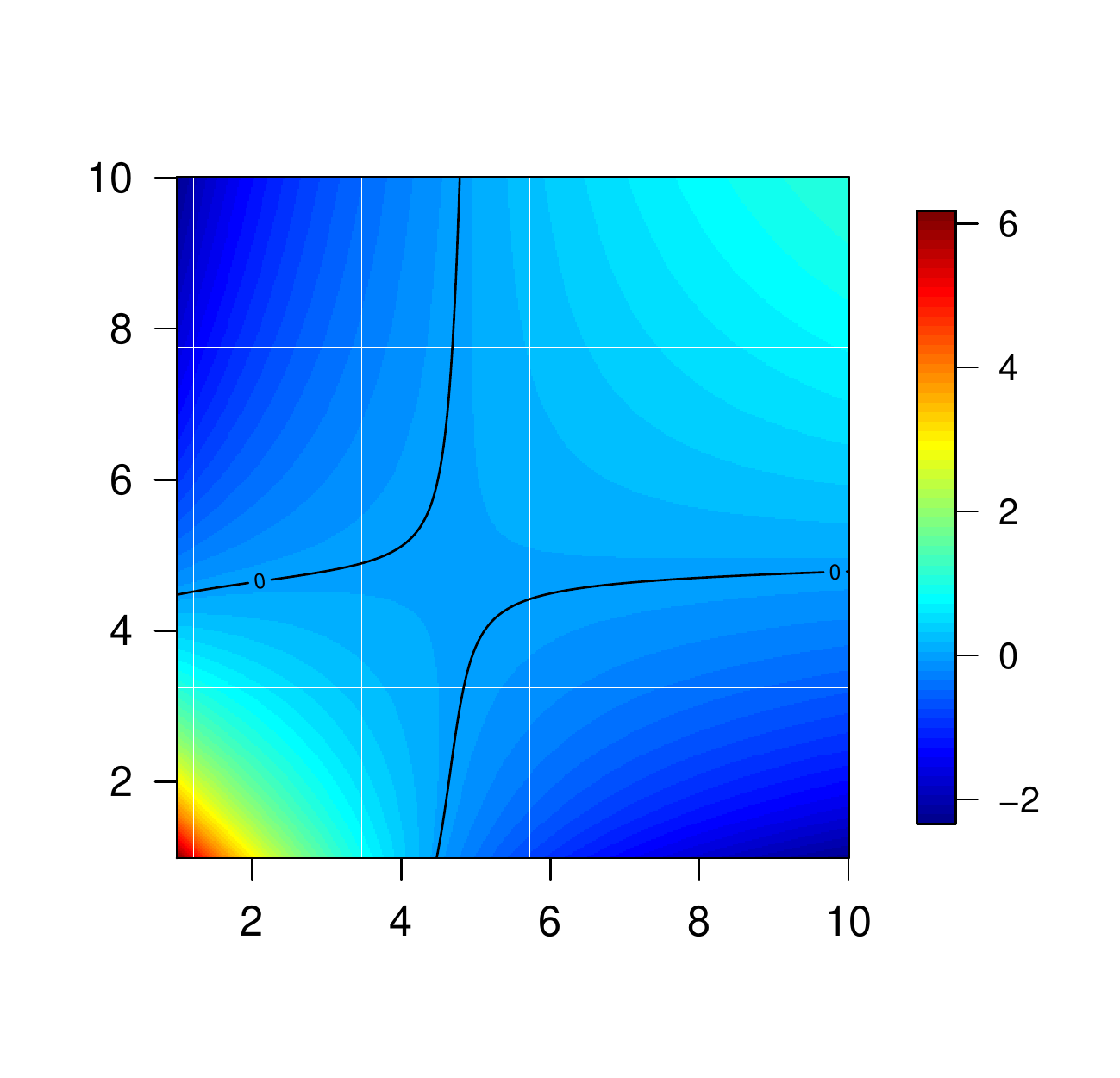}
        \includegraphics[width=0.35\textwidth]{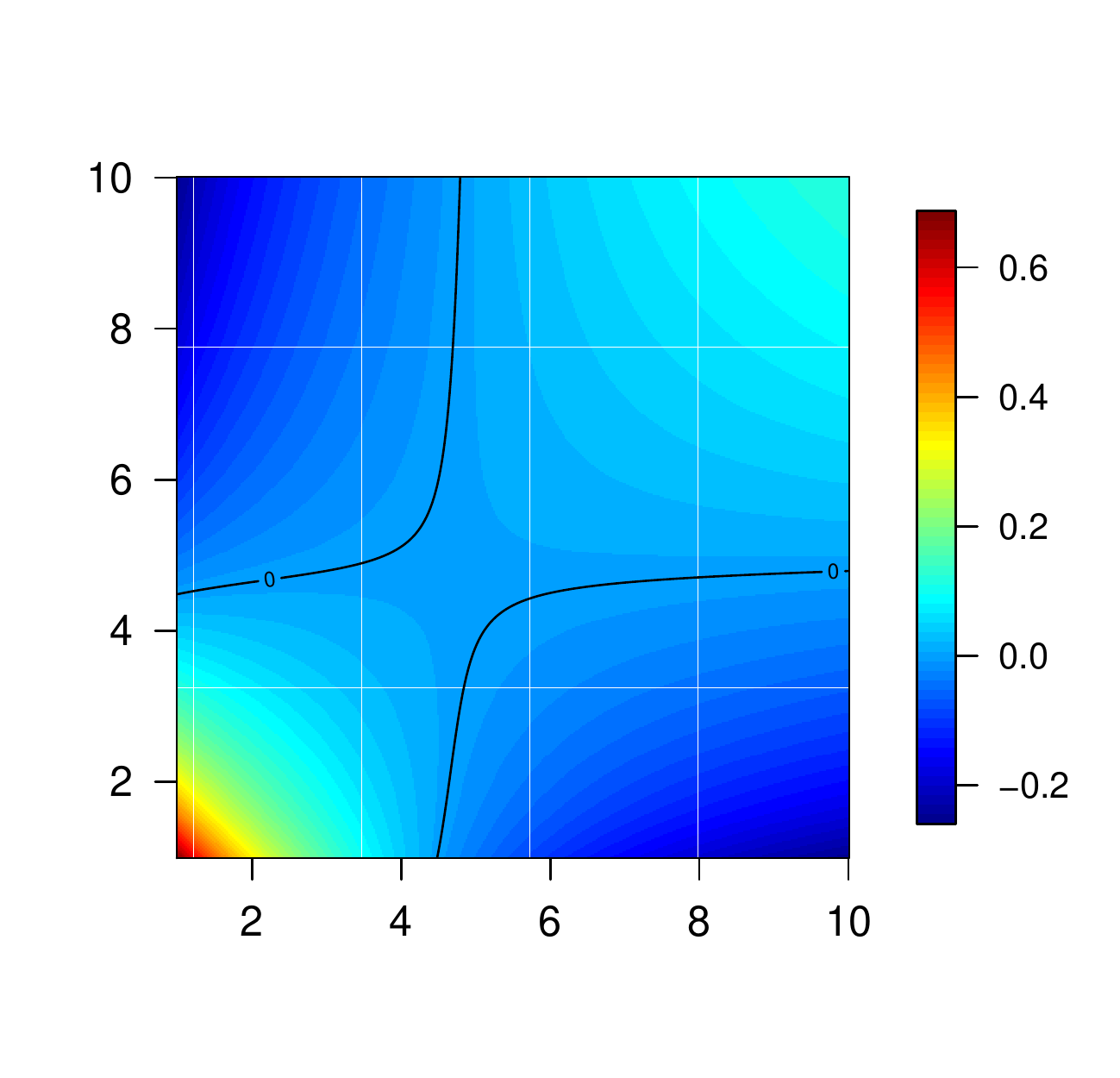}
    \caption{Covariance functions of log-normal $\lambda$-densities (left) and log-normal $\mathsf{P}_0$-densities (right). To appreciate the similarity between covariance structures, colors are \emph{not} given on the same scale.}
		\label{fig:Cov-log-norm}
\end{figure}

For the same log-normal densities, an exponential reference measure $\mathsf{P}^\delta$ was also considered, setting their densities to $p^{\delta}(t)=_{\mathcal{B}(\lambda)} \exp{\left\{-\delta \cdot t\right\}}, t \in I$, with $\delta$ in $\{0.25, 0.75, 1.25\}$.  Note that, for increasing values of $\delta$, the reference gives increasing weight to the left-hand side of the domain $I$. In order to obtain comparable results in terms of scales, the reference measures were all considered as normalized to unity. Figure \ref{fig:Log-normal-exp123} depicts the resulting log-normal densities w.r.t. $\mathsf{P}$,
\begin{equation*} \label{}
f_{\mathsf{P}}(t; \mu_i, \sigma_j) =_{\mathcal{B}(\mathsf{P})} \frac{1}{t}\cdot \exp \left\{-\frac{\ln t - \mu_i}{2\sigma_j^2} + \delta \cdot t \right\}, \quad t \in I,
\end{equation*}
as well as their counterparts in $L^2_0(\mathsf{P})$ and $L^2_{0,\sqrt{\mathsf{P}}}(\lambda)$. As expected, by down-weighting the right-hand side of the domain (i.e. increasing $\delta$), the variability in the tails on the right is eventually completely masked, whereas the opposite trend can be observed in the tails on the left. This is apparent when comparing the log-normal densities (Figure \ref{fig:log-norm-exp123-clr2}) and the corresponding covariance functions (Figures \ref{fig:Cov-log-norm} and \ref{fig:Log-normal-cov}). A further example on densities whose major source of (relative) variability is in the right-hand side of the domain is reported in the Supplementary Material, where truncated Weibull densities are considered.

\begin{figure}
    \centering
		\begin{subfigure}[b]{1\textwidth}
        \centering
        \includegraphics[width=0.31\textwidth]{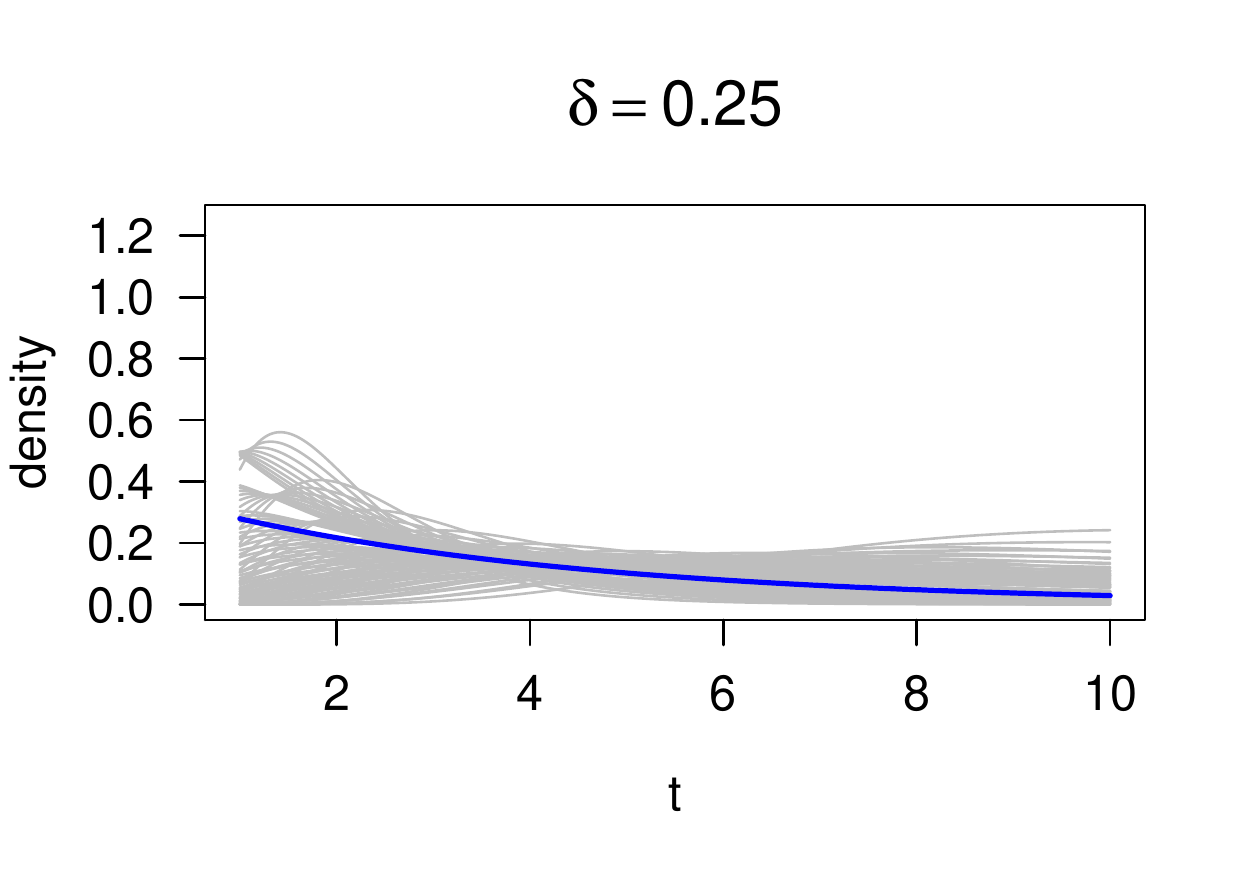}\includegraphics[width=0.31\textwidth]{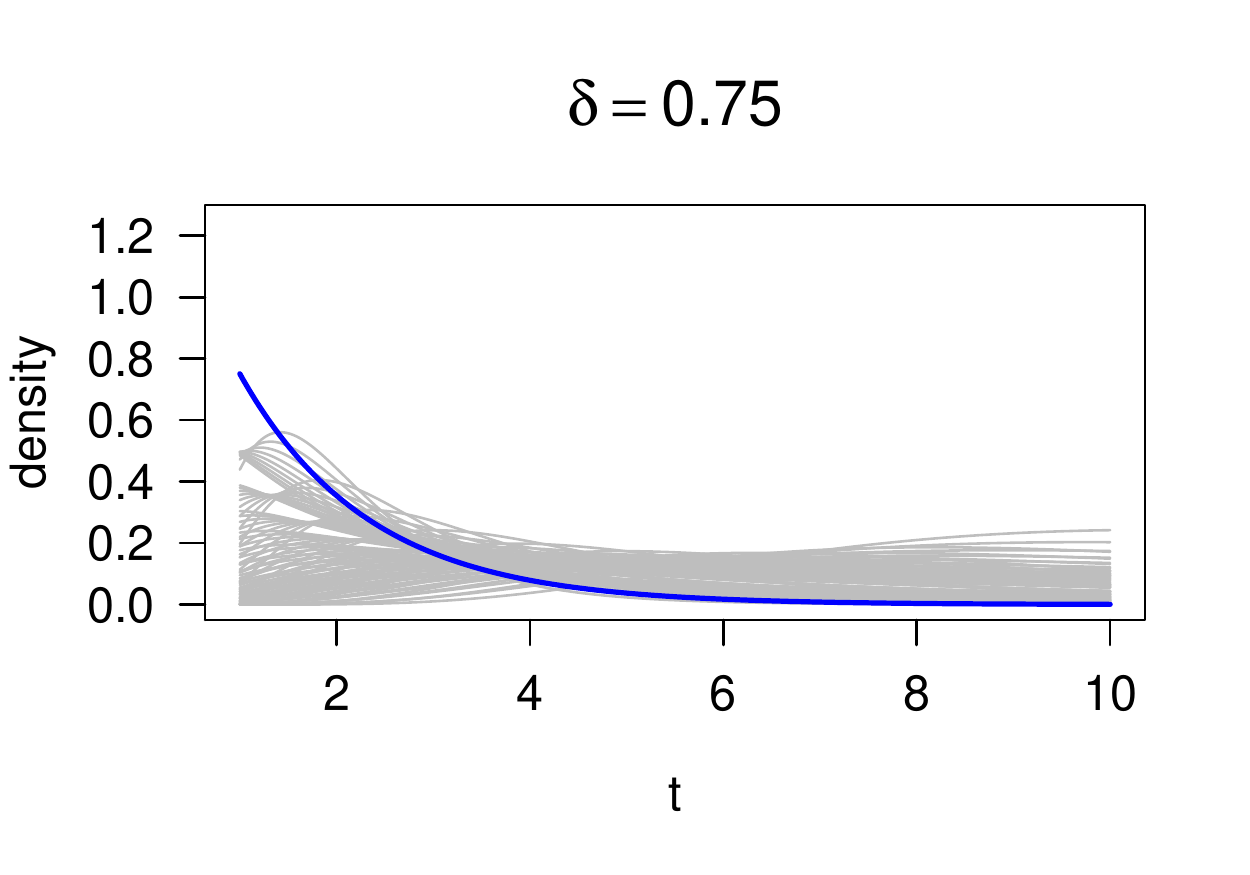}\includegraphics[width=0.31\textwidth]{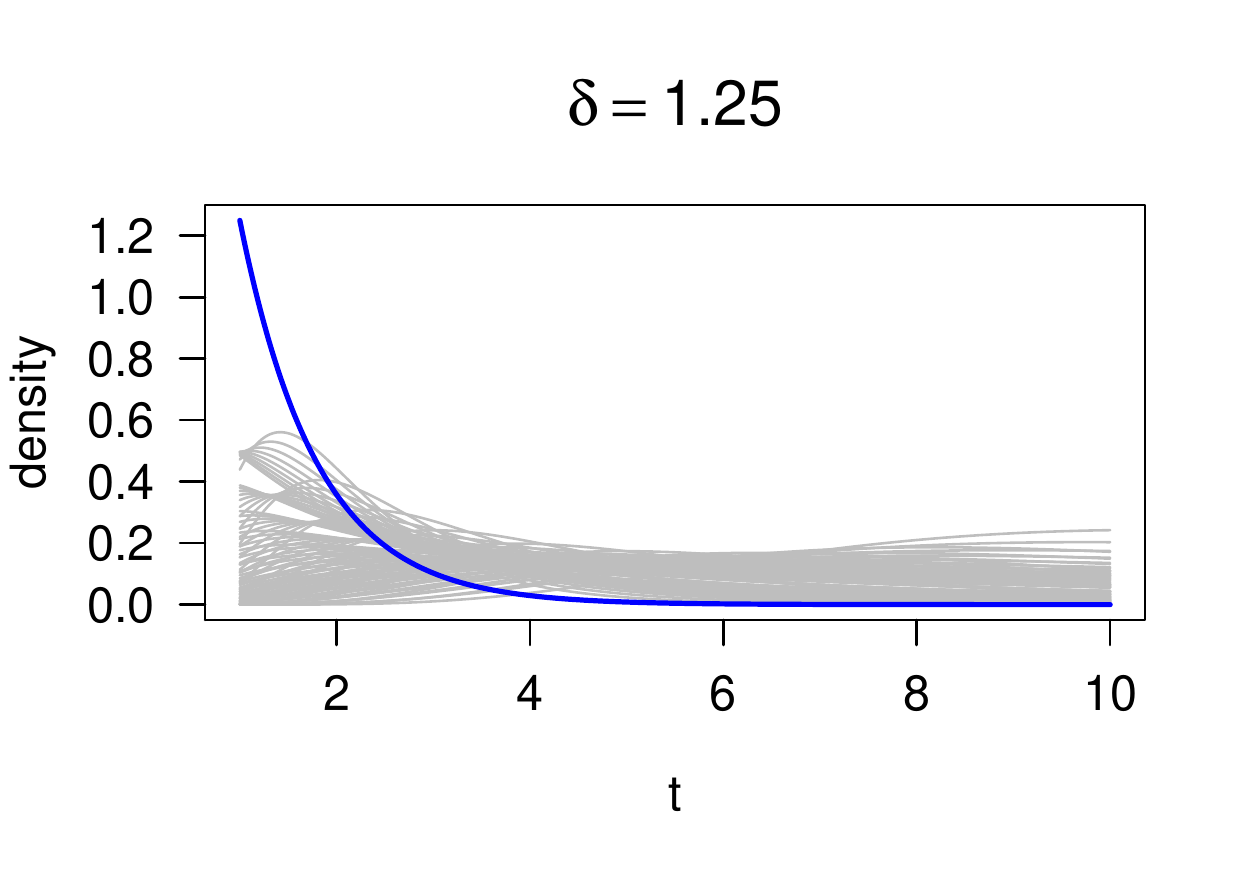}
        \caption{$\lambda$-density functions on $\mathcal{B}^2(\lambda)$ together with the exponential reference densities $\mathsf{P}^\delta$(blue curves).}
		\end{subfigure}
    \begin{subfigure}[b]{1\textwidth}
        \centering
        \includegraphics[width=0.31\textwidth]{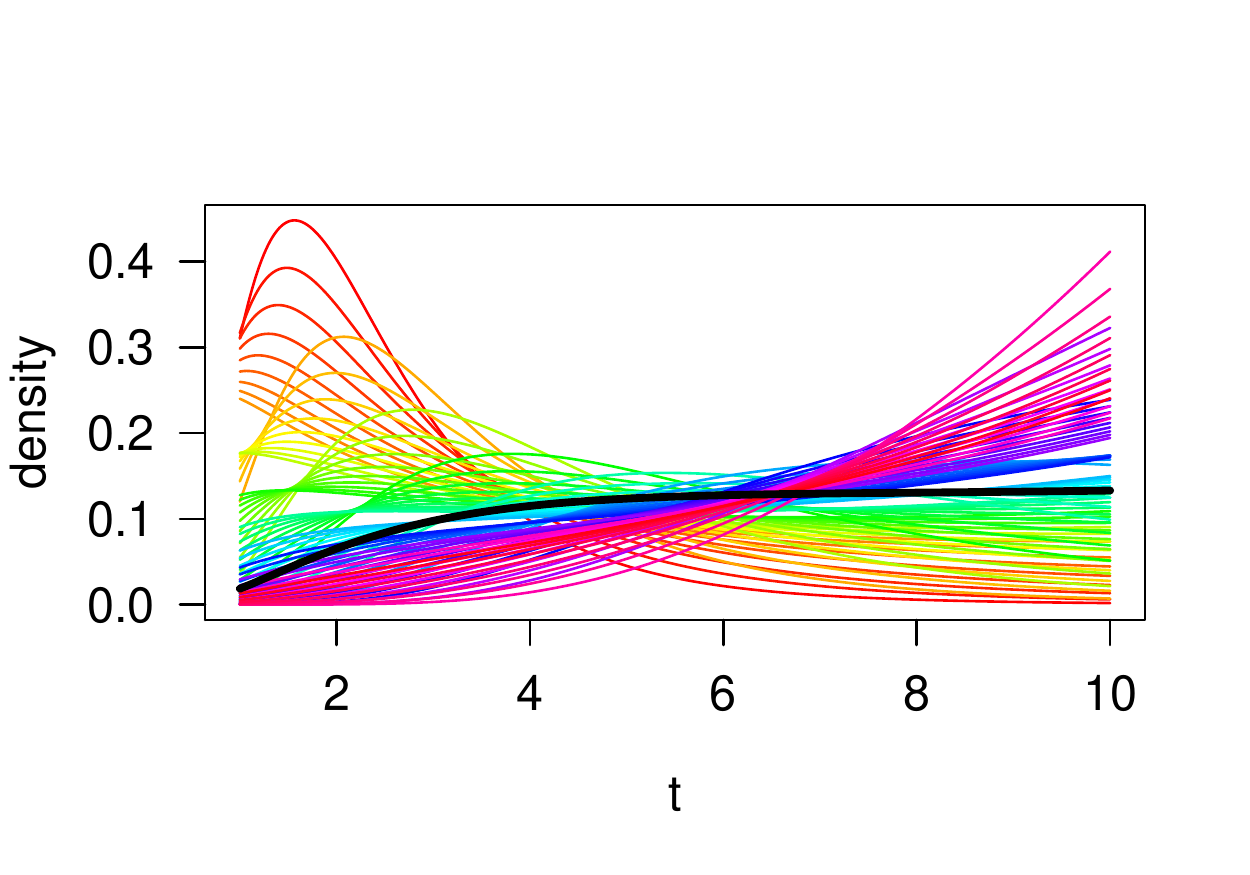}\includegraphics[width=0.31\textwidth]{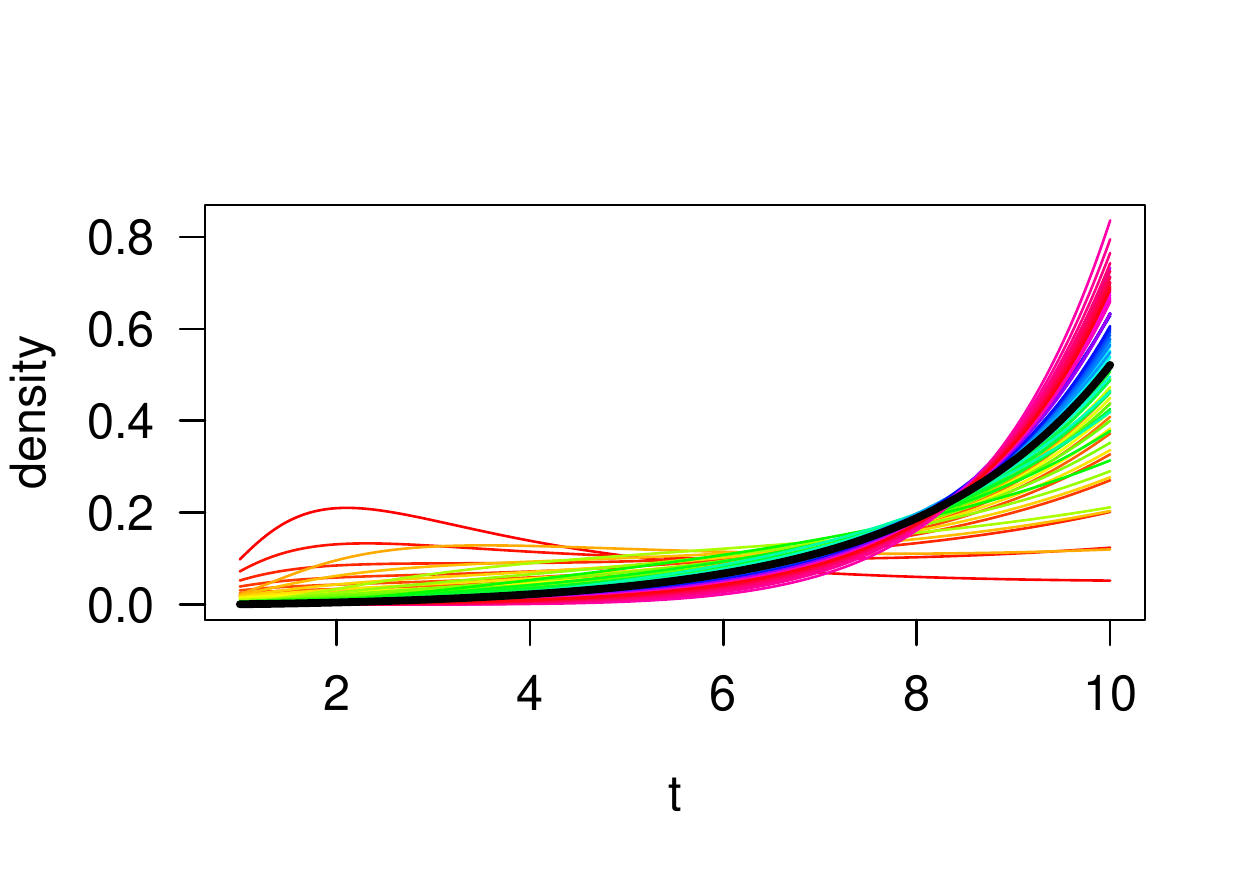}\includegraphics[width=0.31\textwidth]{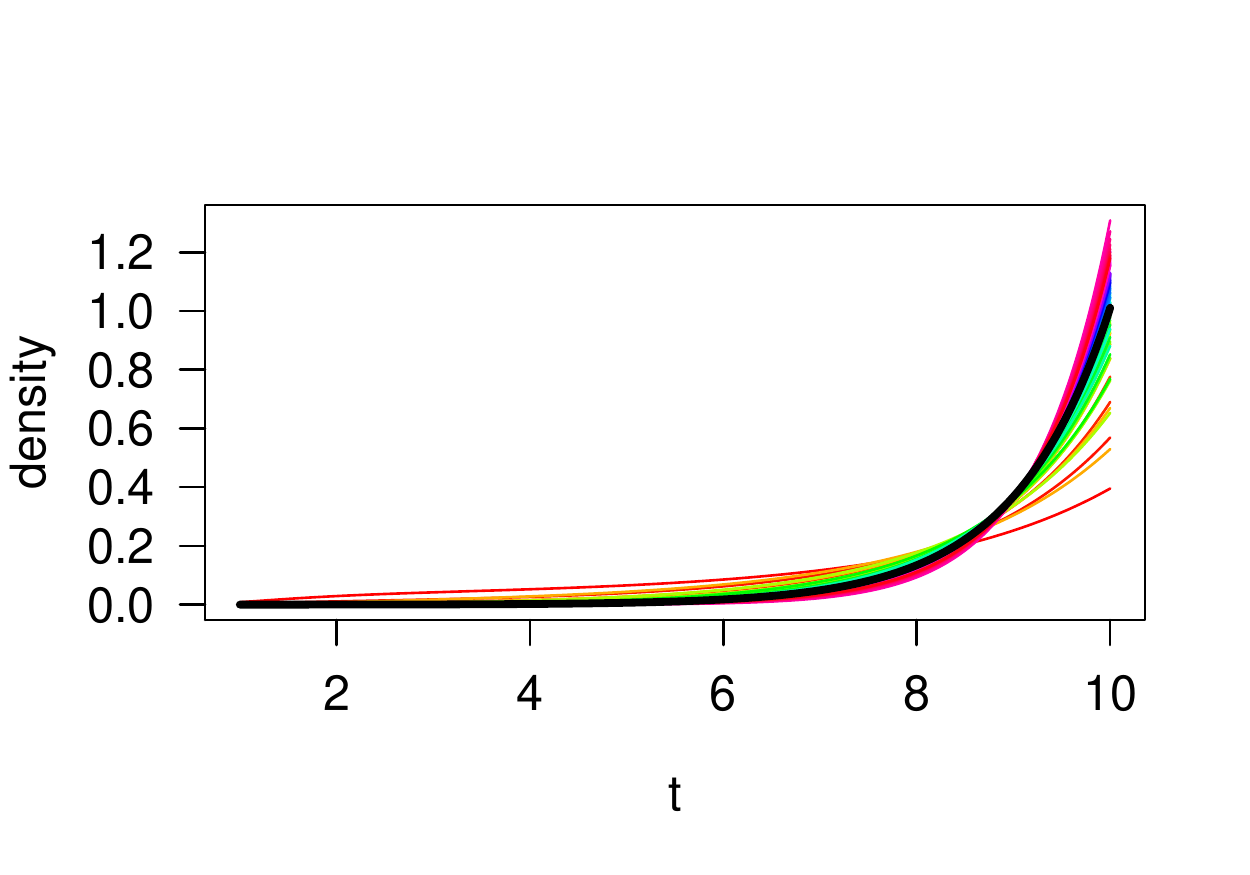}
        \caption{$\mathsf{P}$-density functions on $\mathcal{B}^2({\mathsf{P}})$ ($f_{\mathsf P, ij}$), for the exponential reference densities $\mathsf{P}=\mathsf{P}^\delta$.}
    \end{subfigure}
		    \begin{subfigure}[b]{1\textwidth}
        \centering
        \includegraphics[width=0.31\textwidth]{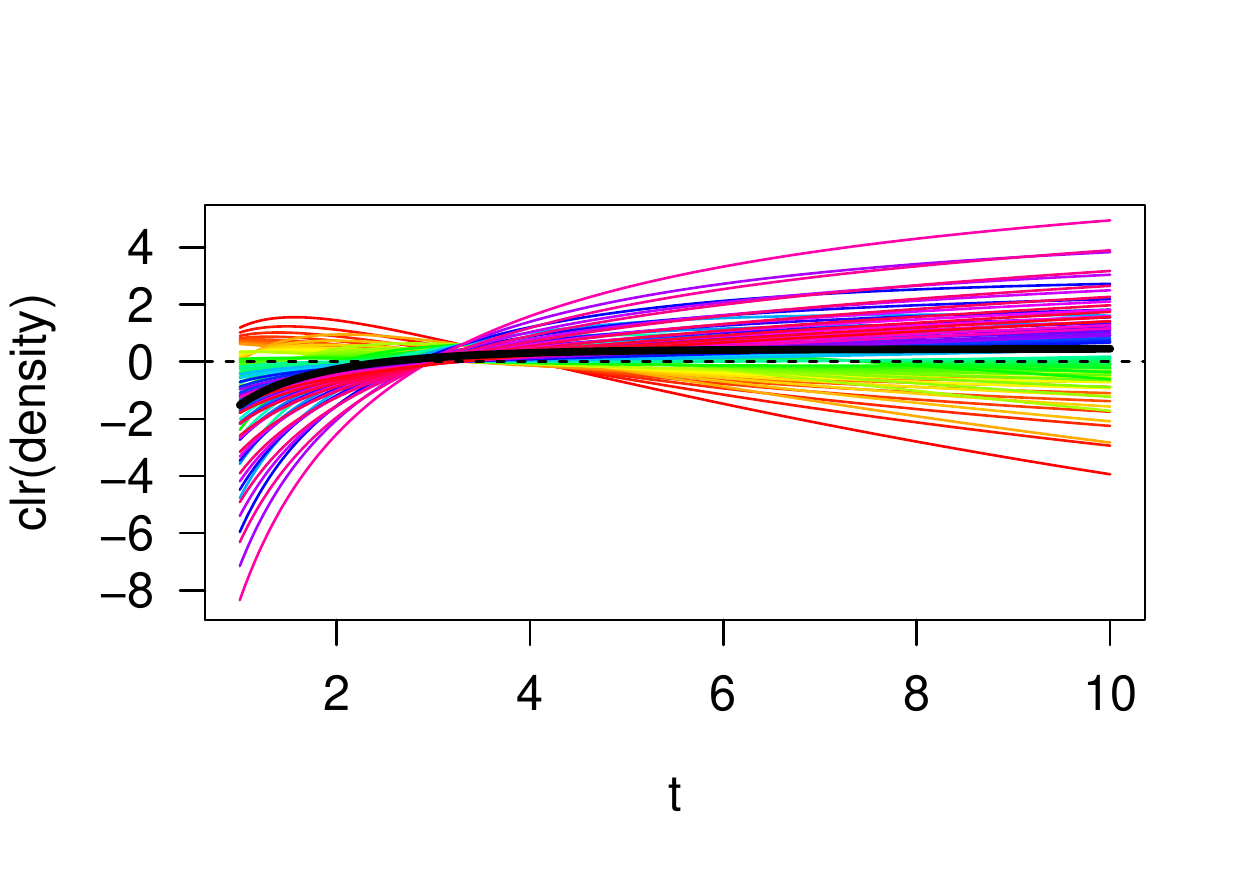}\includegraphics[width=0.31\textwidth]{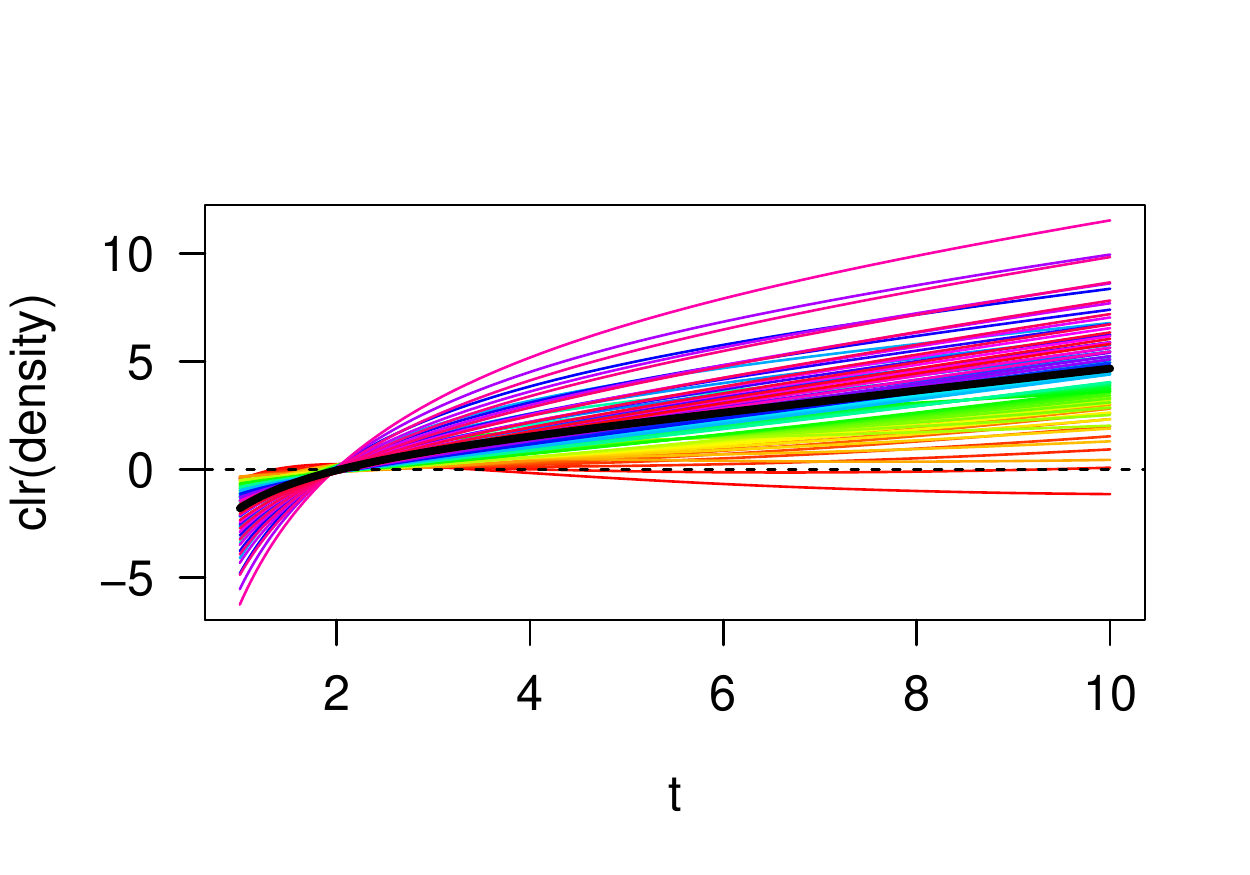}\includegraphics[width=0.31\textwidth]{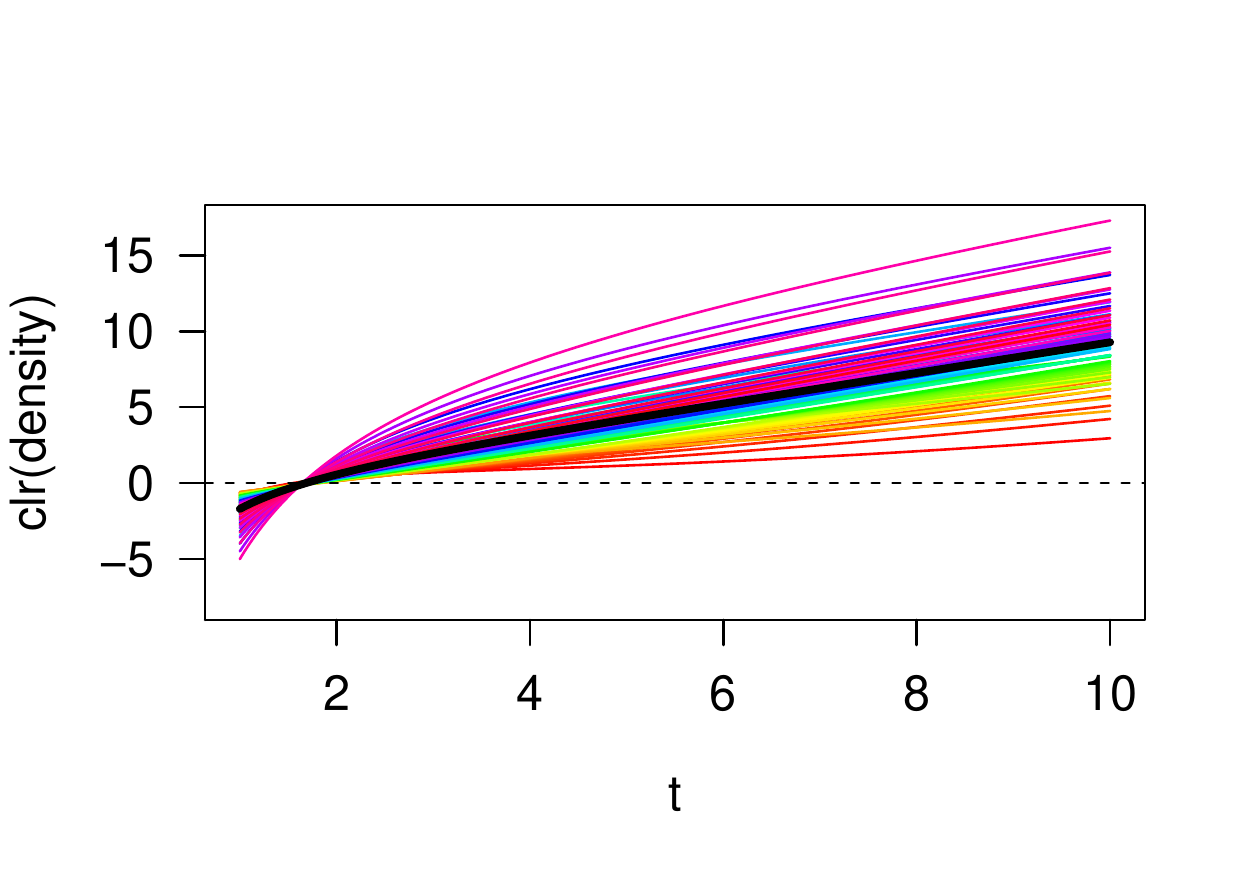}
        \caption{$\clr_{\mathsf{P}}$ transformation of the $\mathsf{P}$-density functions on $L^2({\mathsf{P}})$ ($\clr_{\mathsf{P}}(f_{\mathsf P, ij})$), for $\mathsf{P}=\mathsf{P}^\delta$. }
		\end{subfigure}
		    \begin{subfigure}[b]{1\textwidth}
        \centering
        \includegraphics[width=0.31\textwidth]{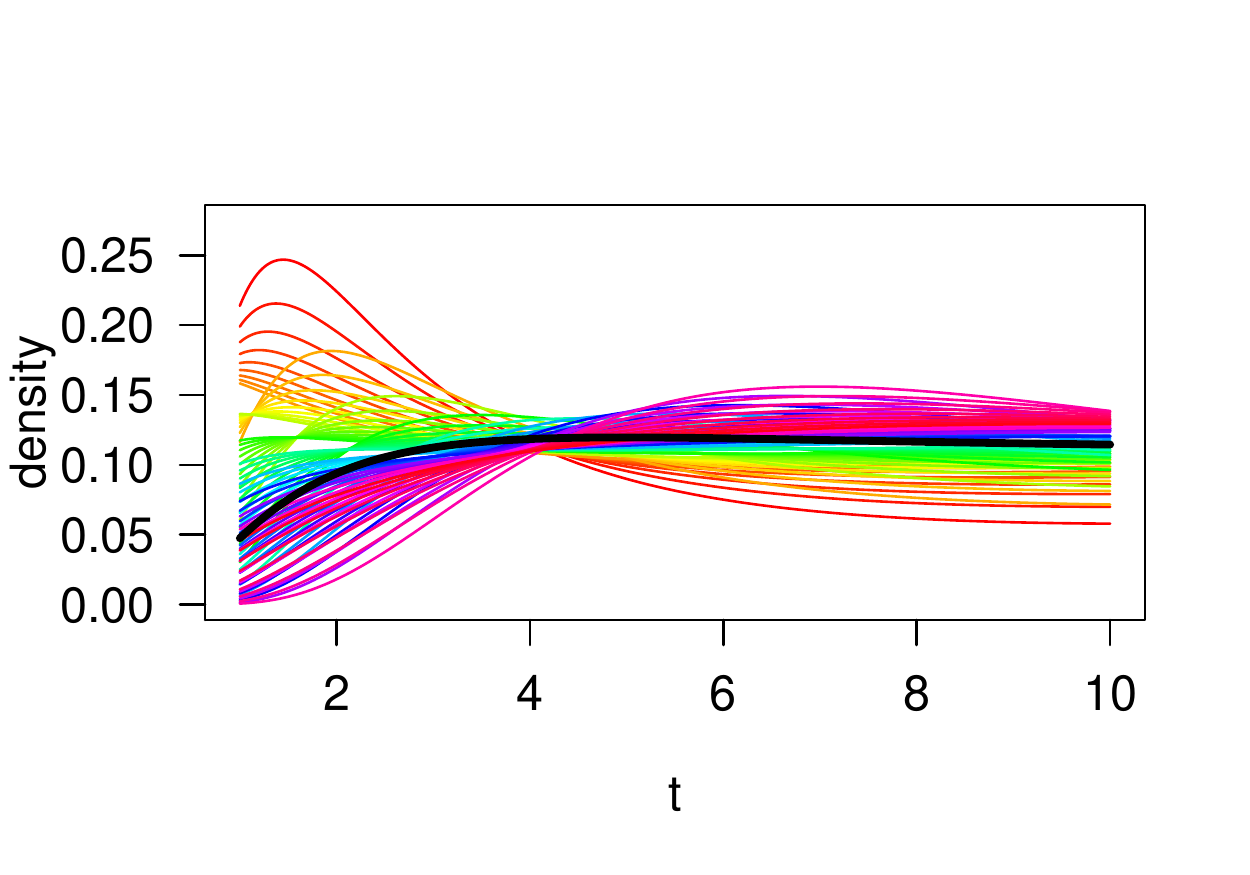}\includegraphics[width=0.31\textwidth]{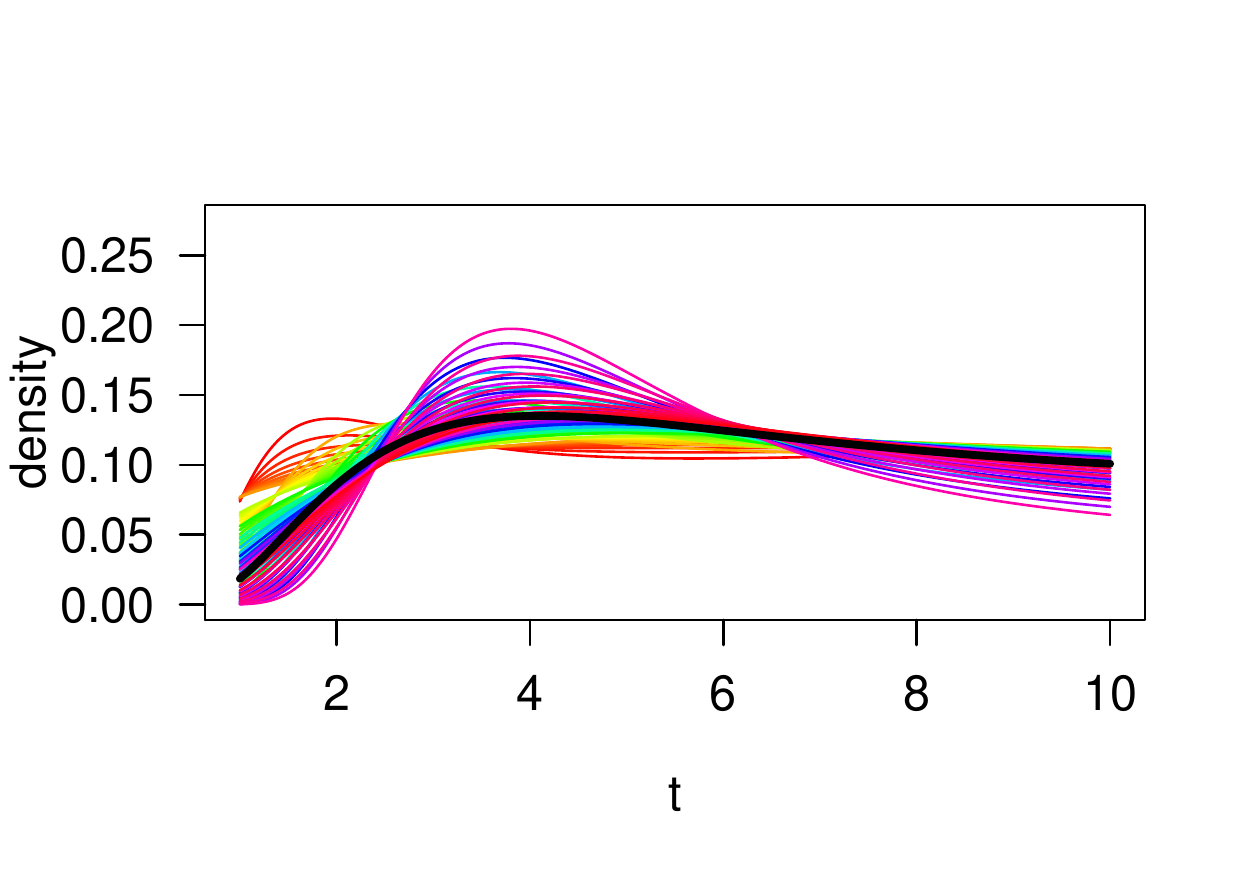}\includegraphics[width=0.31\textwidth]{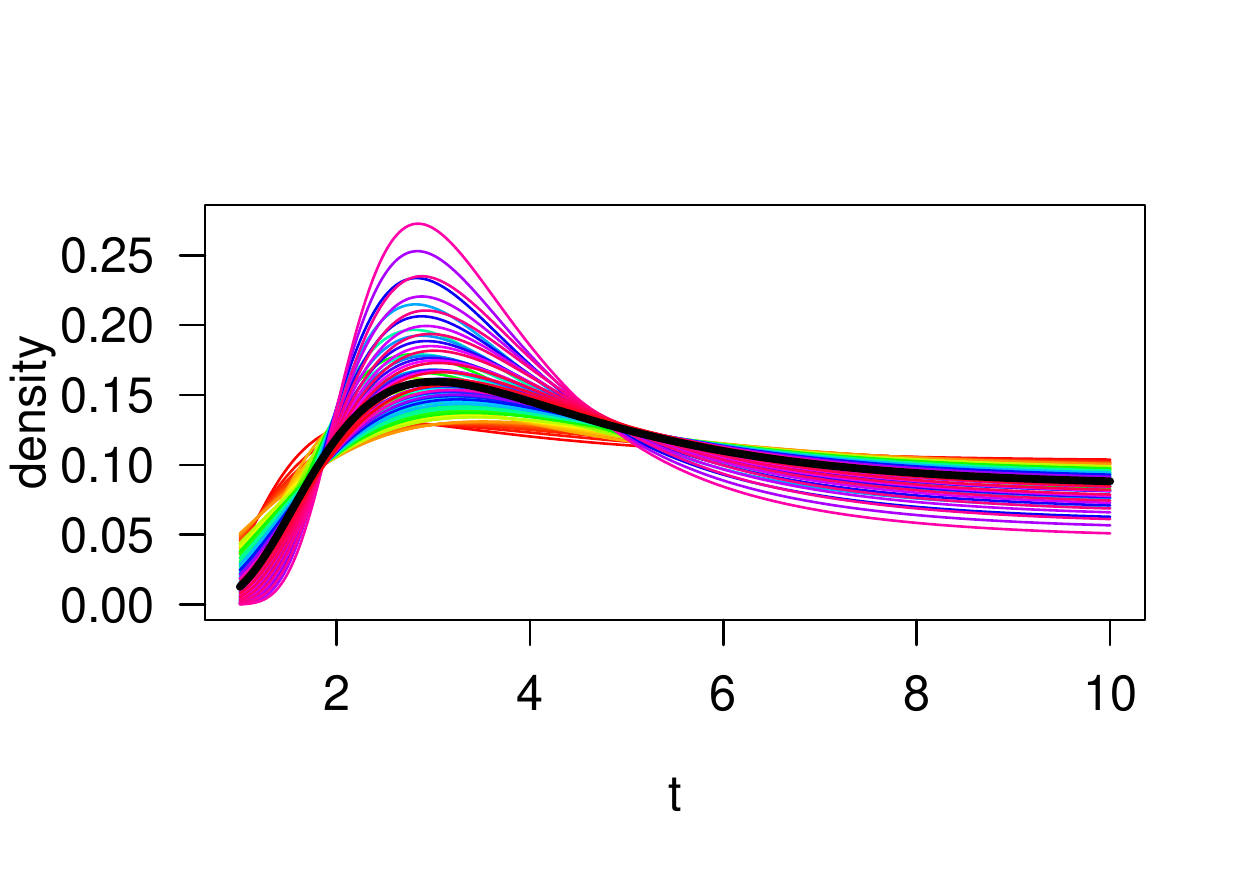}
				\caption{$\mathcal{B}^2$-unweighted version of $\mathsf{P}$-density functions on $\mathcal{B}^2(\lambda)$ (obtained as $\omega^{-1}(f_{\mathsf P, ij})$), for $\mathsf{P}=\mathsf{P}^\delta$.}
				\label{fig:log-norm-exp123-orig}
		\end{subfigure}
		    \begin{subfigure}[b]{1\textwidth}
        \centering
        \includegraphics[width=0.31\textwidth]{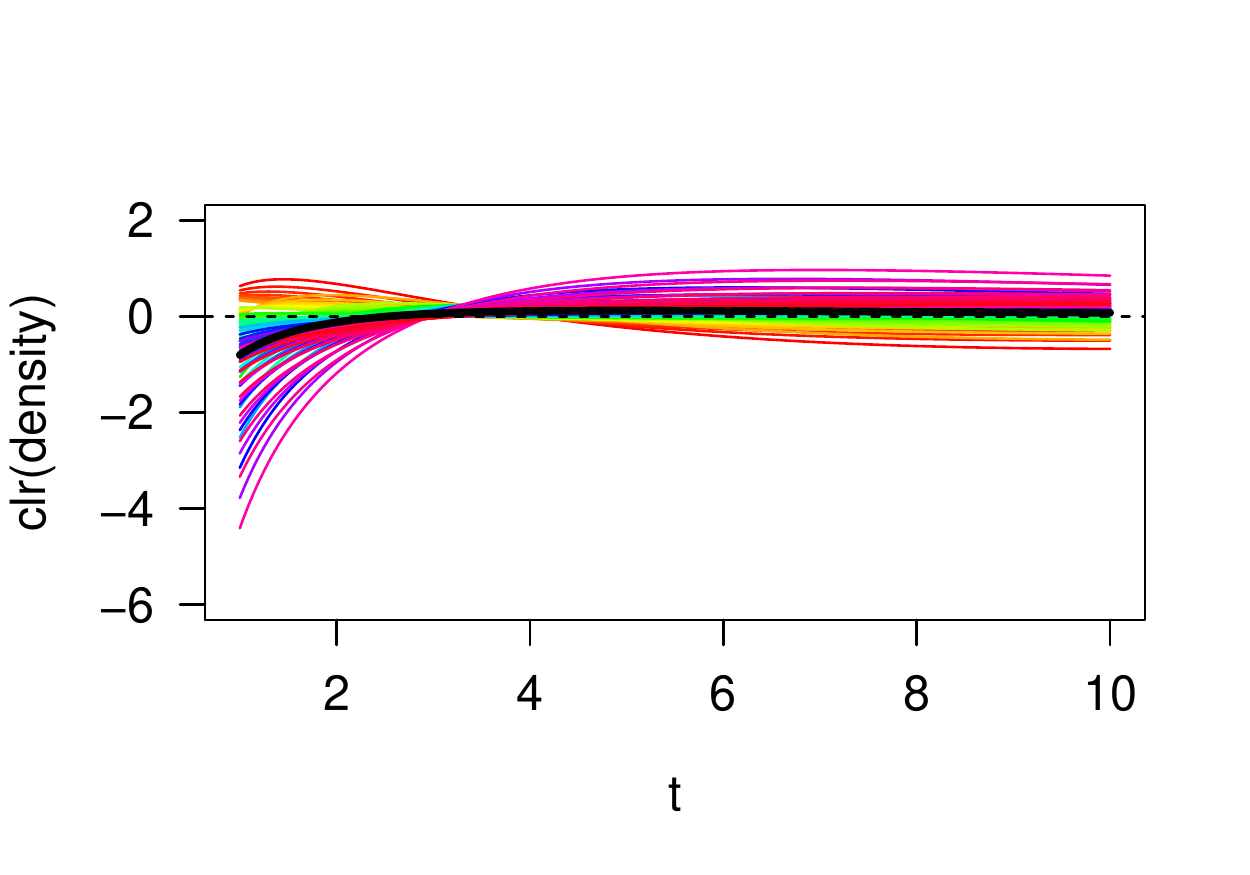}\includegraphics[width=0.31\textwidth]{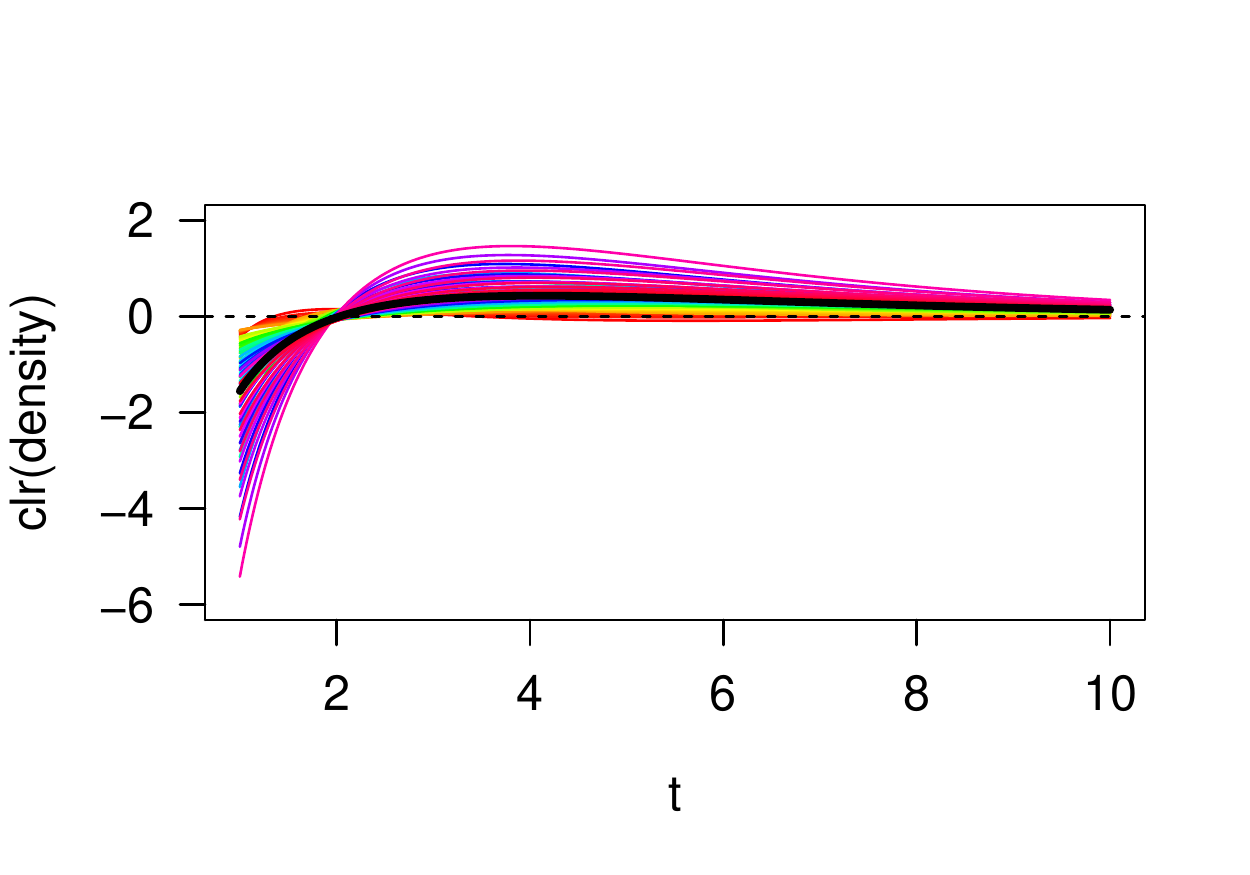}\includegraphics[width=0.31\textwidth]{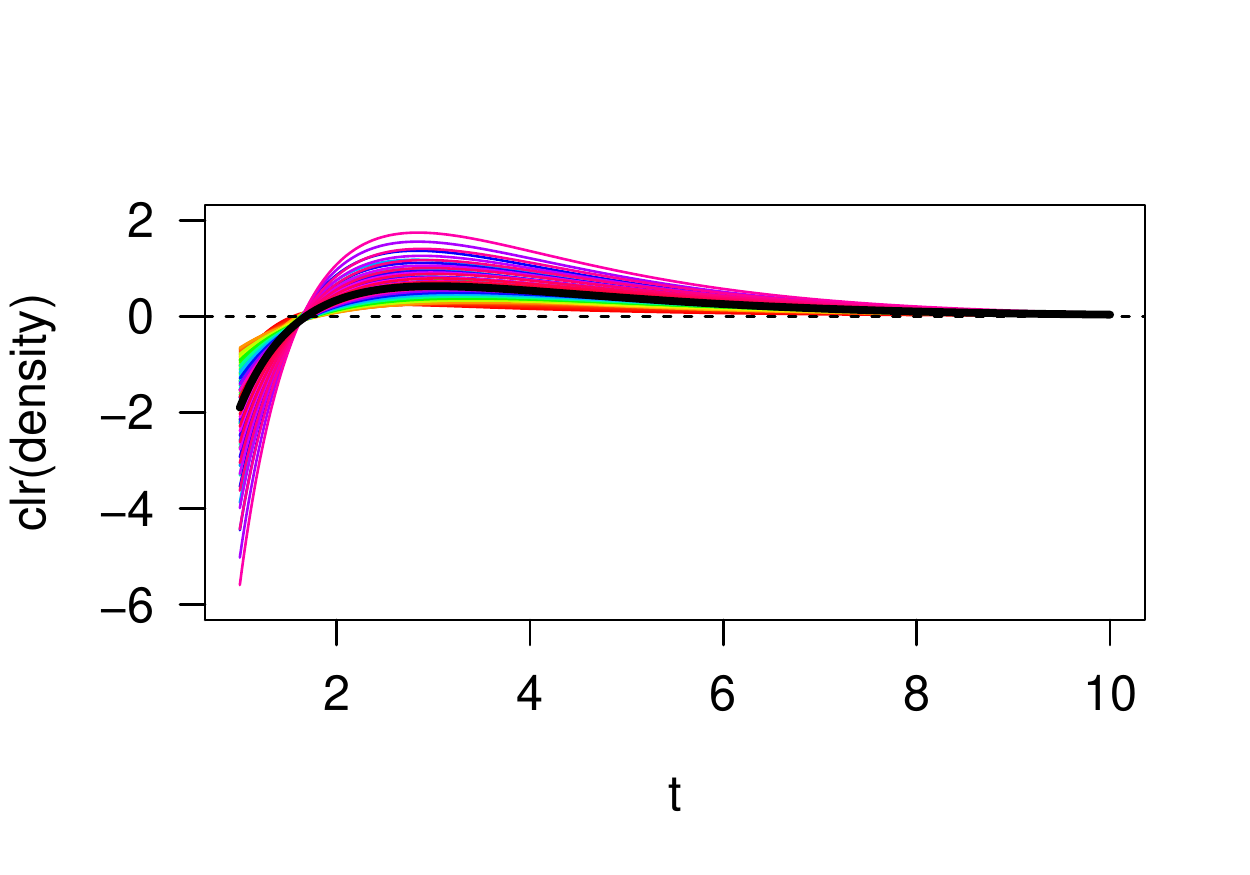}
        \caption{$\clr_u$ transformation of $\mathsf{P}$-density functions in $L^2_{0,\sqrt{\mathsf{P}}}(\lambda)$ (obtained as $\clr_{u}(f_{\mathsf P, ij})$), for $\mathsf{P}=\mathsf{P}^\delta$.}
				\label{fig:log-norm-exp123-clr2}
		\end{subfigure}
    \caption{Log-normal density functions with respect to exponential reference measures with $\delta = 0.25$ (first column), $\delta = 0.75$ (second column) and $\delta = 1.25$ (third column) for parameters $\mu_i = 0.6+0.25\cdot(i-1)$ and $\sigma_j = 0.5+0.07\cdot(j-1)$ for $i, j=1, \ldots,9$ on $I=\left[1,10\right]$.}
		\label{fig:Log-normal-exp123}
\end{figure}

\begin{figure}
    \centering
		    \begin{subfigure}[b]{0.32\textwidth}
        \centering
        \includegraphics[width=\textwidth]{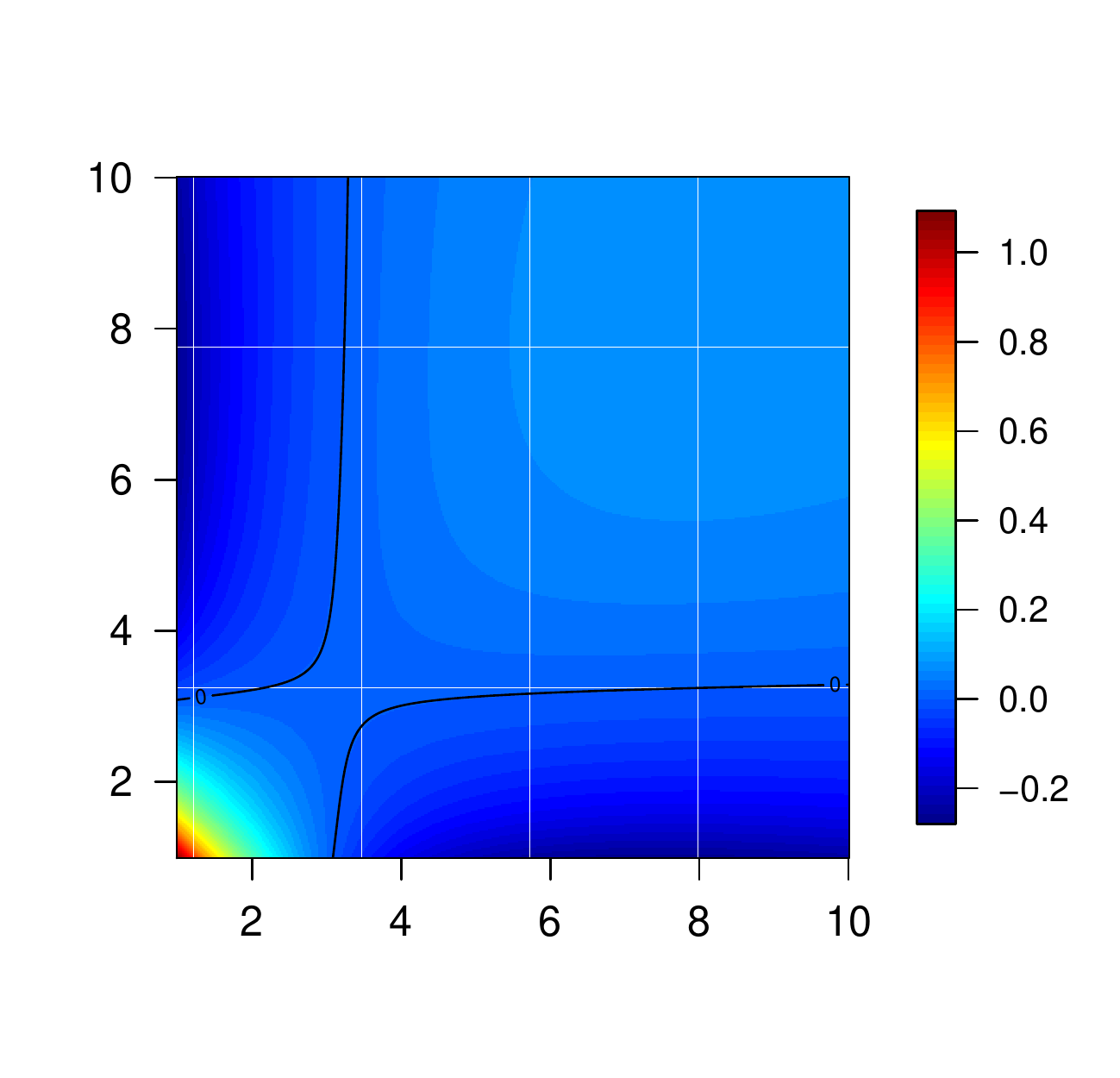}
        \caption{$\mathsf{P}^\delta$, $\delta = 0.25$}
    \end{subfigure}
		    \begin{subfigure}[b]{0.32\textwidth}
        \centering
        \includegraphics[width=\textwidth]{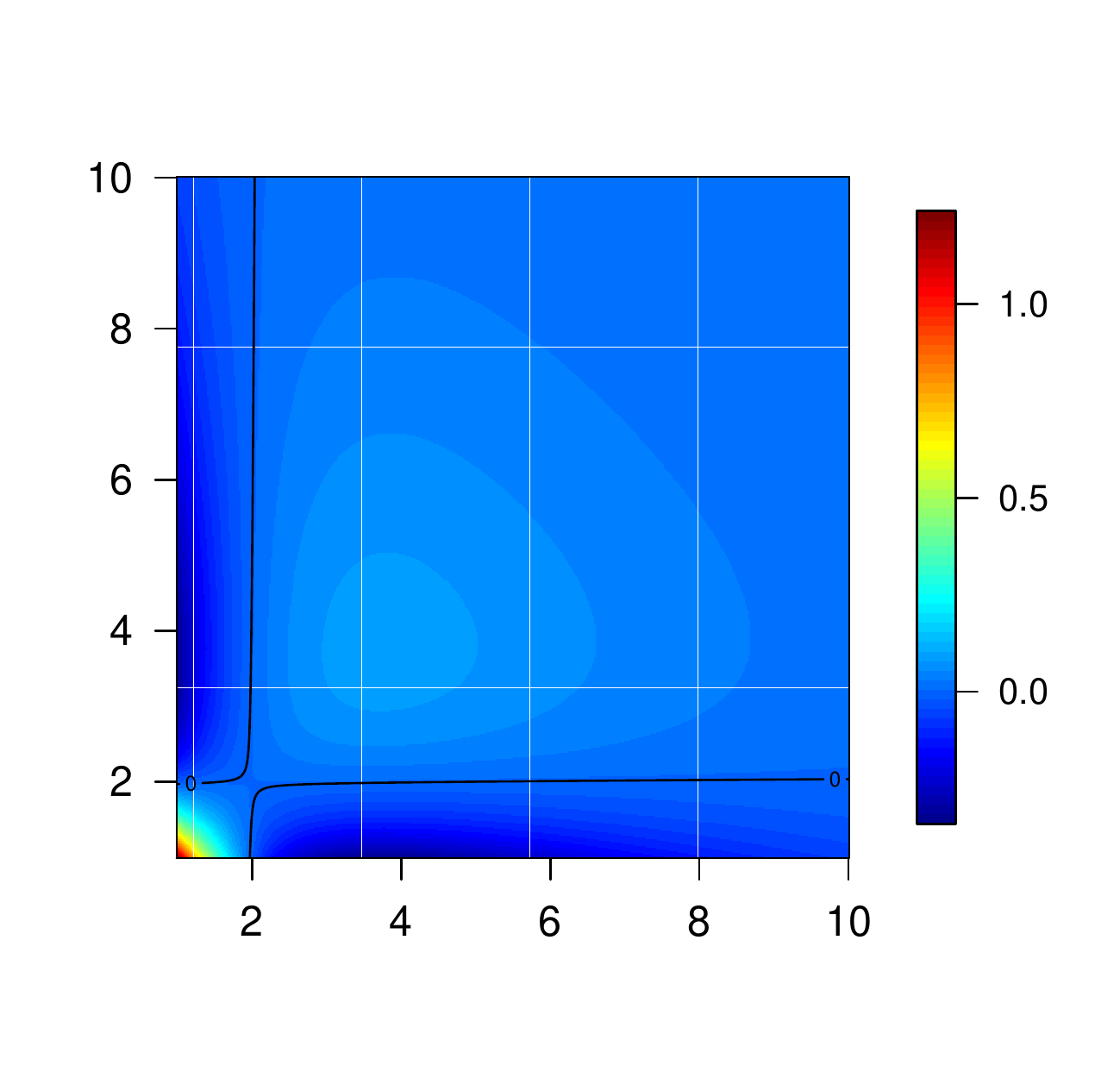}
        \caption{$\mathsf{P}^\delta$, $\delta = 0.75$}
		\end{subfigure}
		    \begin{subfigure}[b]{0.32\textwidth}
        \centering
        \includegraphics[width=\textwidth]{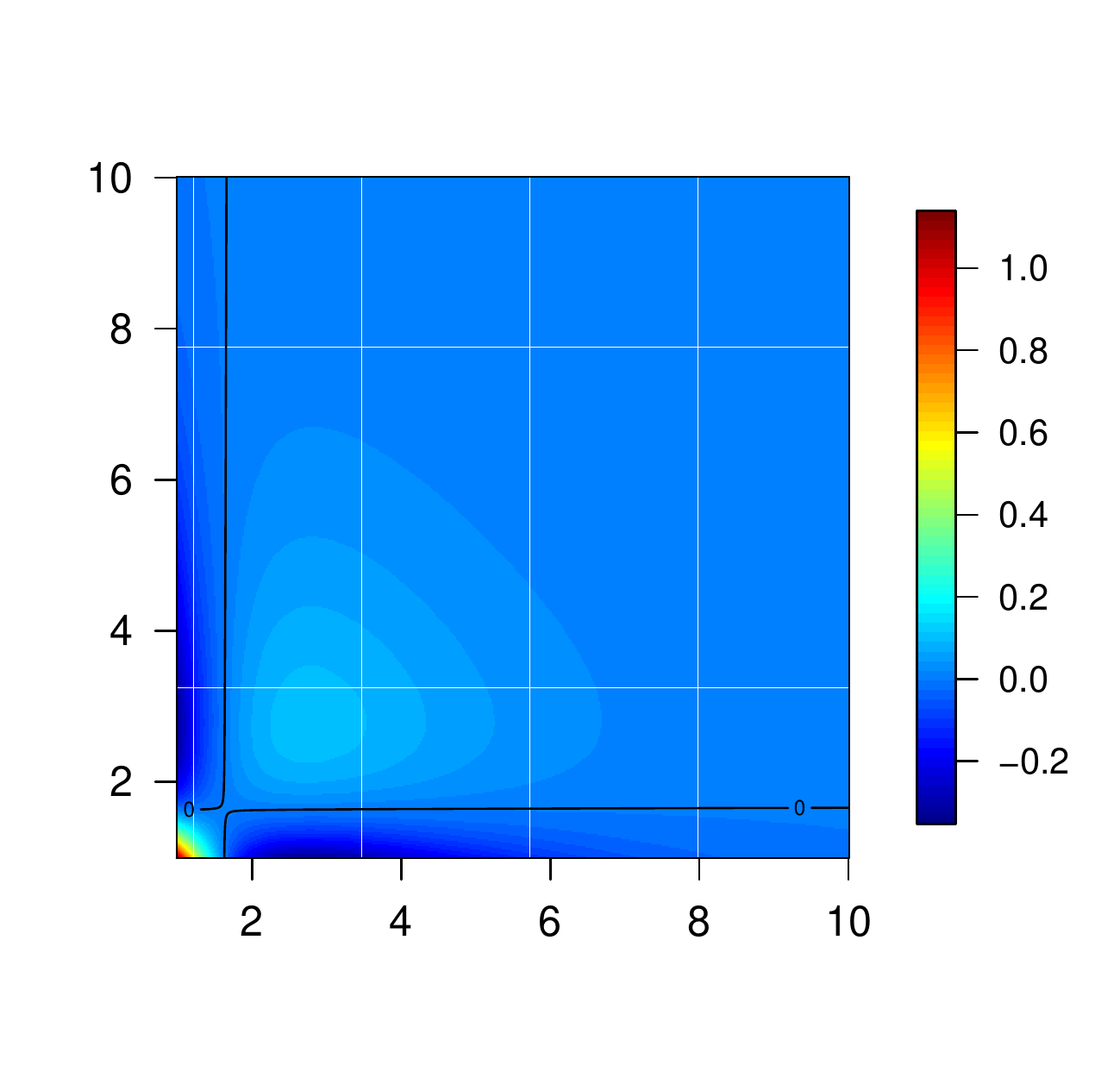}
        \caption{$\mathsf{P}^\delta$, $\delta = 1.25$}
		\end{subfigure}
    \caption{Comparison of covariance functions for log-normal densities w.r.t. the exponential reference measure for different values of parameter $\delta$. To appreciate the patterns of the covariance structures, colors are \emph{not} given on the same scale.
}
		\label{fig:Log-normal-cov}
\end{figure}

\subsection{Change of the reference measure: consequences on SFPCA}
In this subsection, the effect of changing the reference measure is further analyzed in the context of simplicial principal component analysis (SFPCA, \cite{hron16}). The same set of log-normal densities used in Subsection \ref{subsec:sim_log_wei} is considered, by setting the reference measure to either uniform or exponential distribution.

Both datasets considered in Section \ref{subsec:sim_log_wei} belong to a 2-parametric exponential family which forms an affine subspace of the Bayes space whose dimension is precisely the number of parameters \cite{boogaart10}. This feature was highlighted in \cite{hron16} for the case of a Lebesgue reference measure. Accordingly, the original spaces can be reconstructed (without lost of information) by the first two SFPCs  (SFPC$_1$ and SFPC$_2$), forming an orthonormal basis of the corresponding affine subspace. One may expect that changing the reference measure for densities in the exponential family will have an impact on the wSFPCA while preserving the data dimensionality. We also note that the results of wSFPCA under a uniform reference measure are expected to be just a rescaling of those that would be obtained with the SFPCA of \cite{hron16}.

Figures \ref{fig:sfpca-lognor-unif} and \ref{fig:sfpca-lognor-exp123} report the wSFPCA results on the log-normal densities when uniform and exponential reference measures are used respectively. As expected, the first two wSFPCs represent the total variability of the dataset in all cases.
\begin{figure}
    \centering
    \begin{subfigure}[t]{0.4\textwidth}
        \centering
        \includegraphics[width=\textwidth]{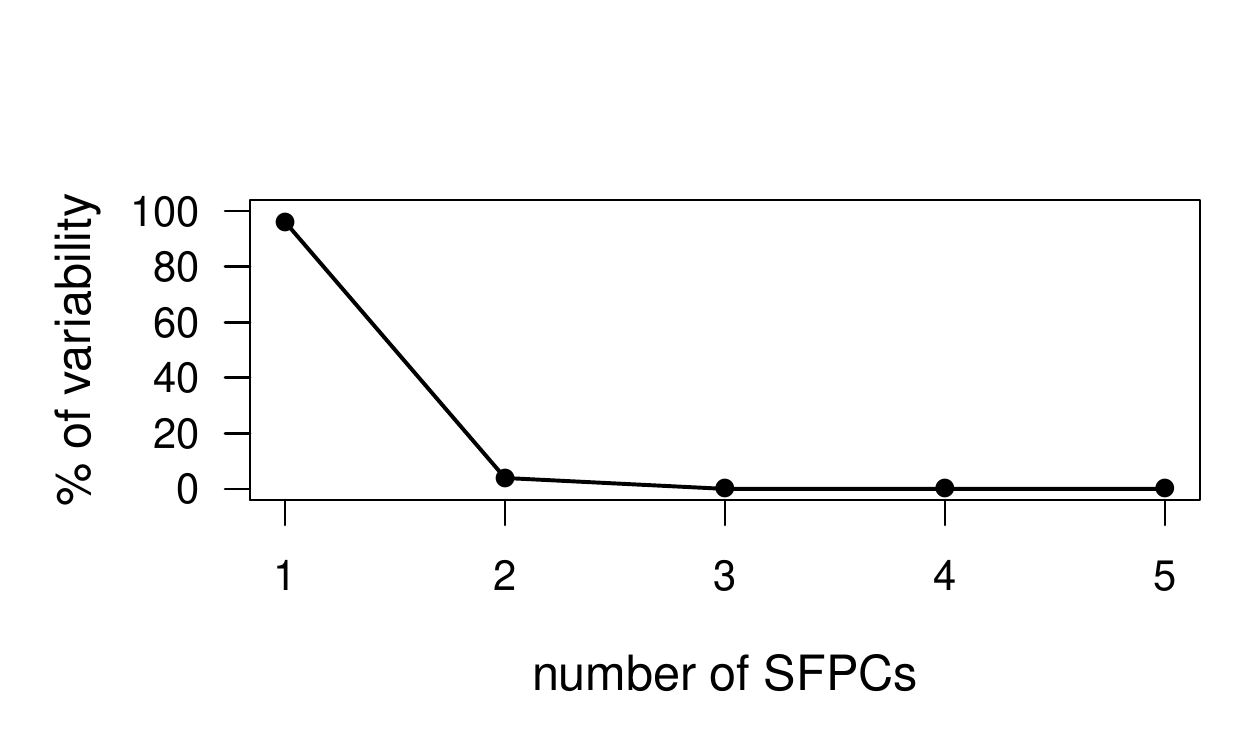}
				\begin{flushleft}
				\caption{Explained variability.}
				\end{flushleft}
    \end{subfigure}
	\begin{subfigure}[t]{0.4\textwidth}
        \centering
        \includegraphics[width=\textwidth]{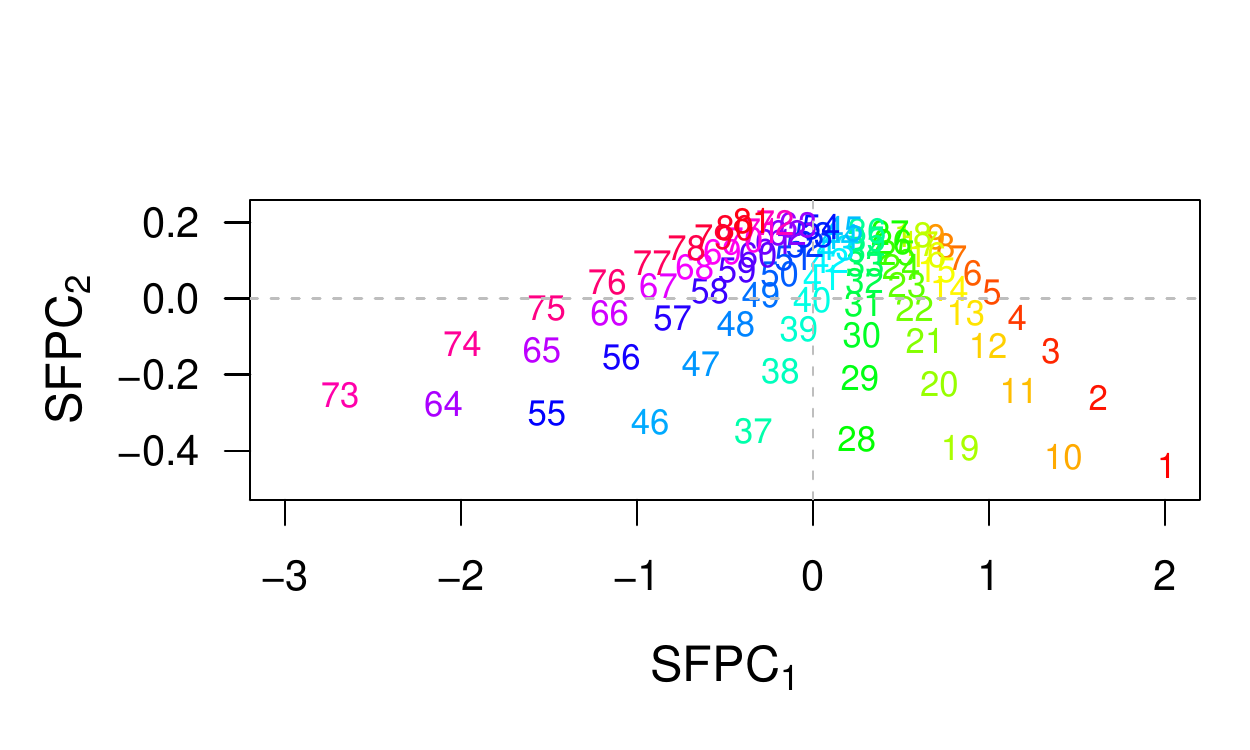}
				\begin{flushleft}
				\caption{Scores for SFPC$_1$ and SFPC$_2$.}
				\label{fig:sfpca-lognor-unif-scor}
				\end{flushleft}
	\end{subfigure}
    \centering
	\begin{subfigure}[t]{1\textwidth}
        \centering
        \includegraphics[width=0.4\textwidth]{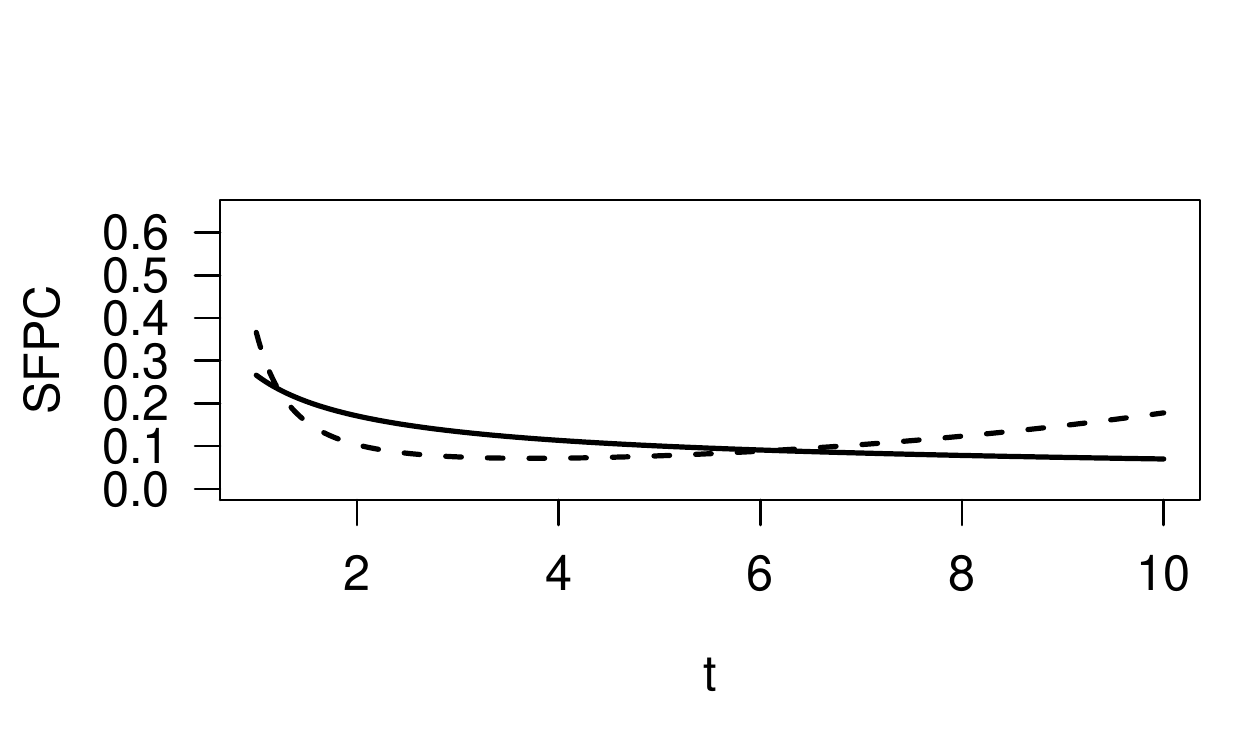}\includegraphics[width=0.4\textwidth]{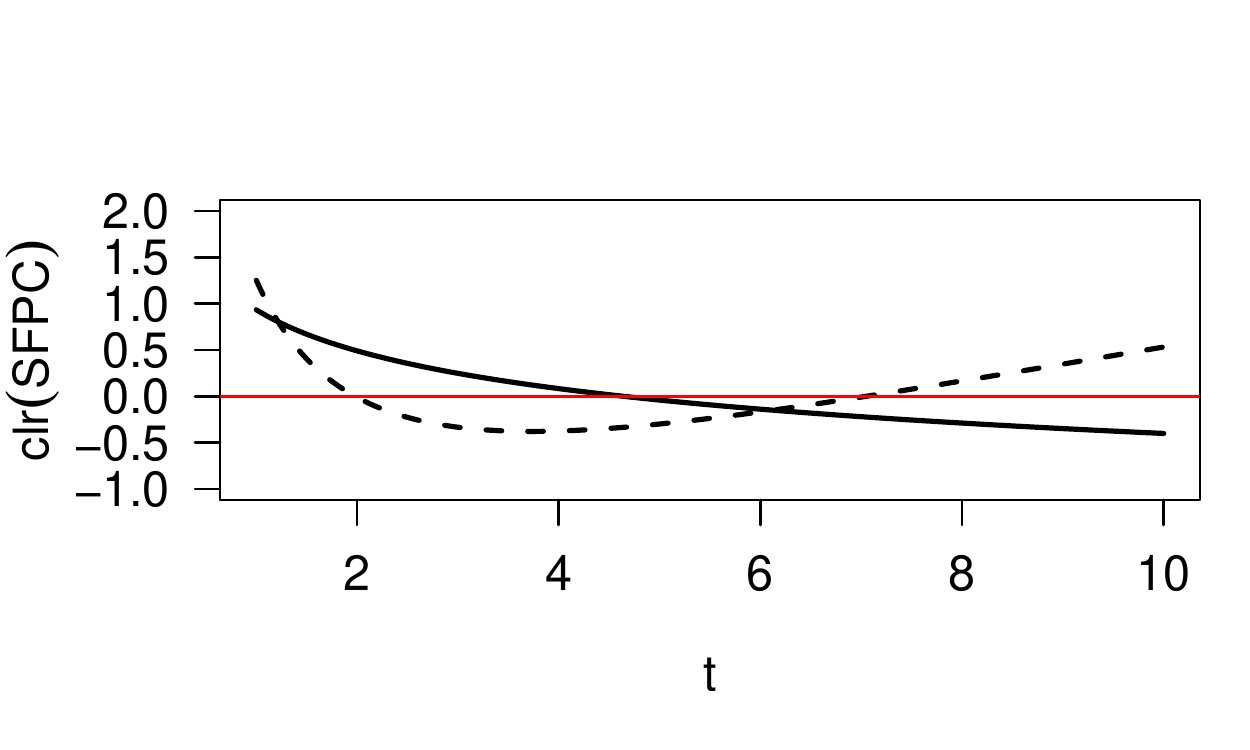}
        \caption{$\mathcal{B}^2$-unweighted SFPC$_1$ (solid line; 96.08$\%$) and SFPC$_2$ (dashed line; 3.92$\%$) (left) and their $\clr_{u}$ transformation (right).
				}
	\end{subfigure}
    \begin{subfigure}[t]{1\textwidth}
        \centering
        \includegraphics[width=0.4\textwidth]{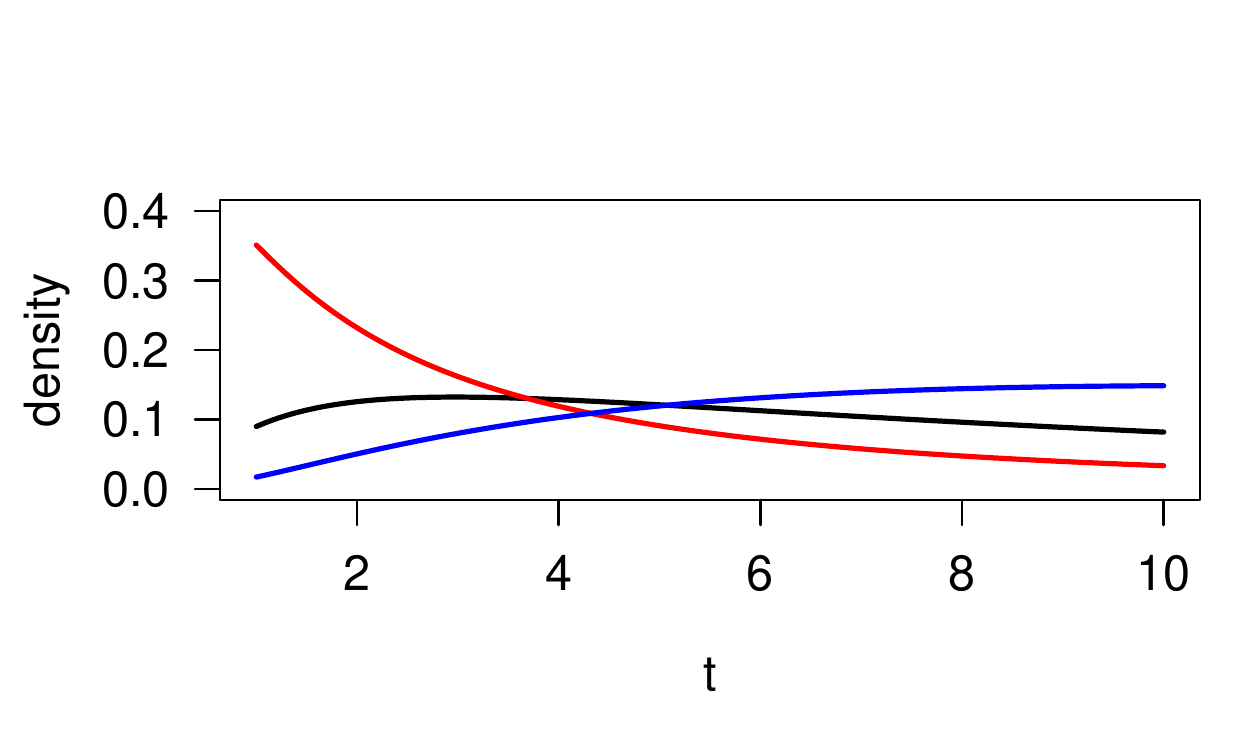}\includegraphics[width=0.4\textwidth]{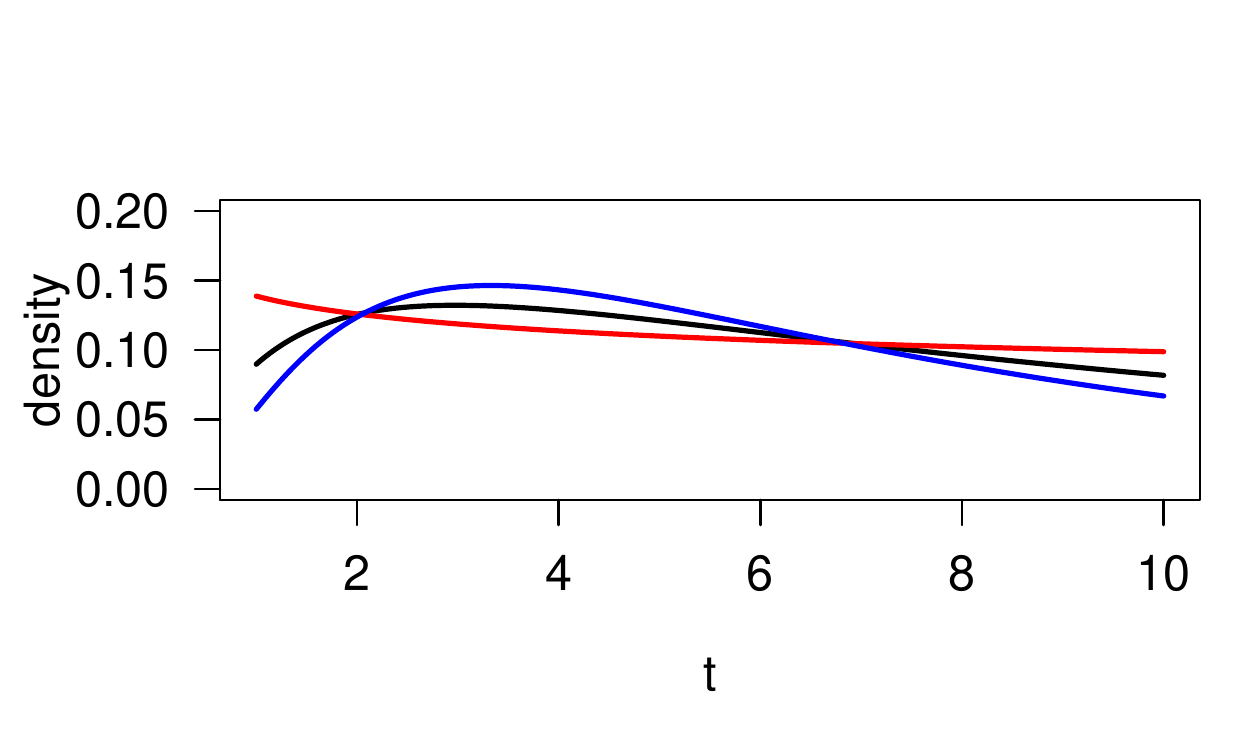}
        \caption{$\mathcal{B}^2$-unweighted version of $\bar{f}_{\mathsf{P}_0} \oplus_{\mathsf{P}_0} / \ominus_{\mathsf{P}_0} 2\sqrt{\rho_1} \odot_{\mathsf{P}_0} \text{SFPC}_{1}$ (left) and of $\bar{f}_{\mathsf{P}_0} \oplus_{\mathsf{P}_0} / \ominus_{\mathsf{P}_0} 2\sqrt{\rho_2} \odot_{\mathsf{P}_0} \text{SFPC}_{2}$ (right).
				}\label{fig:sfpc-log-norm-unif-harm12}
	\end{subfigure}
	\begin{subfigure}[t]{0.4\textwidth}
                \centering
        \includegraphics[width=\textwidth]{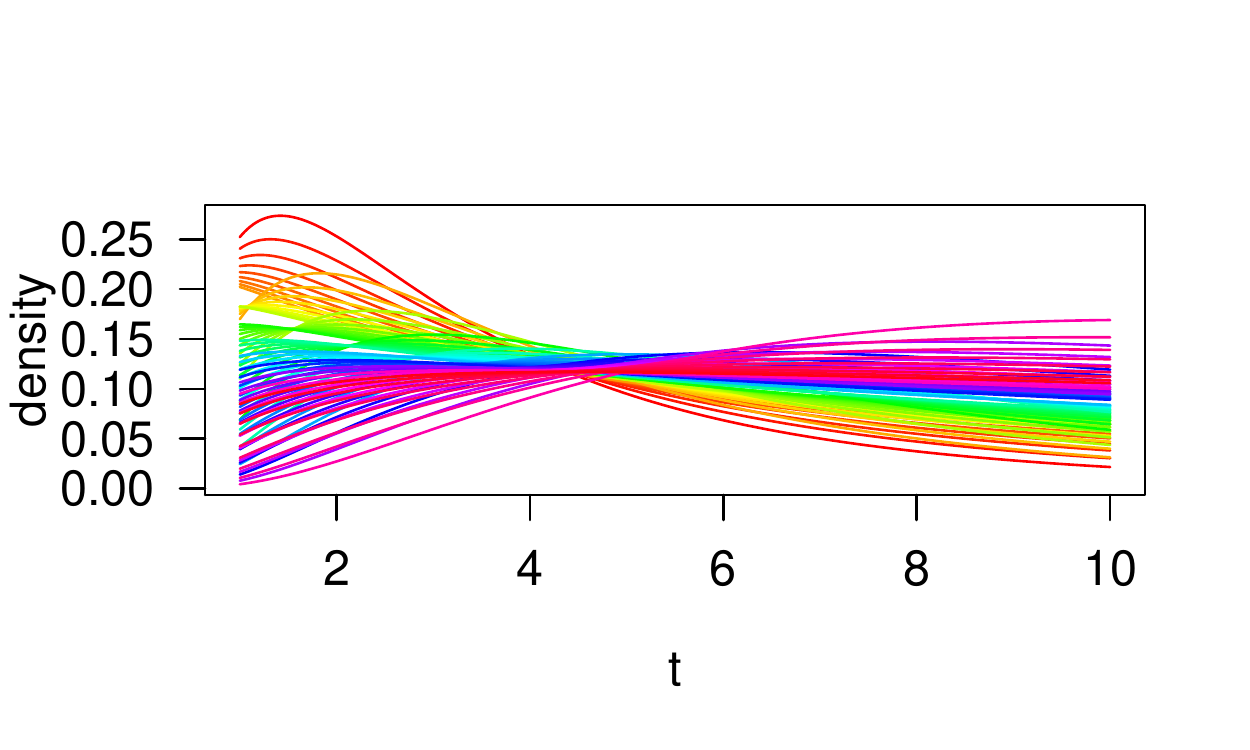} 
				\caption{$\mathcal{B}^2$-unweighted version of the $\mathsf{P}_0$-densities $f_{{\mathsf P}_0,ij}$.}
				\label{proj_unif_12-original}
		\end{subfigure}
		\begin{subfigure}[t]{0.4\textwidth}
                \centering
                \includegraphics[width=\textwidth]{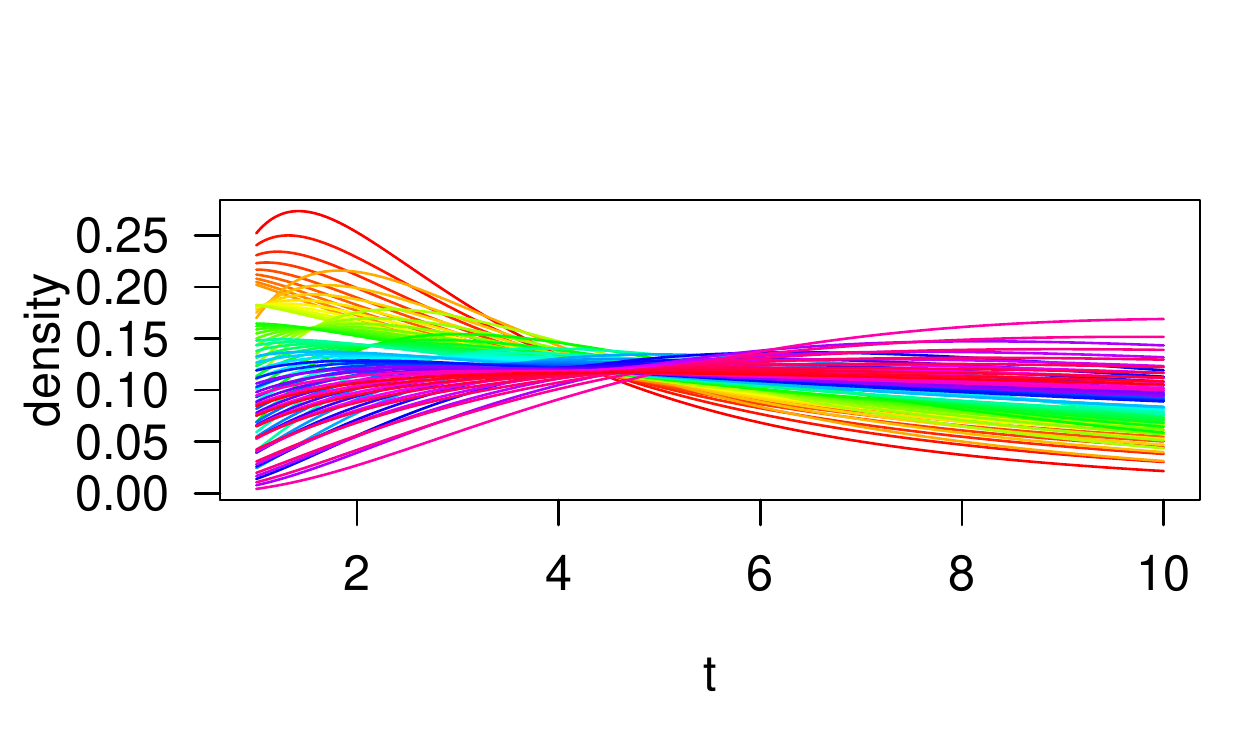}
				\caption{$\mathcal{B}^2$-unweighted version of the approximation of $f_{{\mathsf P}_0,ij}$ via SFPC$_1$ and SFPC$_2$.}
				\label{proj_unif_12}
		\end{subfigure}
    \caption{Results of SFPCA for simulated log-normal densities in the case of a uniform reference measure $\mathsf{P}_0$. Results in panels (c) to (e) are represented in the unweighted spaces $L^2_{0,\sqrt{\mathsf{P}_0}}(\lambda)$, $B^2(\lambda)$. By $\mathcal{B}^2$-unweighted version of $f\in \mathcal{B}^2({\mathsf P}_0)$ it is meant $\omega^{-1}(f_{{\mathsf P}_0}) \in \mathcal{B}^2(\lambda)$. }
		\label{fig:sfpca-lognor-unif}
\end{figure}
\begin{figure}
    \centering
	\begin{subfigure}[t]{1\textwidth}
        \centering
        \includegraphics[width=0.33\textwidth]{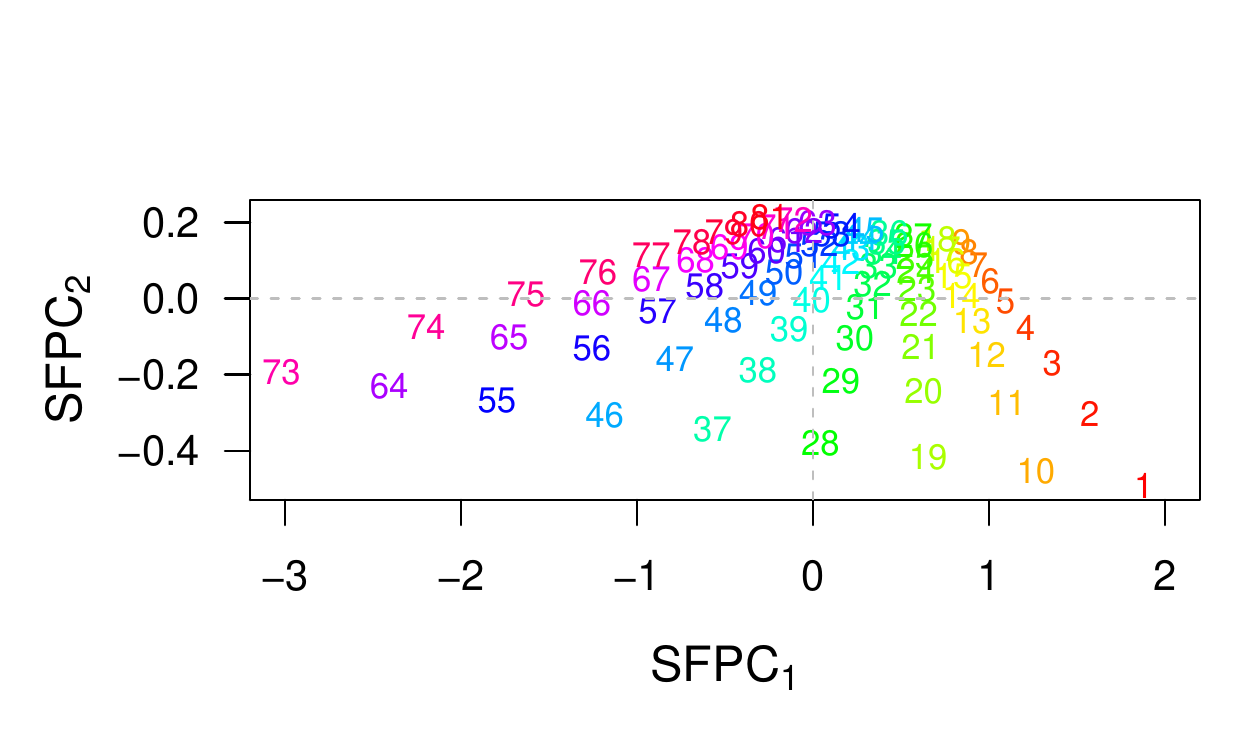}\includegraphics[width=0.33\textwidth]{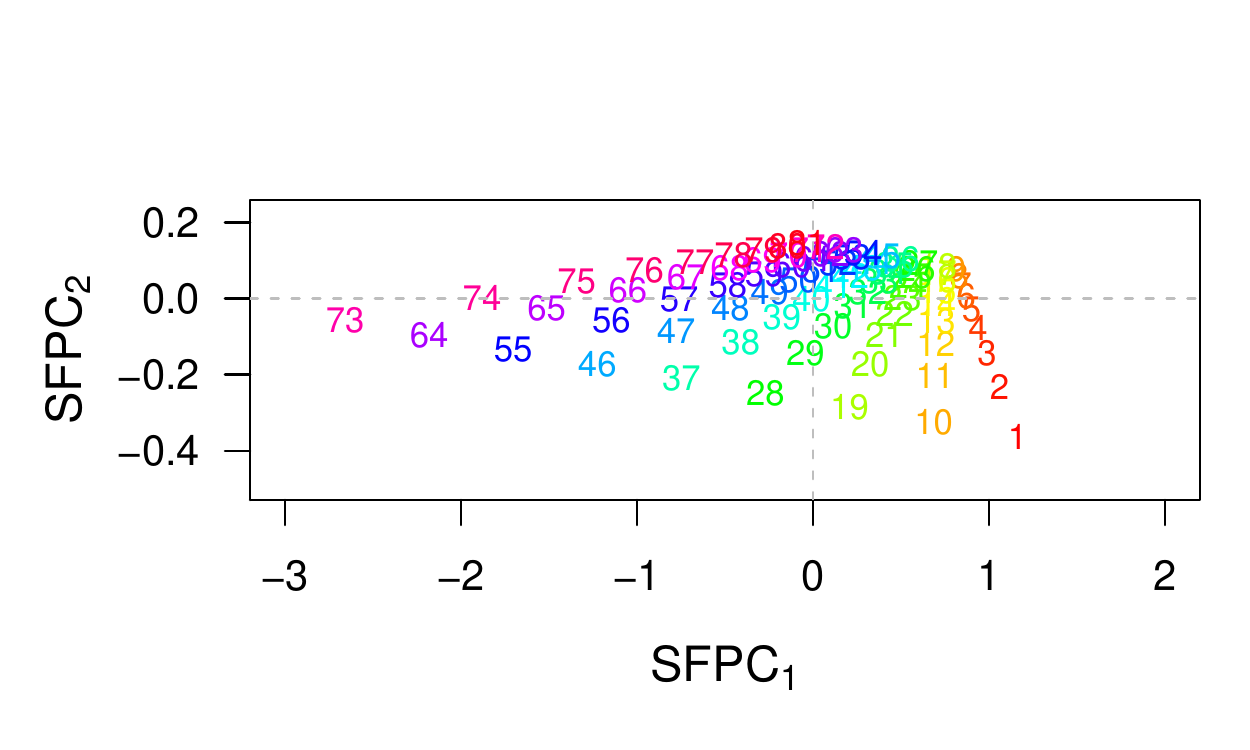}\includegraphics[width=0.33\textwidth]{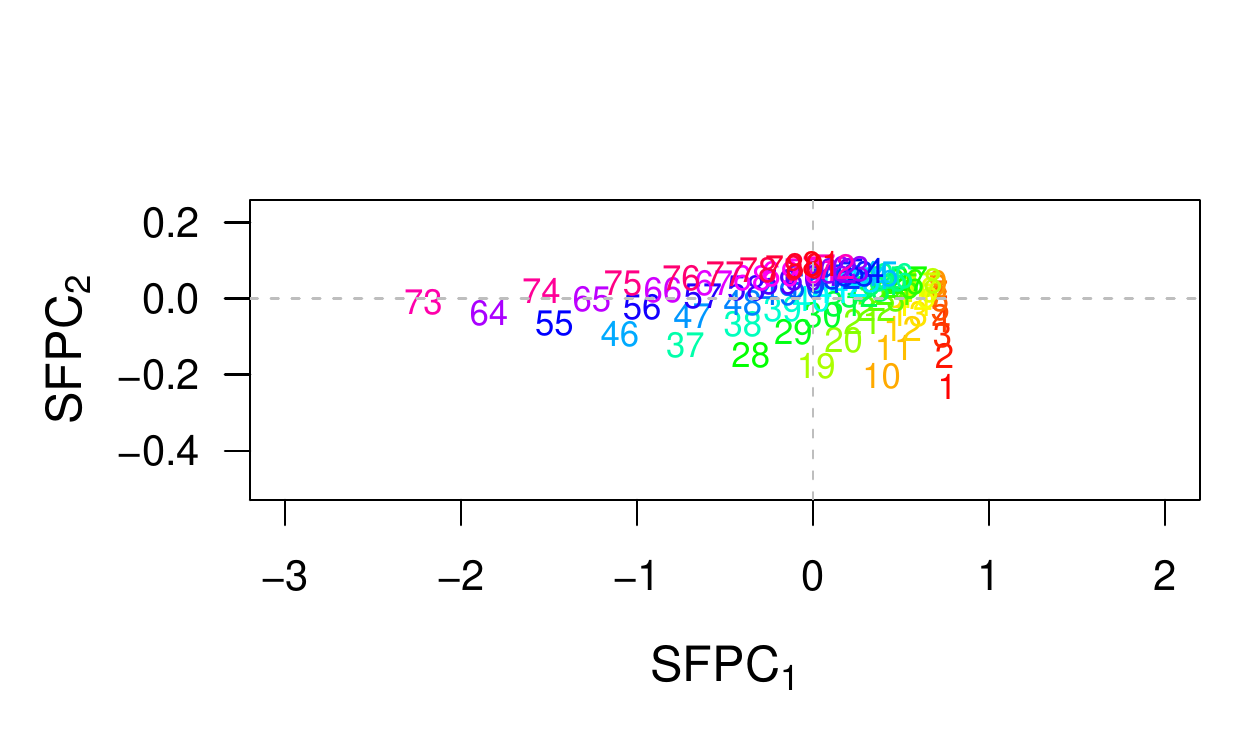}
 	\caption{Scores for SFPC$_1$ and SFPC$_2$.}
    \label{fig:sfpca-lognor-exp123-scor}
	\end{subfigure}
    \centering
    \begin{subfigure}[t]{1\textwidth}
        \centering
        \includegraphics[width=0.33\textwidth]{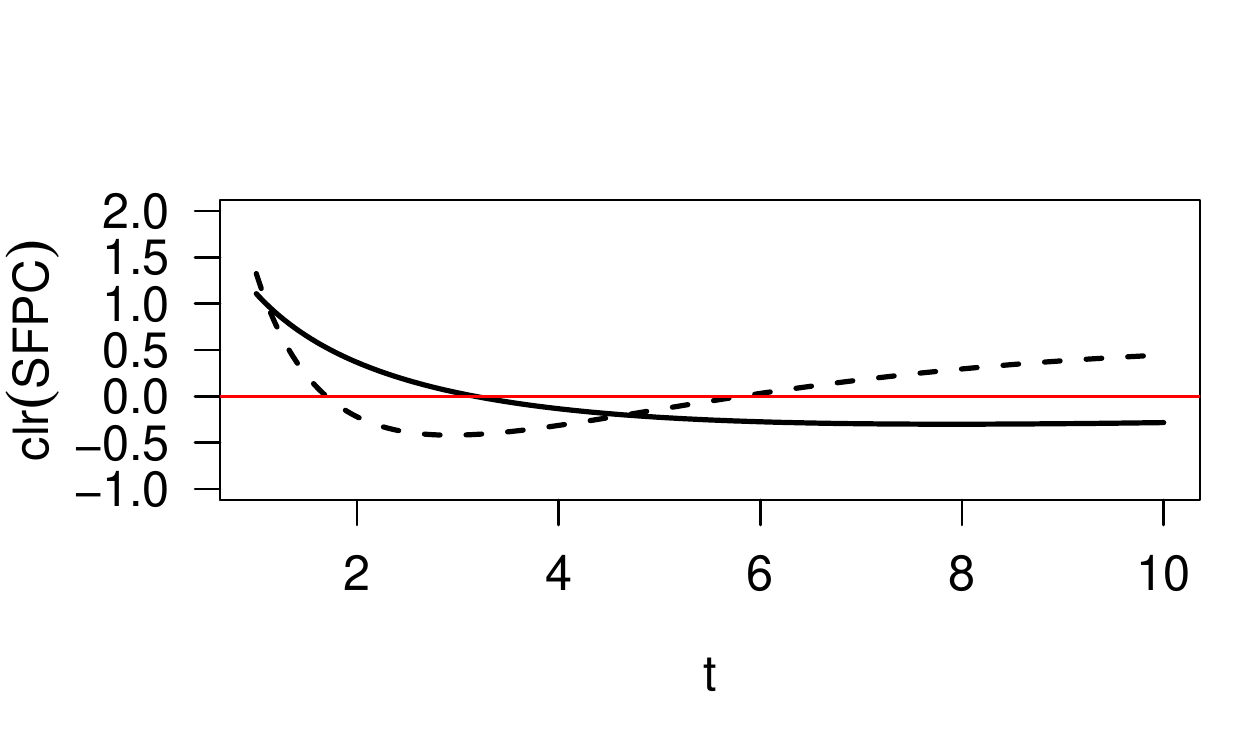}\includegraphics[width=0.33\textwidth]{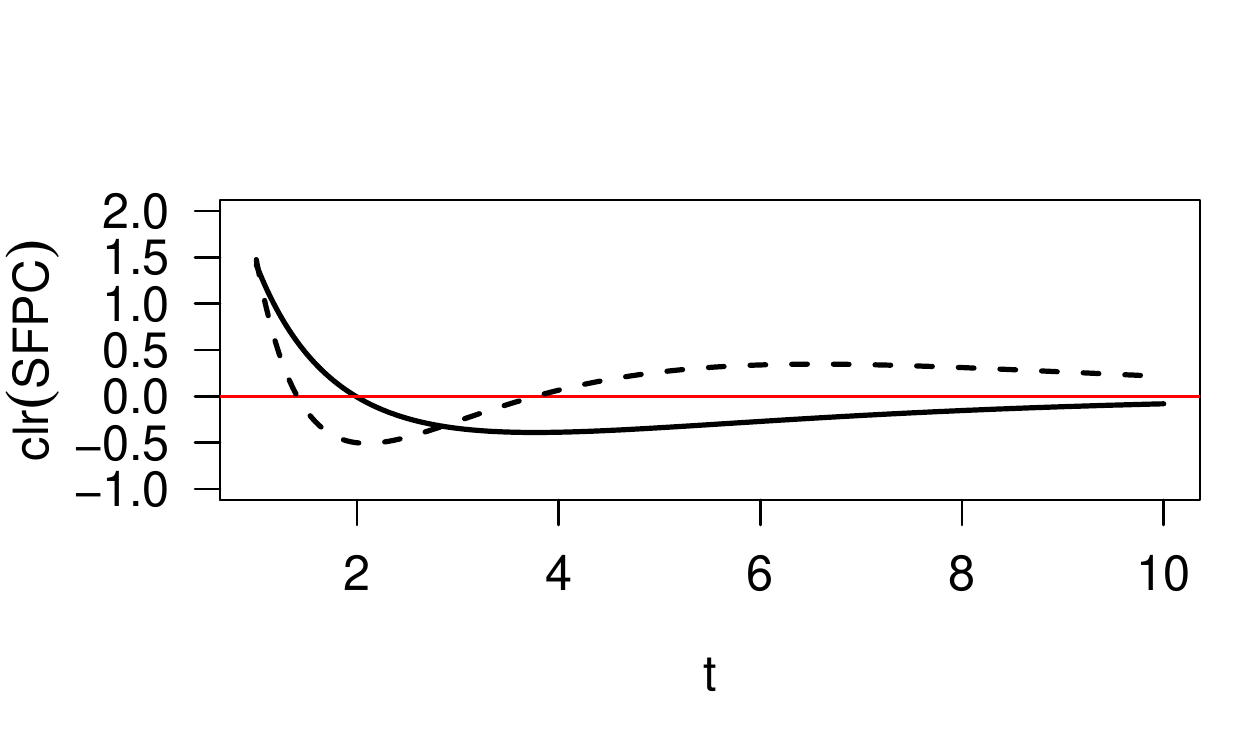}\includegraphics[width=0.33\textwidth]{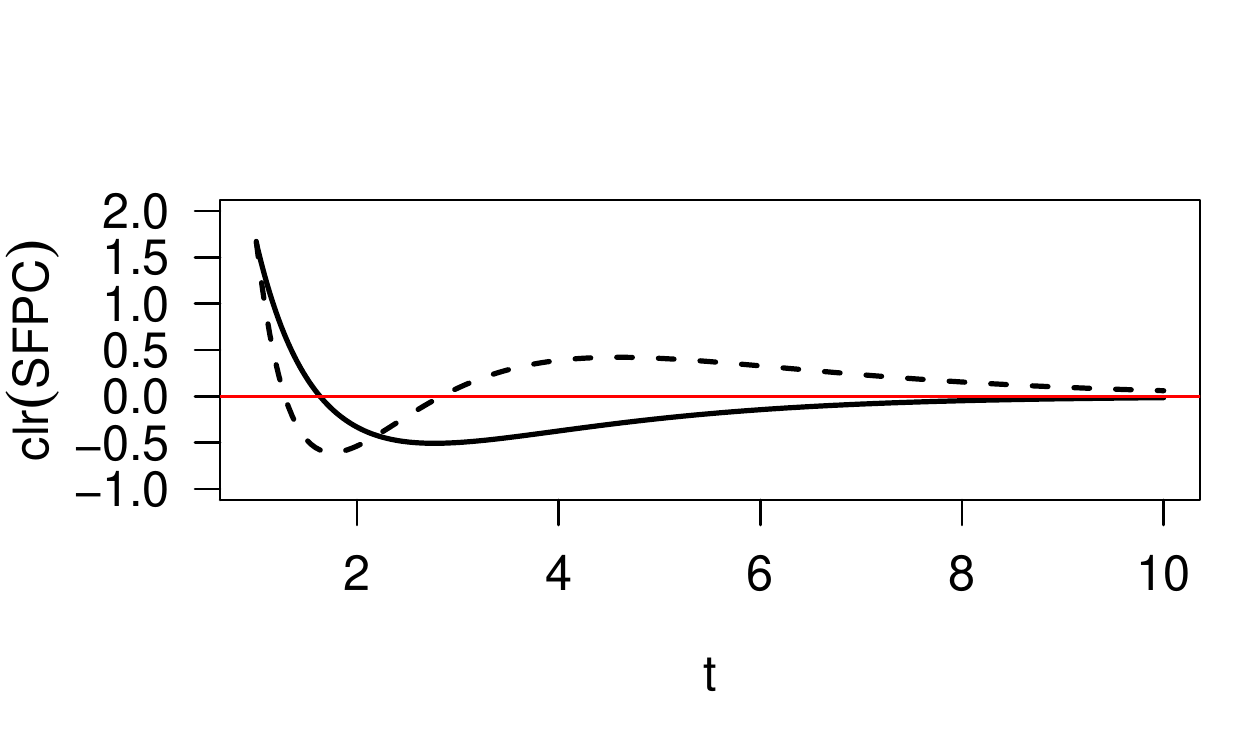}
        \caption{$\clr_u$ transform of the wSFPC$_1$ (solid line; explained variability: 96.48$\%$, 97.80$\%$, 98.76$\%$) and wSFPC$_2$ (dashed line; explained variability: 3.52$\%$, 2.20$\%$, 1.24$\%$).}
	\end{subfigure}
	\begin{subfigure}[t]{1\textwidth}
        \centering
        \includegraphics[width=0.33\textwidth]{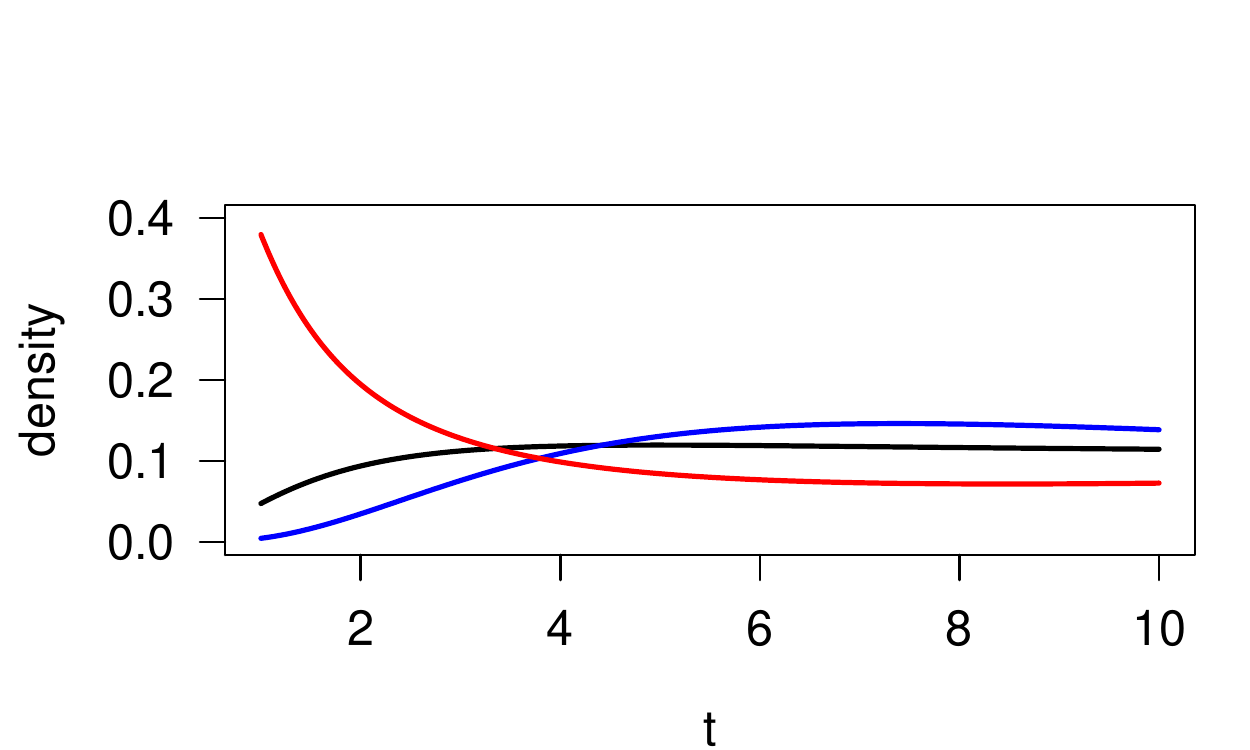}\includegraphics[width=0.33\textwidth]{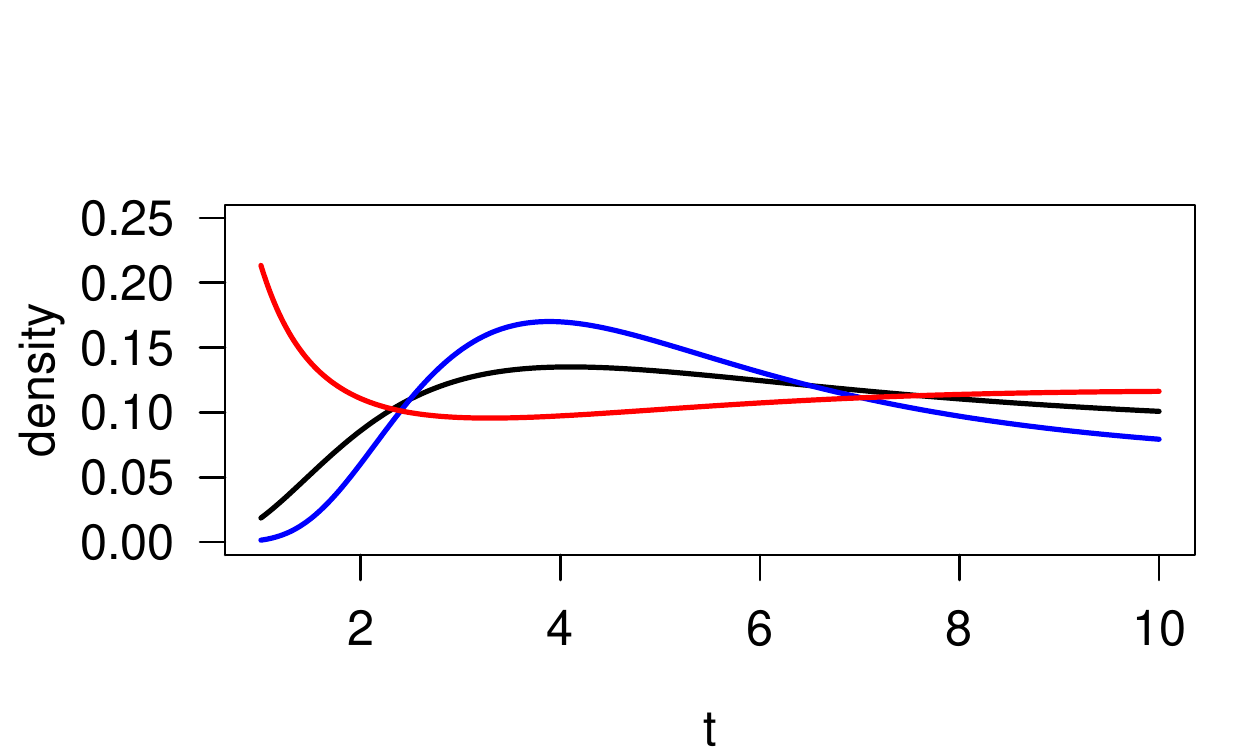}\includegraphics[width=0.33\textwidth]{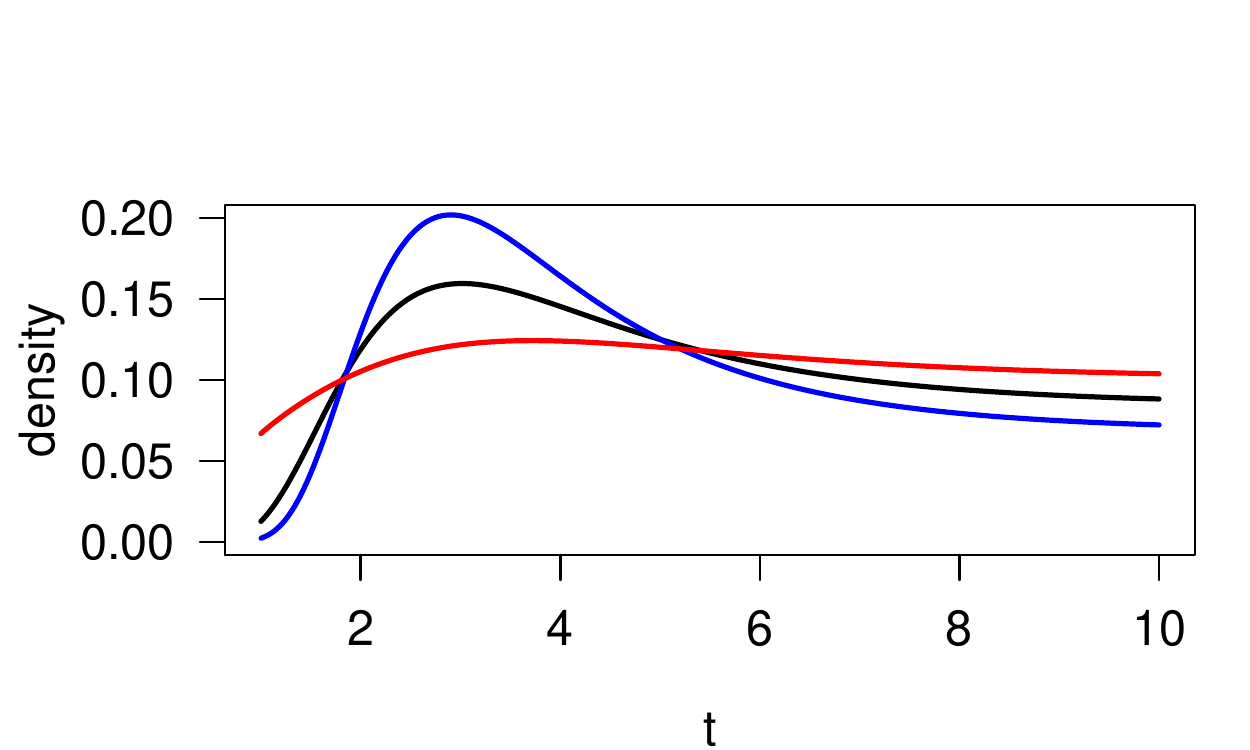}
        \caption{$\mathcal{B}^2$-unweighted version of $\bar{f}_{\mathsf{P}} \oplus_{\mathsf{P}} / \ominus_{\mathsf{P}} 2\sqrt{\rho_1} \odot_{\mathsf{P}} \text{wSFPC}_{1}$ in {$\mathcal{B}^2(\lambda)$}, with $\mathsf{P}=\mathsf{P}^\delta$.}\label{fig:sfpc-log-norm-exp123-harm1}
	\end{subfigure}
	\begin{subfigure}[t]{1\textwidth}
        \centering
        \includegraphics[width=0.33\textwidth]{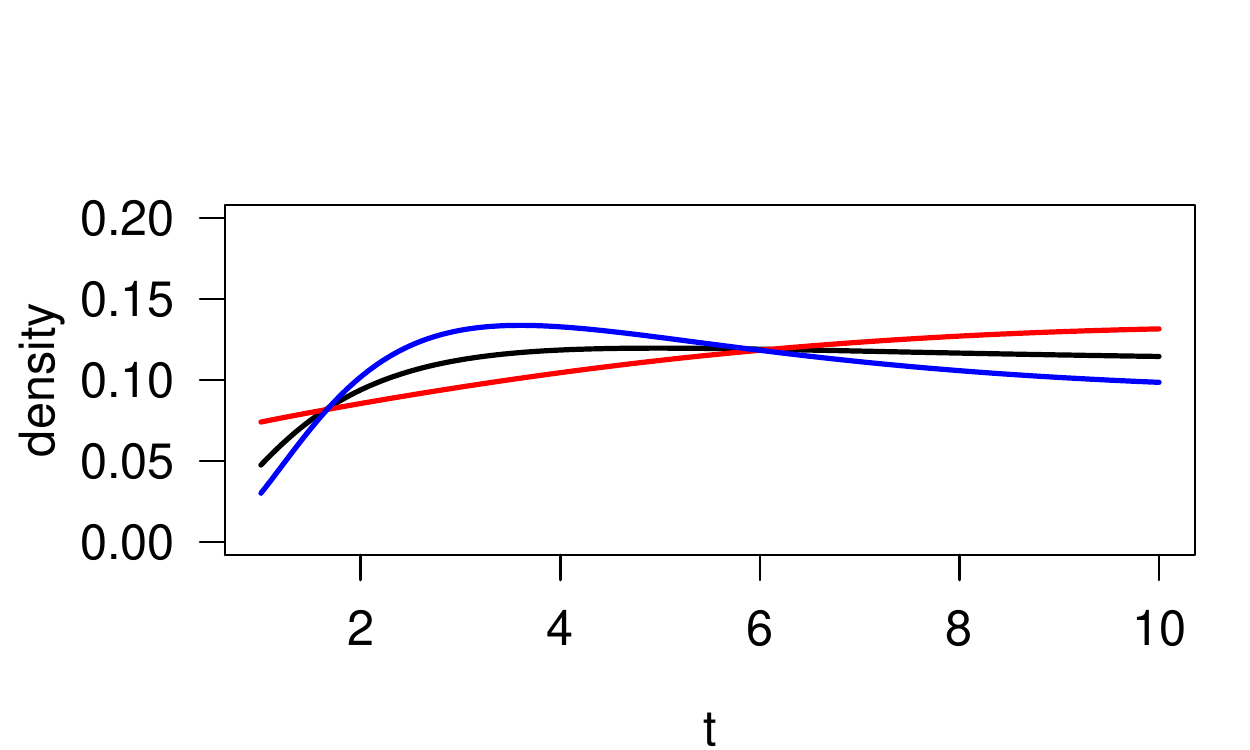}\includegraphics[width=0.33\textwidth]{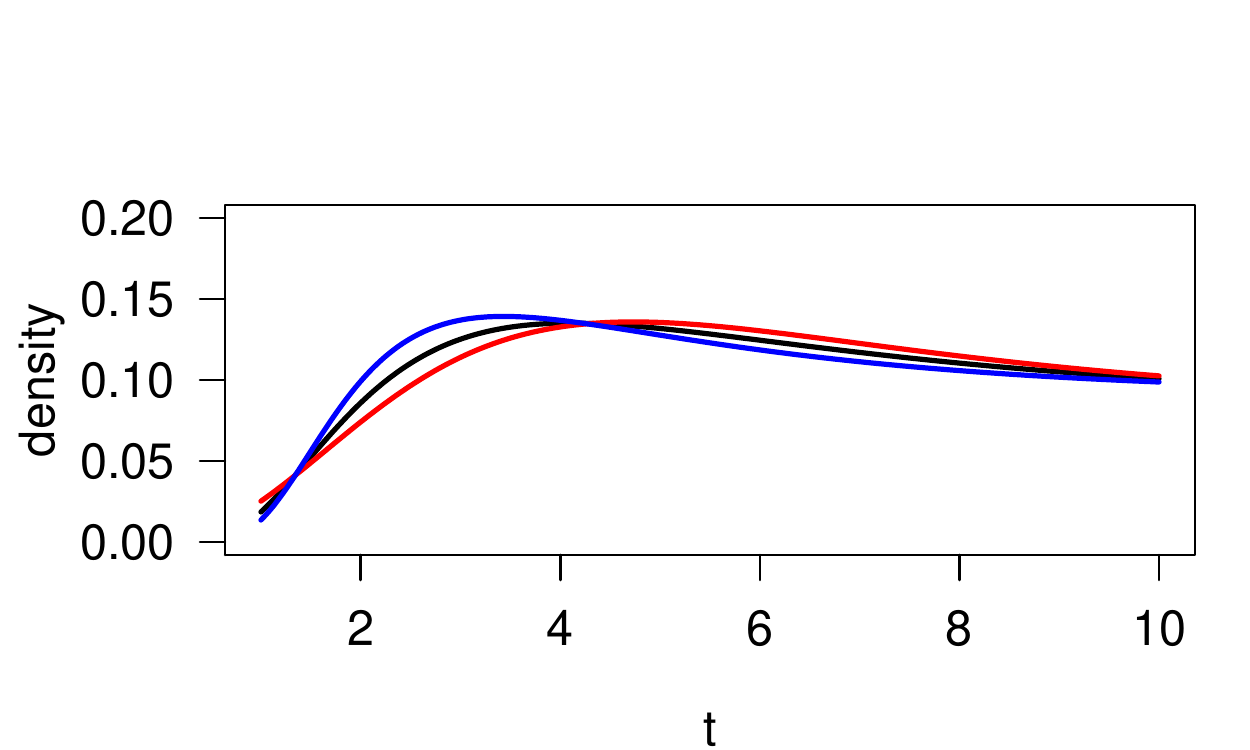}\includegraphics[width=0.33\textwidth]{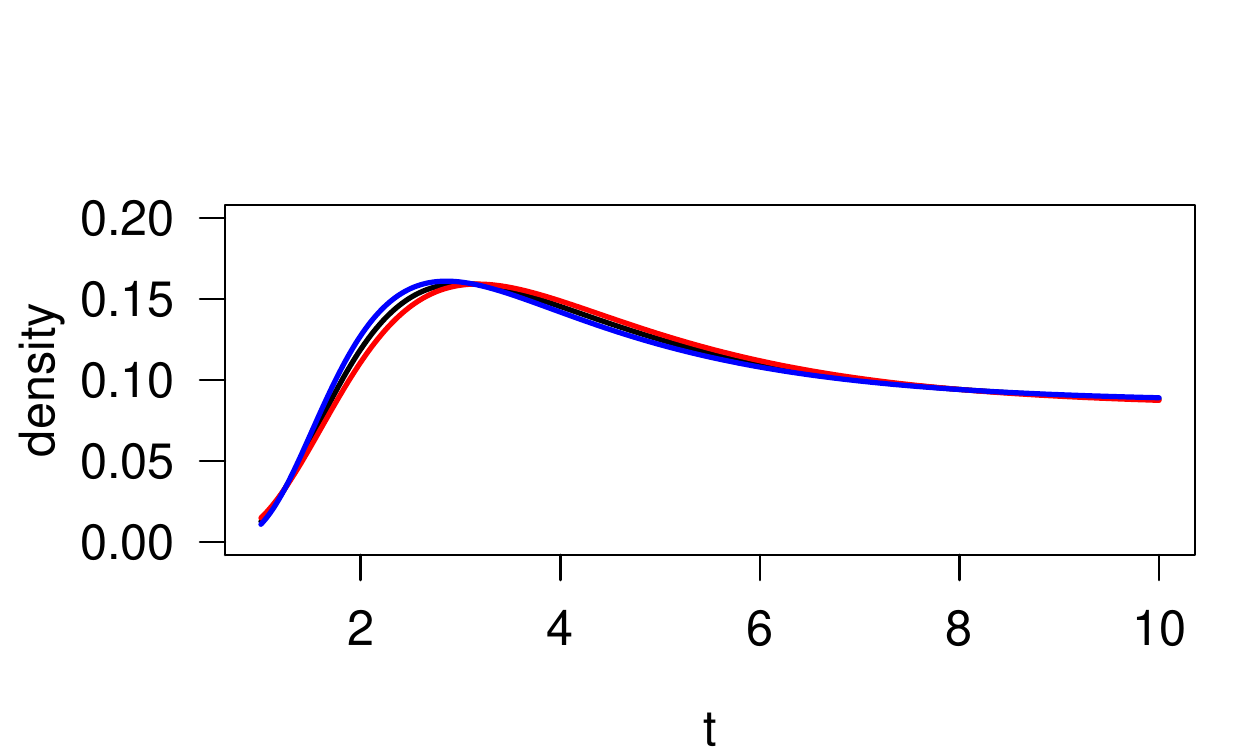}
        \caption{$\mathcal{B}^2$-unweighted version of $\bar{f}_{\mathsf{P}} \oplus_{\mathsf{P}} / \ominus_{\mathsf{P}} 2\sqrt{\rho_2} \odot_{\mathsf{P}} \text{wSFPC}_{2}$ in {$\mathcal{B}^2(\lambda)$}, with $\mathsf{P}=\mathsf{P}^\delta$.}\label{fig:sfpc-log-norm-exp123-harm2}
	\end{subfigure}
	\begin{subfigure}[t]{1\textwidth}
        \centering
        \includegraphics[width=0.32\textwidth]{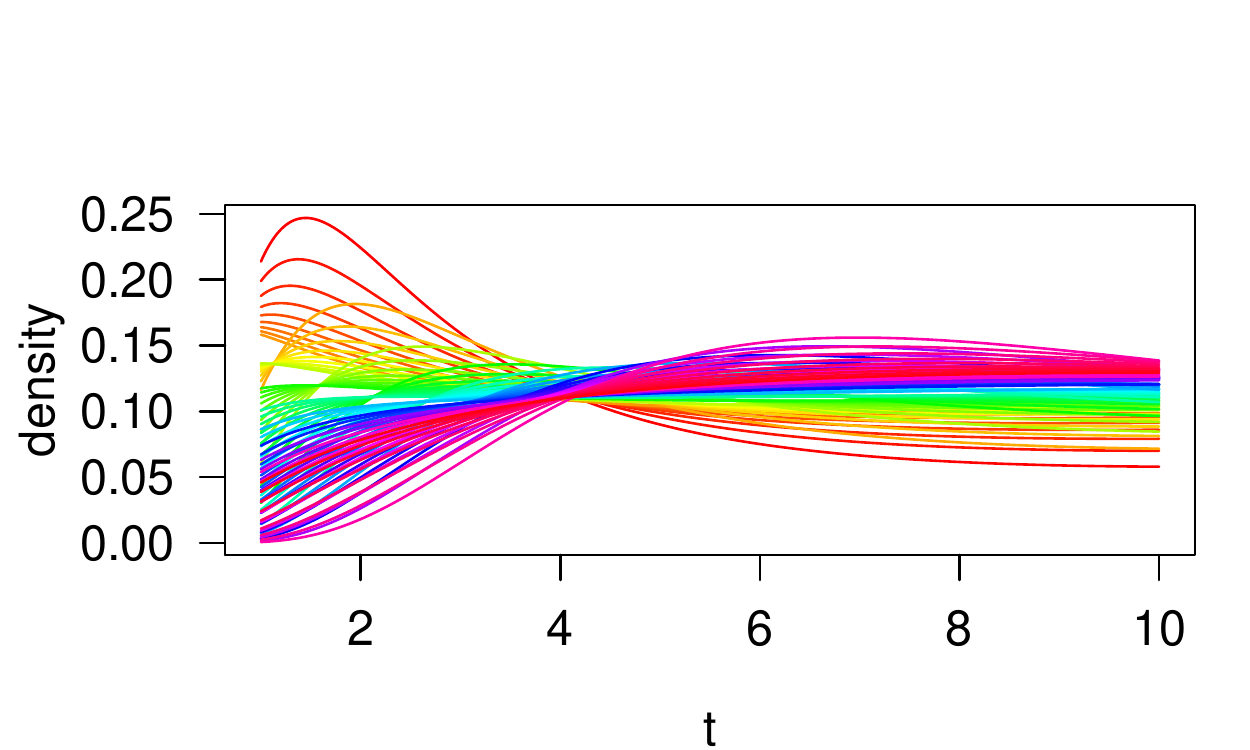} 
        \includegraphics[width=0.32\textwidth]{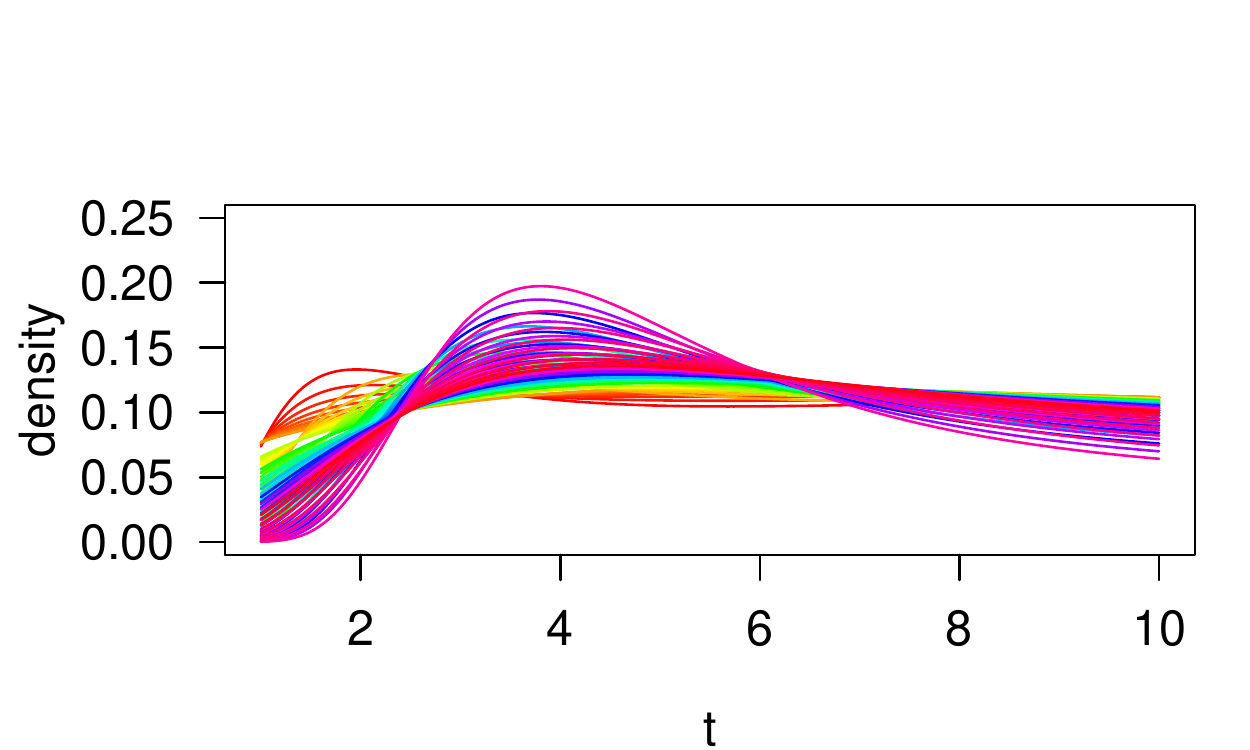} 
        \includegraphics[width=0.32\textwidth]{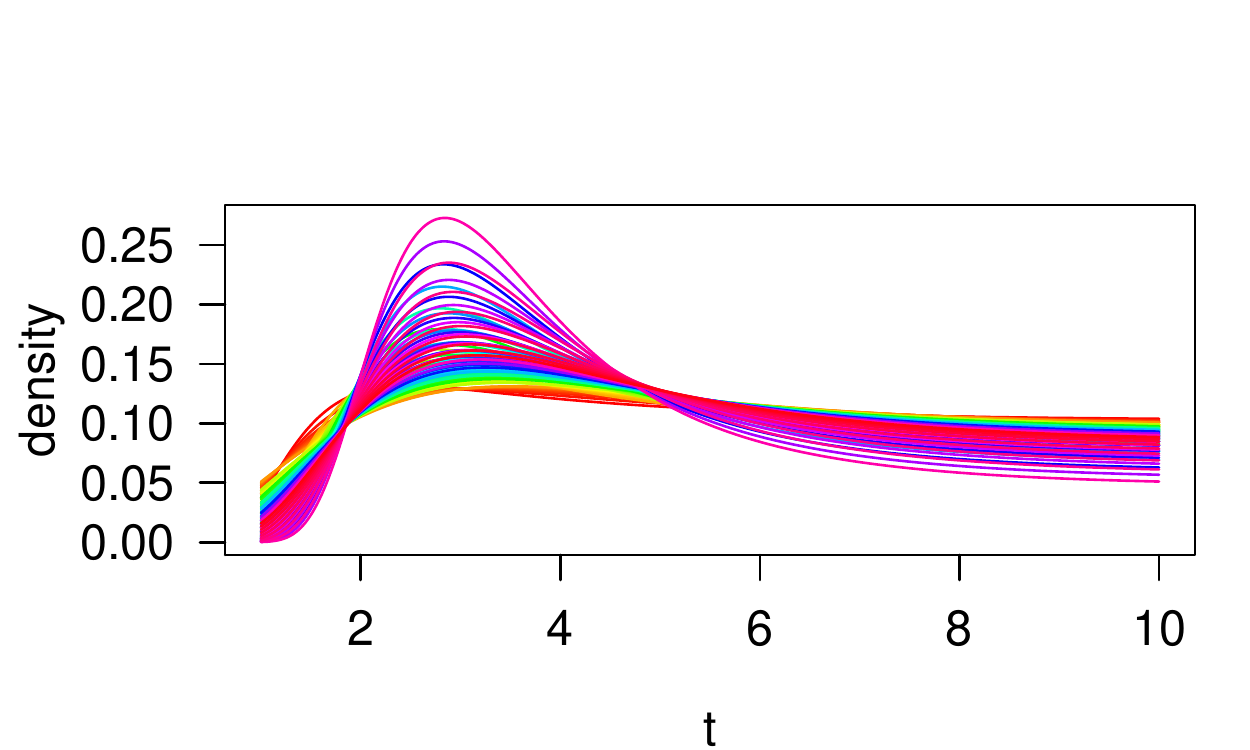} 
        \caption{$\mathcal{B}^2$-unweighted version of the approximation of the $\mathsf{P}$-densities via wSFPC$_1$ and wSFPC$_2$.}
				\label{proj-lognorm-exp123}
		\end{subfigure}
    \caption{Results of SFPCA for simulated log-normal densities in case of exponential reference measures with $\delta=0.25$ (first column), $\delta=0.75$ (second column) and $\delta=1.25$ (third column). By $\mathcal{B}^2$-unweighted version of $f\in \mathcal{B}^2({\mathsf P}^\delta)$ it is meant $\omega^{-1}(f_{{\mathsf P}^\delta}) \in \mathcal{B}^2(\lambda)$
}
\label{fig:sfpca-lognor-exp123}
\end{figure}

When placing more emphasis on the left-hand side of the support $I$ by increasing the parameter $\delta$ of the exponential reference measure (Figure \ref{fig:sfpca-lognor-exp123}), the portion of explained variability increases in SFPC$_1$ (and thus decreases in SFPC$_2$). Regardless of the reference measure, the first clr-wSFPC suggests that the main contribution to the total variability is associated with a contrast between the left-hand side of the domain and the other side. It should be pointed out that clr-transformed densities always display a contrast due to the zero integral constraint. However, it is worth noticing that the zero-crossing point moves to the left when the reference measure is changed using higher values of $\delta$. The same pattern is observed for the second clr-wSFPC since it still highlights the variability in the left-hand side of the domain, but additionally it presents a contrast between the central and the rest. These conclusions are further supported by Figures \ref{fig:sfpc-log-norm-unif-harm12}, \ref{fig:sfpc-log-norm-exp123-harm1} and \ref{fig:sfpc-log-norm-exp123-harm2}, where, for $\mathsf{P}=\mathsf{P}_0$ and $\mathsf{P}=\mathsf{P}^\delta$ respectively, the mean density is perturbed $(\oplus_\mathsf{P} / \ominus_\mathsf{P})$ by the SFPC powered $(\odot_\mathsf{P})$ to twice the standard deviation $\sqrt{\rho}$ along the corresponding direction $\xi_{\mathsf{P}}$ (i.e. $\bar{f}_\mathsf{P} \oplus_\mathsf{P} / \ominus_\mathsf{P} (2\sqrt{\rho_j} \odot_\mathsf{P}\xi_{\mathsf{P},j})$, where the $(\rho_j, \xi_{\mathsf{P},j})$ is the $j$th eigenpair of the covariance operator $V$). 
These results suggest that, when a uniform reference $\mathsf{P}_0$ or an exponential $\mathsf{P}^\delta$ with $\delta=0.25$ are considered, the main mode of variability resides in the left-hand side of the domain. Changing the reference measure to $\mathsf{P}^\delta$ has the effect of inflating the variability of the data in the central-left section of the domain (around the interval $[2,4]$, see also Figure \ref{fig:log-norm-exp123-orig}), with a direct effect on the variability displayed along the first wSFPC.

Figures \ref{fig:sfpca-lognor-unif-scor} and \ref{fig:sfpca-lognor-exp123-scor} display the score plots of wSFPCA under a $\mathsf{P}_0$ and $\mathsf{P}^{\delta}$ respectively. The symbols represent the indices of the data points, with $f_{\mathsf{P}, ij}$ being represented through the index $\kappa = j+9(i-1)$, $i,j = 1,...,9$. Recalling that the sampling design considers $\mu_i = 0.6+0.25\cdot(i-1)$ and $\sigma_j = 0.5+0.07\cdot(j-1)$ for $i, j=1, \ldots,9$; note that SFPC$_1$ arranges the densities according to parameter $\mu_i$ whereas SFPC$_2$ according to parameter $\sigma_j$.

Finally, Figures \ref{proj_unif_12} and \ref{proj-lognorm-exp123} display the projection of the log-normal densities on the basis generated by the first two wSFPCs, each represented in the unweighted $\mathcal{B}^2({\lambda})$ space (i.e., after $\clr_u$ transformation). These results confirm that the dimensionality of the affine spaces of $\mathcal{B}^2({\mathsf{P}})$, for ${\mathsf P}={\mathsf P}_0$ and  ${\mathsf P}={\mathsf P}^\delta$, spanned by the log-normal family is indeed captured by the first two wSFPCs.

Consistent results are obtained from a second data set consisting of Weibull densities. For the sake of brevity, they are not included here and briefly shown in the supplementary material.

\section{A real-world example: weighted SFPCA of Italian income data}\label{sec:casestudy}
As an illustrative example, we apply wSFPCA to income data from the \emph{Survey on Household Income and Wealth} (SHIW) conducted by the Italian Central Bank. They include almost 8000 interviewed households composed of 19907 individuals and 13266 income-earners and freely available on the internet \cite{shiw}. We focus on annual net disposable income 
of households in all the 20 Italian regions. These were further grouped into three natural areas according to their geographical location to examine possible differences between regions (see Figure \ref{fig:map-it-income}).
\begin{figure}
    \centering
    \includegraphics[width=0.30\textwidth]{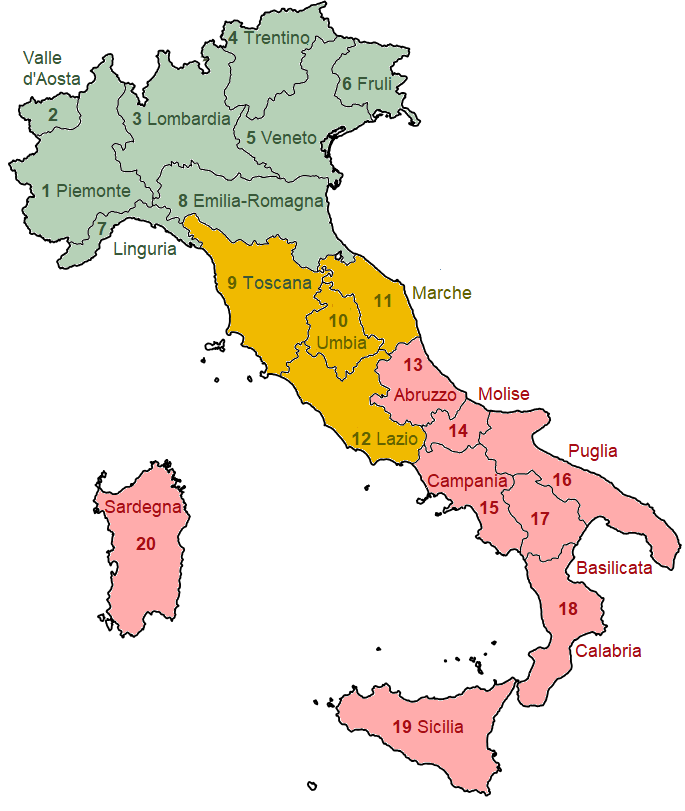}
	\qquad
	\includegraphics[width=0.33\textwidth]{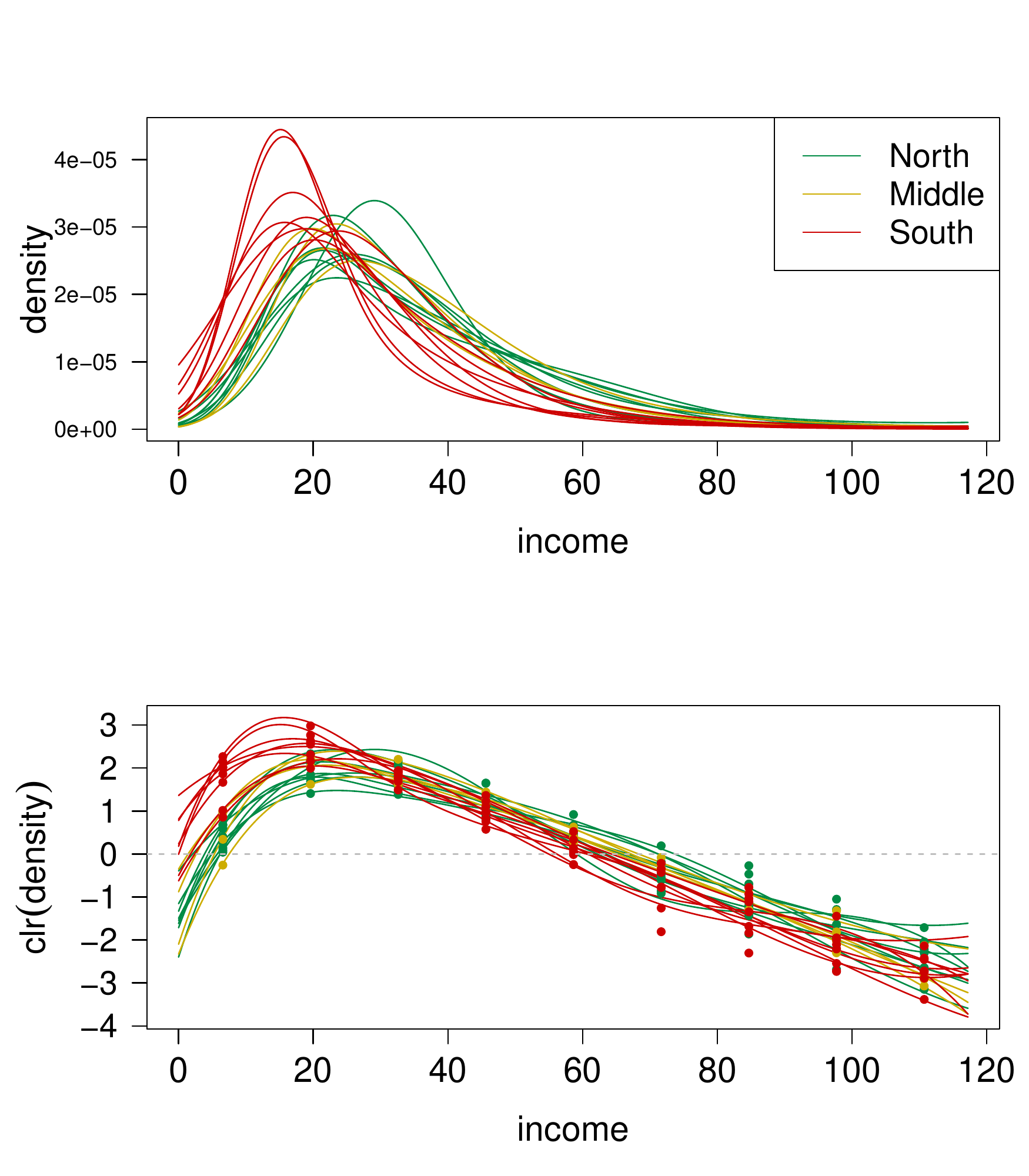}
	\caption{Italian SHIW income data. Left: map of Italy and its 20 regions with color distinguishing northern (green), middle (gold) and southern and island (red) regions according to the National statistical institute (ISTAT). Right: smoothed income densities in Bayes $\mathcal{B}^2(\lambda)$  and clr spaces (using Lebesgue reference measure); income is expressed in $10^3$ {k\euro}.
}
\label{fig:map-it-income}
\end{figure}

The raw income data from individual regions were aggregated into histogram data -- proportions of 9 equidistant income classes determined by Sturges rule -- for non-zero incomes up to 117.22 {k\euro}. Only incomes below the 99\%-quantile were used and extreme values were excluded. Following \cite{Mach}, a discrete version of the clr transformation was applied and the results were smoothed via a B-spline basis in $L^2_0$ with Lebesgue reference measure. Cubic smoothing splines were employed with support $I=[0,117.22]$ {k\euro} and 4 knots at income values 0, 30, 70 and 117.22 {k\euro} (the parameters were set to obtain a good fit of the raw density data yet avoiding overfitting). In Figure \ref{fig:map-it-income} (right) the resulting density functions (with respect to the Lebesgue measure) are displayed in the Bayes space $\mathcal{B}^2(\lambda)$ as well as in $L^2_0(\lambda)$ after the $\clr_{\lambda}$ transformation. The color scheme matches that used for the geographical map (Figure \ref{fig:map-it-income} left). Visual inspection of Figure \ref{fig:map-it-income} suggests that a regional pattern may be present, as northern regions seem to be associated with higher incomes than southern ones. This probably relates to the fact that a large number of business and industries are based in the north. The cost of living is not homogeneous across the regions either, which may also play a major role in determining the actual salaries.

In the following, we describe the results of wSFPCA when the reference measure is set (i) to the Lebesgue measure, (ii) the exponential measure $\mathsf{P}^\delta$ (Section \ref{sec:simu}), and (iii) the measure $\mathsf{P}^m$ corresponding to the unweighted sample mean of the data as in \cite{boogaart14}. Figure \ref{fig:income-it-exp2-mean} displays the (ii) and (iii) cases, together with the corresponding $\mathcal{B}^2$-unweighted densities ($\omega^{-1}(f_{\mathsf{P}})$).

\begin{figure}
    \centering
	\begin{subfigure}[t]{0.4\textwidth}
        \centering
        \includegraphics[width=\textwidth]{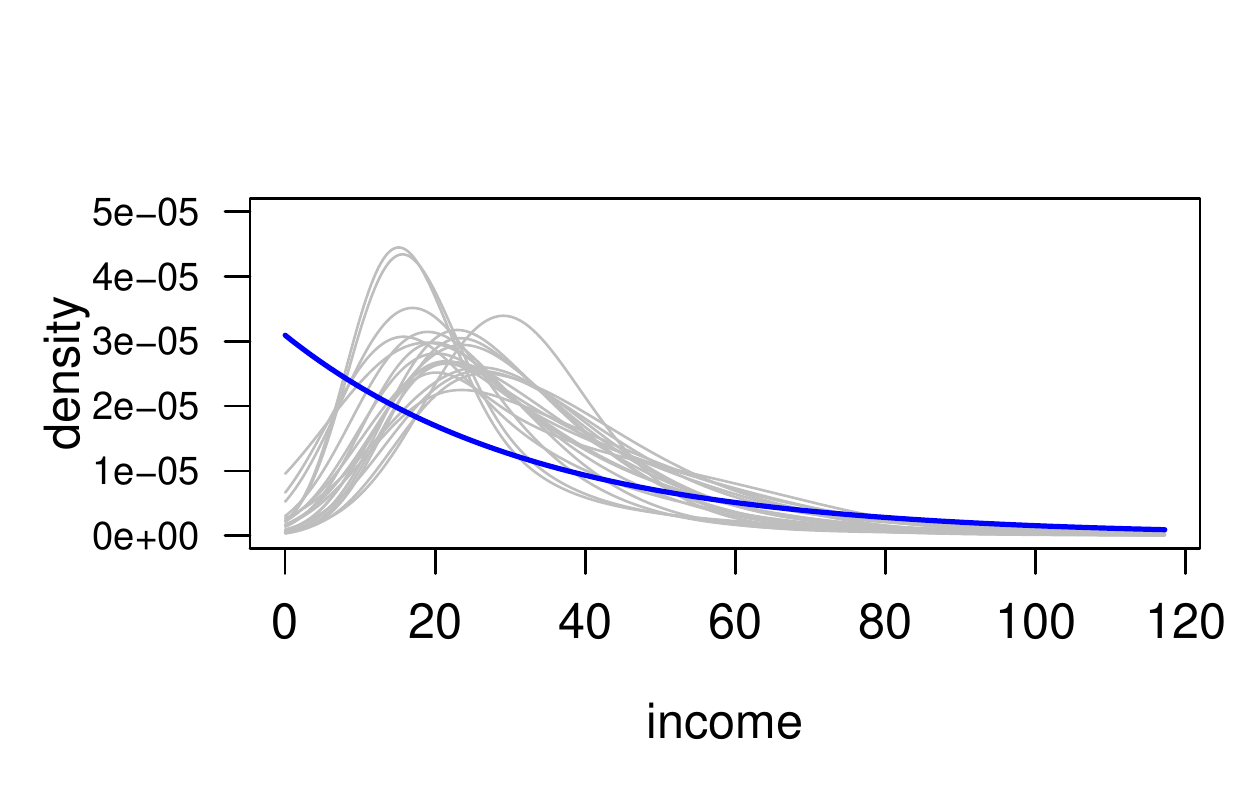}
        \caption{Data (grey lines) and exponential reference density with $\delta=3\times 10^{-5}$ (blue line).}
	\end{subfigure}
	\begin{subfigure}[t]{0.4\textwidth}
        \centering
        \includegraphics[width=\textwidth]{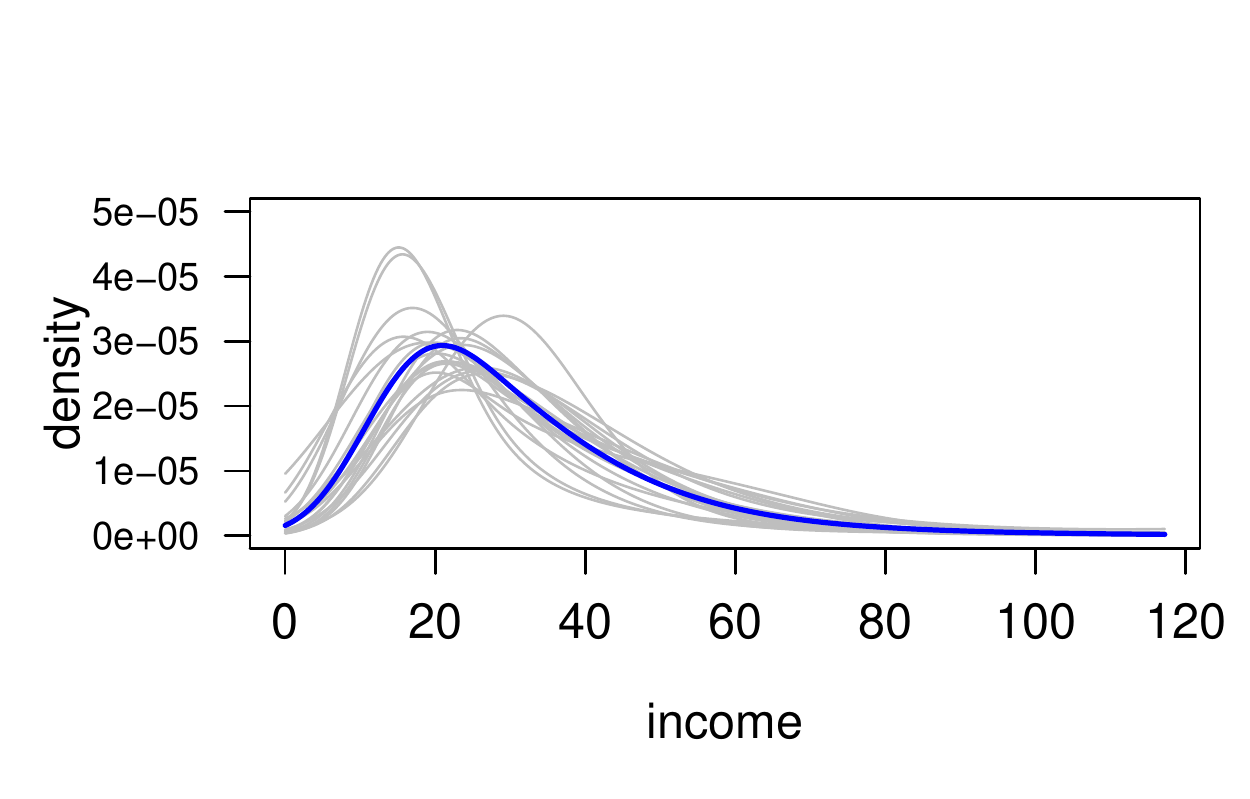}
        \caption{Data (grey lines) and mean reference density (blue line).}
	\end{subfigure}
	\begin{subfigure}[t]{0.4\textwidth}
        \centering
        \includegraphics[width=\textwidth]{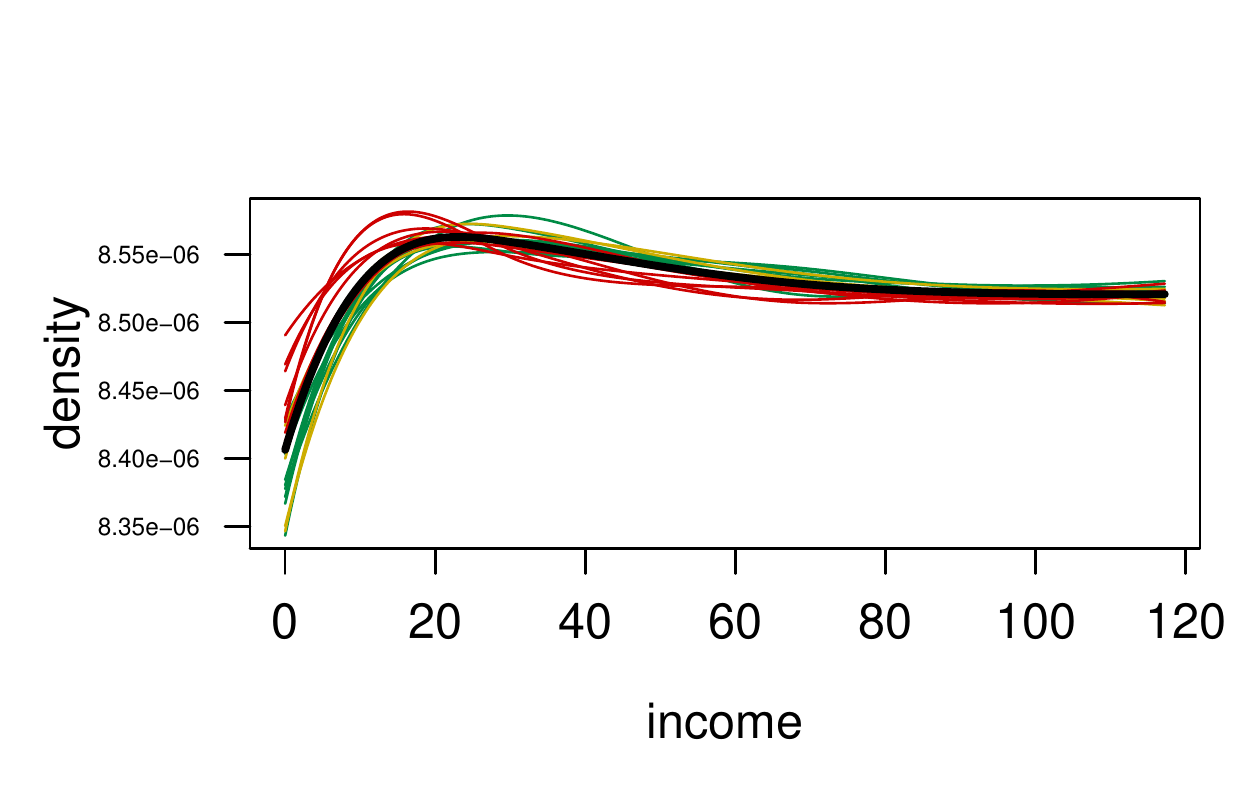}
        \caption{$\mathcal{B}^2$-unweighted version of $\mathsf{P}$-densities when $\mathsf{P}=\mathsf{P}^\delta$ and their mean function (black line).}
	\end{subfigure}
	\begin{subfigure}[t]{0.4\textwidth}
        \centering
        \includegraphics[width=\textwidth]{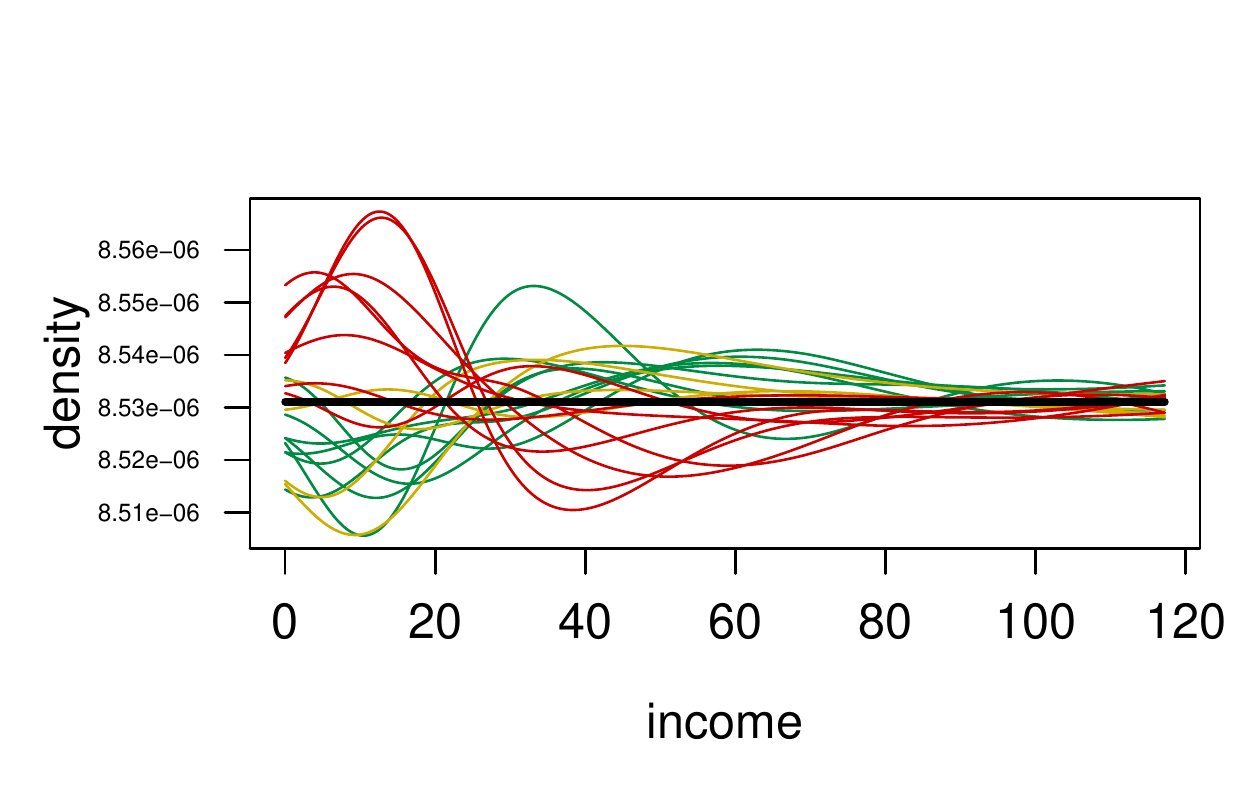}
        \caption{$\mathcal{B}^2$-unweighted version of $\mathsf{P}$-densities when $\mathsf{P}=\mathsf{P}^m$ and their mean function (black line).}
	\end{subfigure}
\caption{Income densities in case of exponential (left column) and mean reference measure (right column). By $\mathcal{B}^2$-unweighted version of $f\in \mathcal{B}^2({\mathsf P})$ it is meant $\omega^{-1}(f_{{\mathsf P}}) \in \mathcal{B}^2(\lambda)$.
}
\label{fig:income-it-exp2-mean}
\end{figure}

\paragraph*{\textbf{SFPCA w.r.t. Lebesgue measure}} SFPCA was performed by considering the Lebesgue reference measure as in \cite{hron16}. The results are reported in the first column of Figure \ref{fig:sfpca-income-leb-exp2-mean}. Figure \ref{fig:sfpca-income-leb-exp2-mean} displays the $\clr_\lambda$ transform of the first two SFPCs. The first clr-SFPC is interpreted as a contrast between the bottom band of the income distribution (i.e., income lower than 36.6 {k\euro}) and the rest. The second SFPC still contrasts low against high incomes, but provides further insight into differences in the central band of the distribution (i.e. for middle income values). These findings are also well reflected in Figure \ref{fig:sfpca-income-leb-exp2-mean}c-d, which displays variation along the first and second SFPCs respectively with respect to the sample mean. Having fixed the sign of the clr-SFPC as in Figure \ref{fig:sfpca-income-leb-exp2-mean}b, high scores along the first principal direction are predominantly associated with regions characterized by more low-income households than the average and, conversely, low scores are expected for high-income regions. Similarly, Figure \ref{fig:sfpca-income-leb-exp2-mean}d supports the interpretation of the second principal direction. From Figure \ref{fig:sfpca-income-leb-exp2-mean}a, the first SFPC can be clearly associated with geographical location, as the northern and central regions (higher incomes) appear well separated from the southern regions (lower incomes) along this direction. Finally, the approximation of smoothed density data using only the first SFPC is shown in Figure \ref{fig:sfpca-income-leb-exp2-mean}e. Comparing this with the actual data (Figure \ref{fig:map-it-income} right), the goodness of the approximation can be appreciated.

\paragraph*{\textbf{wSFPCA w.r.t. exponential measure}} An exponential reference measure $\mathsf{P}=\mathsf{P}^{\delta}$ was used in order to emphasize the relative scale of income values, with $\delta$ optimizing a data-driven criterion. In particular, $\delta$ maximises regional discrimination along the first principal directions. We remark that other criteria may be of interest, e.g. one may want to attain a certain rate of explained variability by the first SFPCs, or to select the reference measure that best fits the data. Following our criterion, Table \ref{table1} presents the classic decomposition of the total sum of squares (SS$_T$) into between-groups (SS$_B$) and within-groups (SS$_W$) sum of squares when the scores for the wSFPC$_1$ using $\mathsf{P}=\mathsf{P}^\delta$ are modeled via a one-way analysis of variance (ANOVA), using the Italian regions (north, center, south) as factor. Amongst the tested reference measures $\mathsf{P}^{\delta}$, we selected the one associated with the highest ratio SS$_B$/SS$_T$ (i.e. the highest discrimination between groups), which is $\delta=3\times 10^{-5}$. Note that we could otherwise consider a Fisher's canonical direction as in ordinary discriminant analysis, which provides the direction of maximum discrimination between groups.

\begin{table}
\centering
\begin{tabular}{ l || c | c | c|| r}
                & SS$_B$ &  SS$_W$  & SS$_T$ & SS$_B$/SS$_T$  \\
\hline	
	Uniform                    & 1.5030 & 0.7130 & 2.2160 & 0.6782  \\
	Exp($1.5\times 10^{-5}$)   & 2.1811 & 0.8425 & 3.0236 & 0.7214 \\
	Exp($3\times 10^{-5}$)     & 2.3631 & 0.9111 & 3.2742 & \textbf{0.7217} \\
	Exp($6\times 10^{-5}$)     & 1.8540 & 0.9128 & 2.7668 &  0.6701 \\
	Exp($1.2\times 10^{-4}$)   & 0.8609 & 0.8515 & 1.7124 &  0.5027 \\
	\hline
\end{tabular}
\caption{ANOVA sum of squares decomposition for the first SFPC scores based on uniform and exponential reference measures using region as factor. The exponential measures were $\mathsf{P}^\delta$, with $\delta \in \left\{1.5\times 10^{-5}, 3\times 10^{-5}, 6\times 10^{-5}, 1.2\times 10^{-4}\right\}$.}
\label{table1}
\end{table}

SFPCA was performed on the dataset consisting of exponentially weighted distributions (Figures \ref{fig:income-it-exp2-mean}a-b). The results are reported in the second column of Figure \ref{fig:sfpca-income-leb-exp2-mean}. The score plot (\ref{fig:sfpca-income-leb-exp2-mean}a) shows that the configuration of the scores well represents geographical locations, even though it is somehow similar to the one obtained obtain using the Lebesgue reference measure. However, the amount of variability explained along the first two SFPCs is higher in comparison to the unweighted case (Figure \ref{fig:sfpca-income-leb-exp2-mean}b).

It is worth noticing that the interpretation of the wSFPCs appears to be affected by the change in the reference measure. Indeed, although the first SFPC (Figure \ref{fig:sfpca-income-leb-exp2-mean}b) still represents a contrast between low and high incomes households, the second SFPC displays a contrasts within the low-income group. This could prompt further interesting economic interpretation, as this contrast might be related with an unequal redistribution of wealth within the lower income class. For instance, the (annual) poverty threshold in 2008 for a household of two members was 11.8~{k\euro}, which is roughly half of the zero-crossing of SFPC$_1$ and close to the zero-crossing of SFPC$_2$. Hence, weighting according to the relative scale of income data could help to signaling unequal redistribution of wealth, particularly amongst the low-income population.

{\paragraph*{\textbf{wSFPCA w.r.t. the sample mean}} A different view is obtained when the reference measure is set to the sample mean of the data $\bar{f}$  (density w.r.t. Lebesgue reference measure), computed as
\[
\bar{f}(t) = \frac{1}{N} \odot \bigoplus_{i=1}^N f_i(t), \qquad t \in I.\]
Recall that the reference measure determines the origin of the space $\mathcal{B}^2(\mathsf{P})$, which is a $\mathsf{P}$-density represented by a constant function. This is unchanged when mapped to $\mathcal{B}^2(\lambda)$ through the $\mathcal{B}^2$-unweighting map $\omega^{-1}$. For this reason, the sample weighted mean density in $\mathcal{B}^2(\lambda)$ appears as an uniform density in Figure \ref{fig:income-it-exp2-mean}d. In this case, the representation of the $\mathcal{B}^2$-unweighted data (Figure \ref{fig:income-it-exp2-mean}d) provides additional information about the dispersion of income distributions around their mean. Note that this has to be interpreted as usual (unweighted) PDFs. The distributions vary in different ways across regions: the income distributions in southern regions tend to be more concentrated than the average around low income (the average being represented as a uniform distribution), and they are less concentrated for higher incomes. The opposite is observed for northern and central regions. This is also well reflected by the wSFPCA output which is summarized in the third column of Figure \ref{fig:sfpca-income-leb-exp2-mean}.

The first wSFPC -- depicted in panel (b) -- still contrasts the bottom of the distributions (income below 25.23 k\euro) against their middle and top. The second wSFPC shows differences especially between the middle band of the distributions (income in $\left[18.77, 45.06\right]$ k\euro) and the top band (income over 45.06 k\euro). Note that a higher dispersion of the scores wSFPC$_2$ is observed for northern and central regions in relation to southern regions, which appear almost constant along the second mode of variation. In fact, wSFPC$_2$ seems to reveal a different distribution of wealth in the central band of the income distributions. Lombardia and Friuli regions tend to concentrate more medium-high incomes than the mean. Contrarily, the Valle d'Aosta region is characterized by low-medium incomes, appearing as an outlier along wSFPC$_2$.
The approximation of the sampled income distributions by the first SFPC (capturing almost 80\% of the total variability)  well reflects the data structure, as can be appreciated by comparing Figure \ref{fig:income-it-exp2-mean}d and Figure \ref{fig:sfpca-income-leb-exp2-mean}e.

\begin{figure}
    \centering
	\begin{subfigure}[t]{1\textwidth}
        \centering {\includegraphics[width=0.33\textwidth]{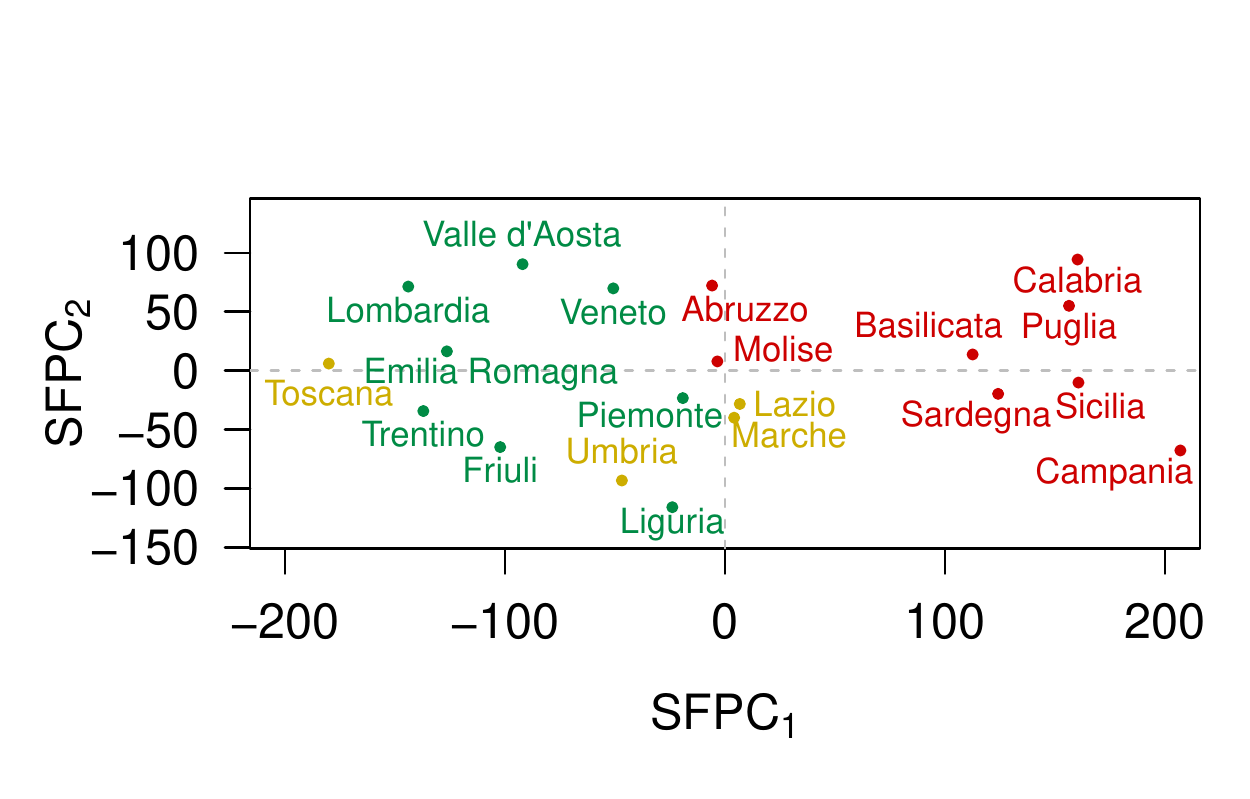}\includegraphics[width=0.33\textwidth]{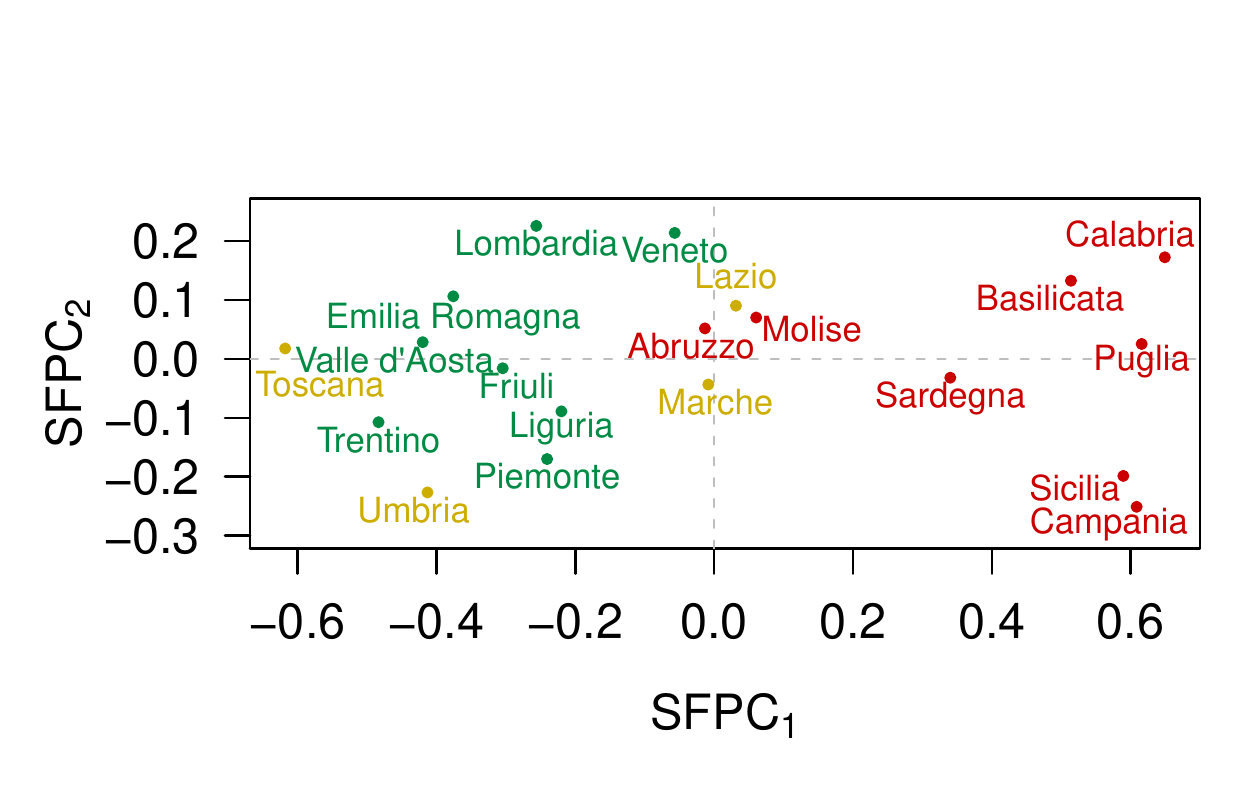}\includegraphics[width=0.33\textwidth]{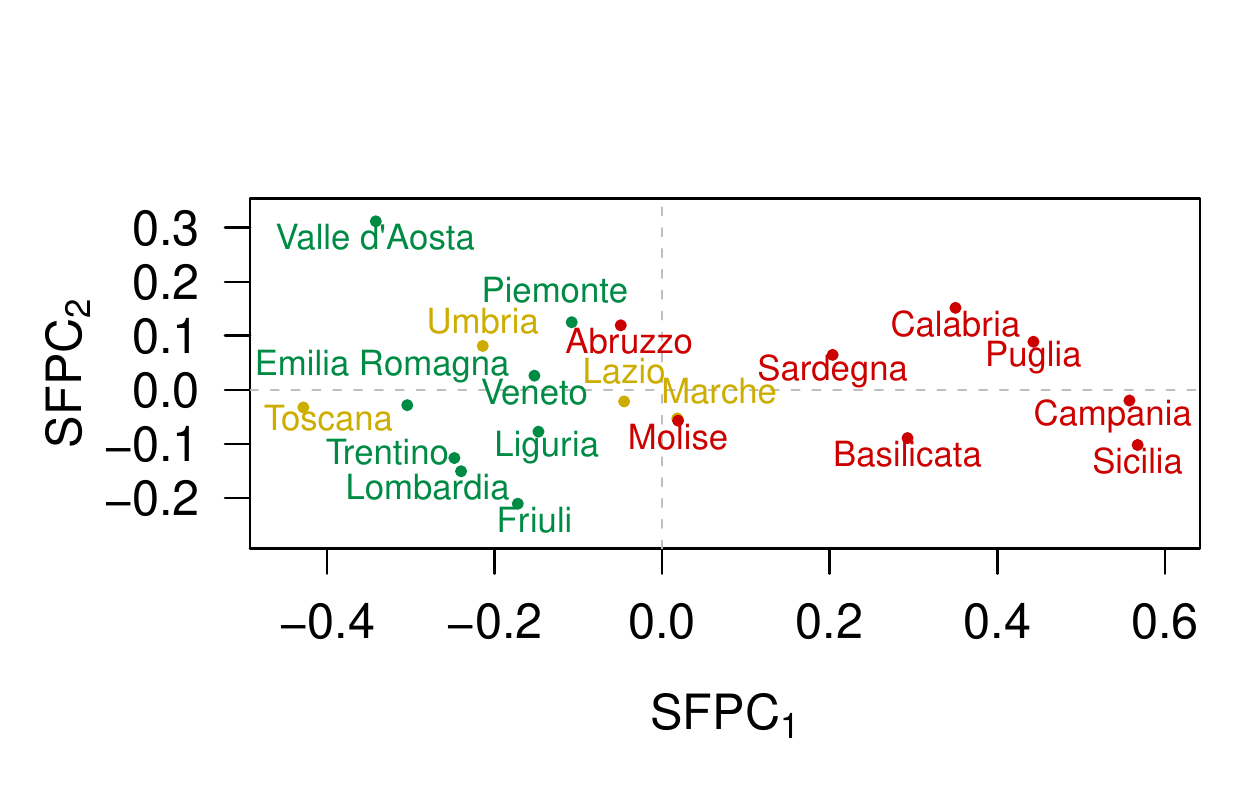}}
		\caption{Scores for SFPC$_1$ and SFPC$_2$}
	\end{subfigure}
	\begin{subfigure}[t]{1\textwidth}	\includegraphics[width=0.33\textwidth]{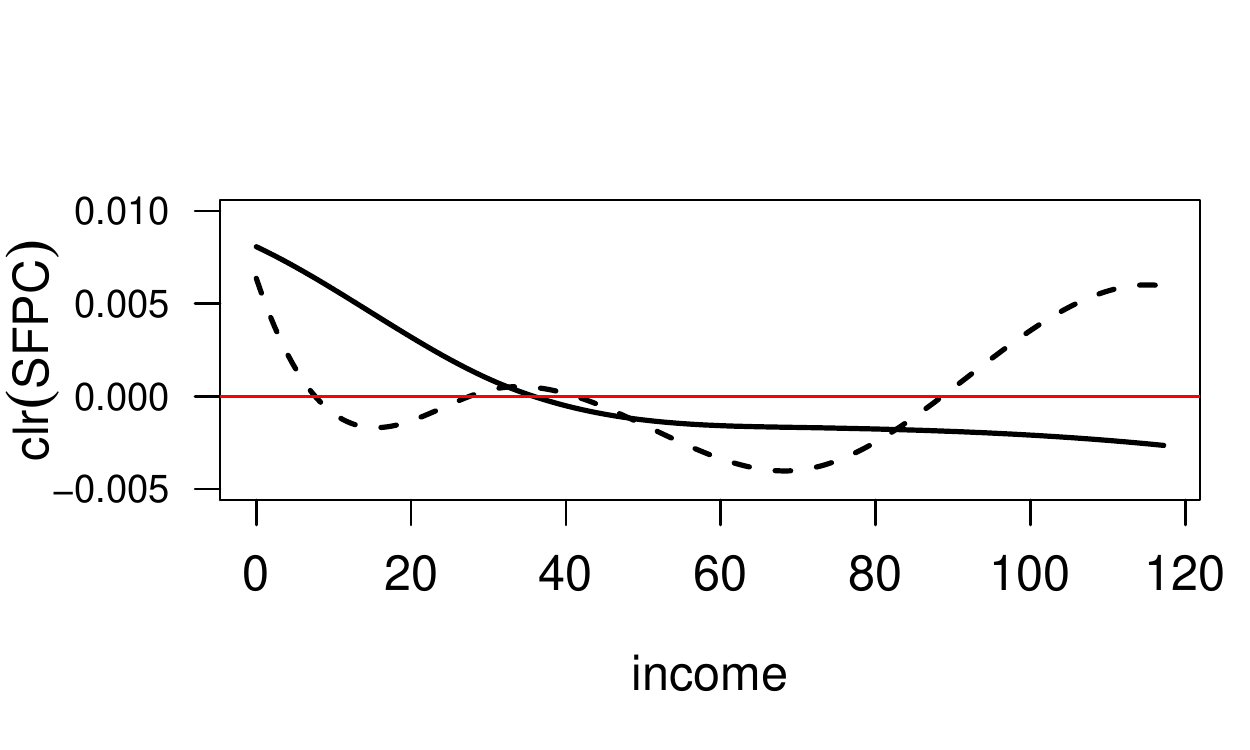}\includegraphics[width=0.33\textwidth]{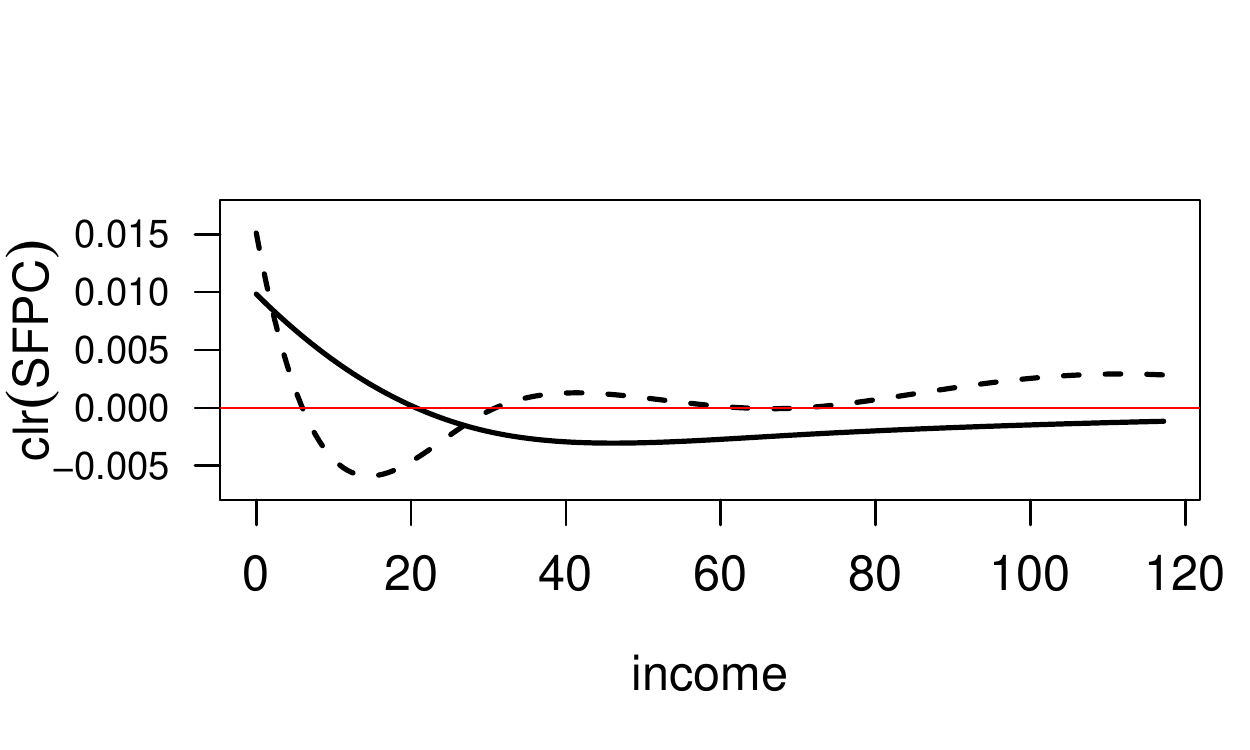}\includegraphics[width=0.33\textwidth]{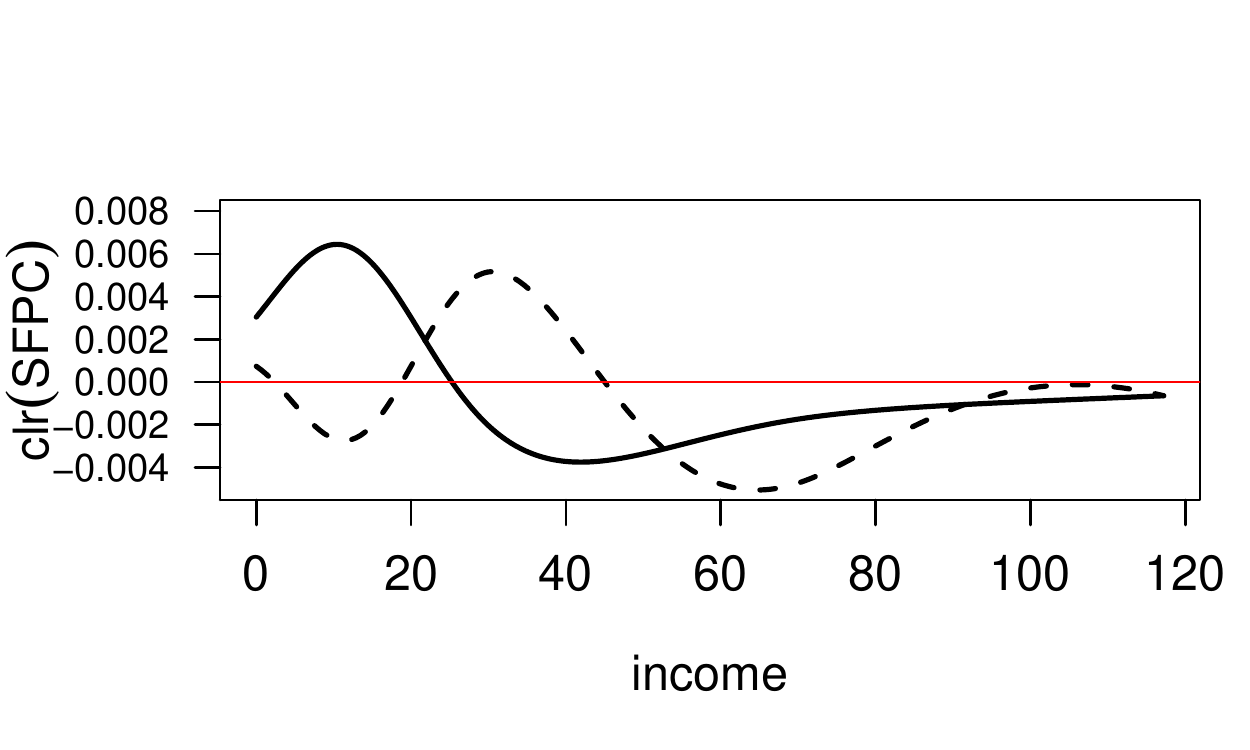}
		\caption{$\clr_u$ transform of the wSFPC$_1$ (solid line; explained variability: 66.08, 80.99, 79.93) and of the wSFPC$_2$ (dashed line; explained variability: 18.14, 9.35, 13.22).}	
	\end{subfigure}
	\begin{subfigure}[t]{1\textwidth} \includegraphics[width=0.33\textwidth]{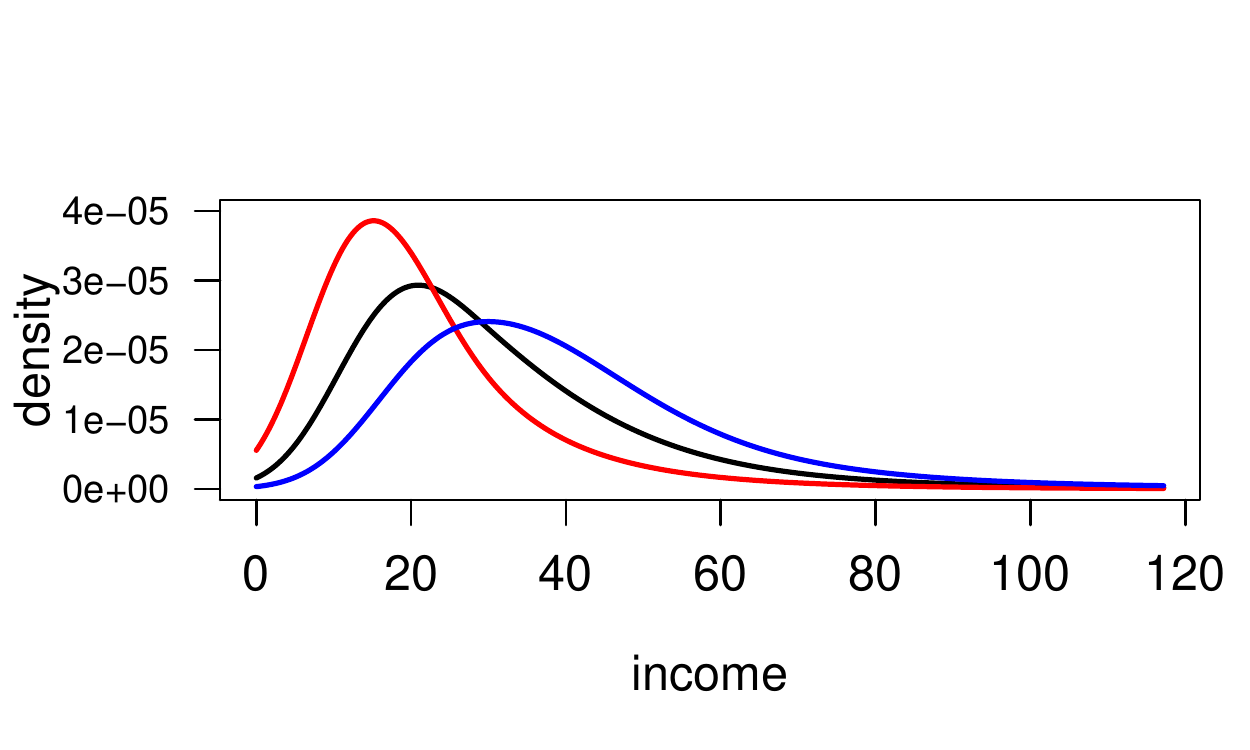}\includegraphics[width=0.33\textwidth]{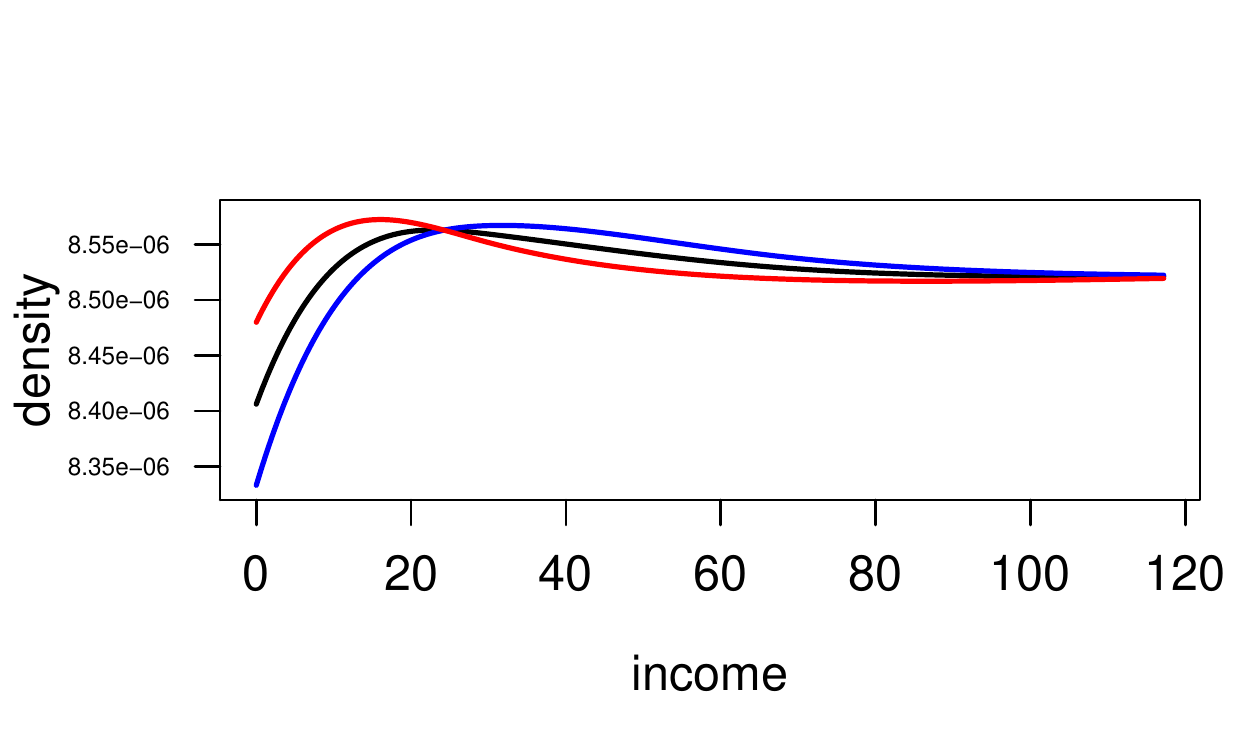}\includegraphics[width=0.33\textwidth]{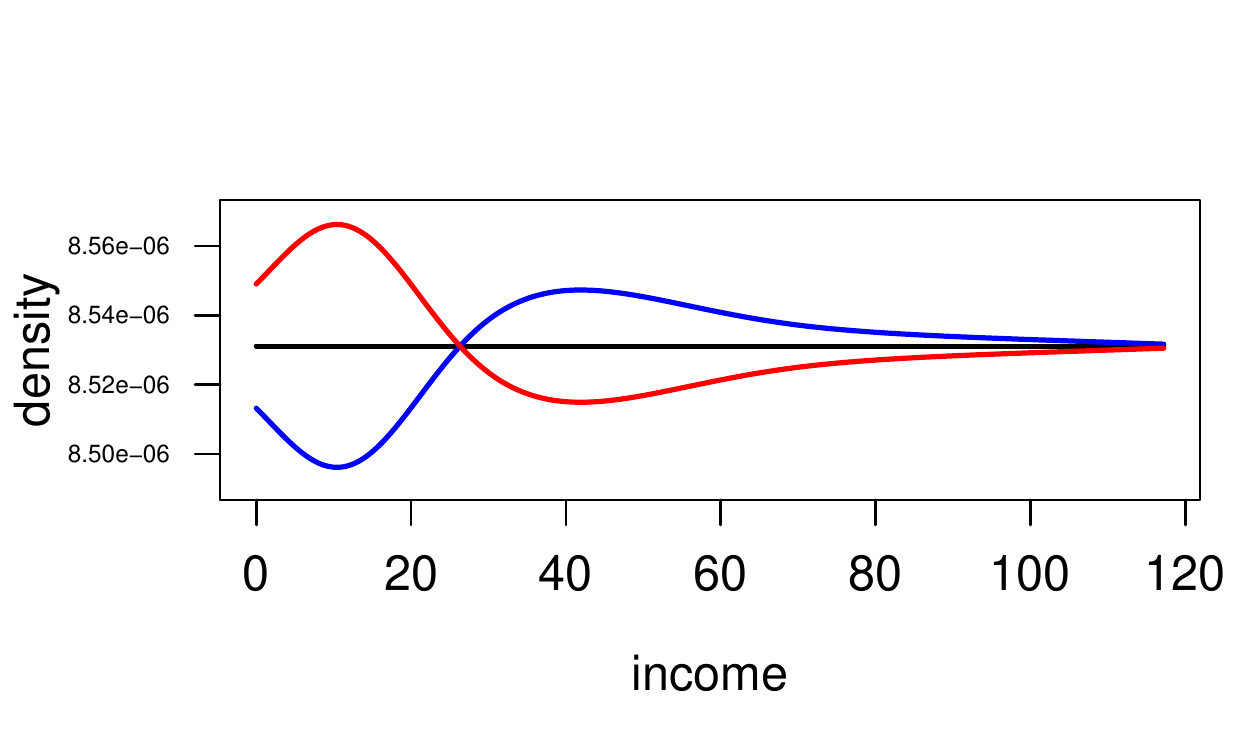}	
		\caption{$\mathcal{B}^2$-unweighted version of $\bar{f}_{\mathsf{P}} \oplus_{\mathsf{P}} / \ominus_{\mathsf{P}} 2\sqrt{\rho_1} \odot_{\mathsf{P}} \text{wSFPC}_{1}$ in {$\mathcal{B}^2(\lambda)$}, with $\mathsf{P}=\lambda, \mathsf{P}^\delta, \mathsf{P}^m$}	
	\end{subfigure}
	\begin{subfigure}[t]{1\textwidth}\includegraphics[width=0.33\textwidth]{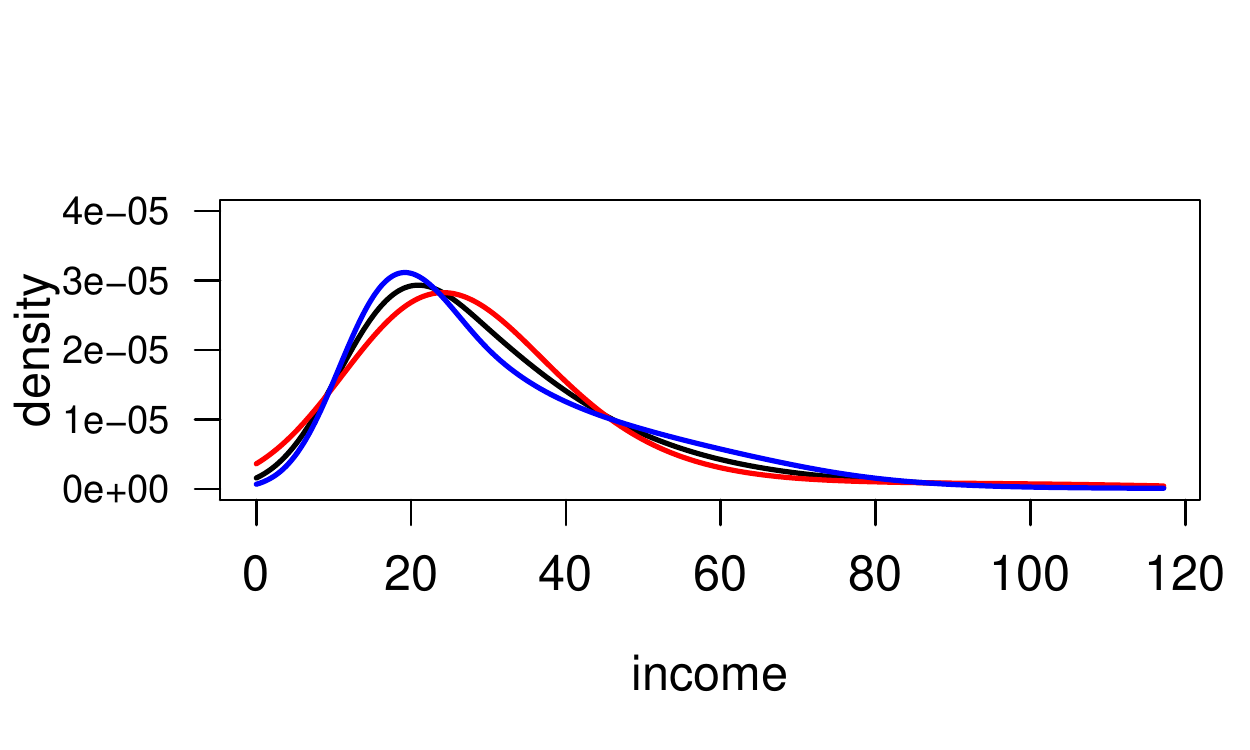}\includegraphics[width=0.33\textwidth]{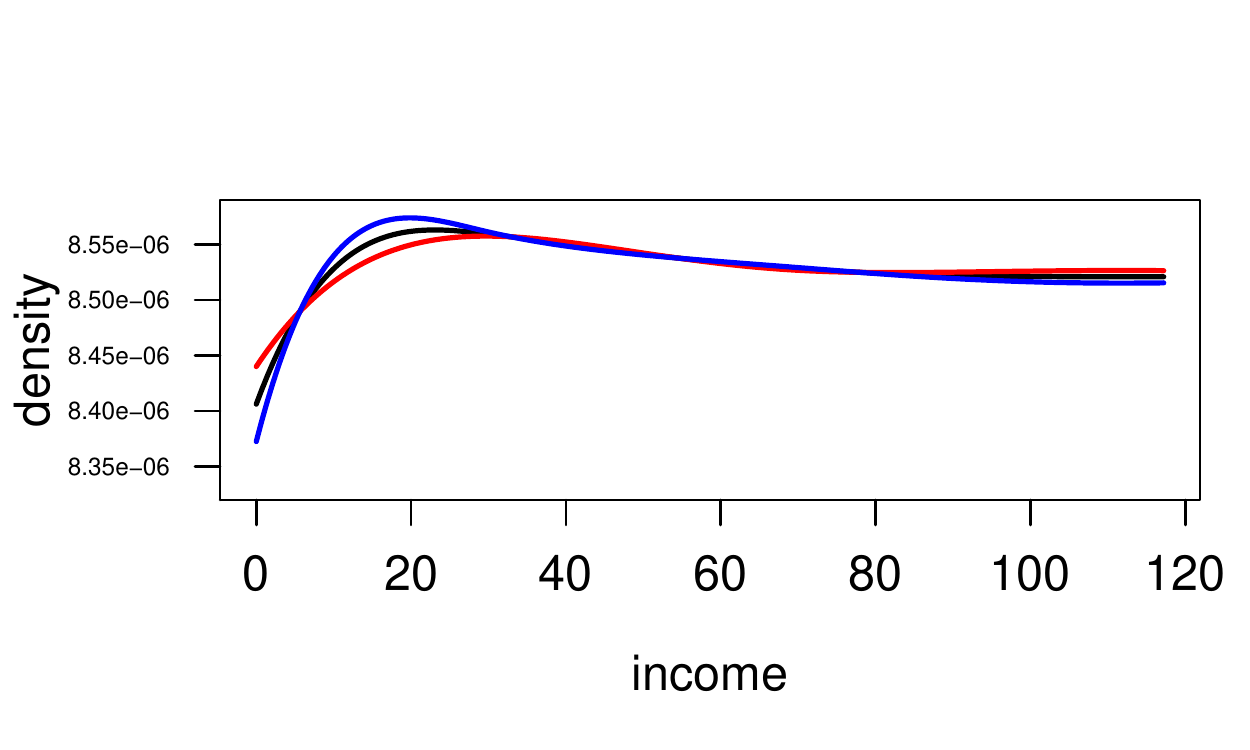}\includegraphics[width=0.33\textwidth]{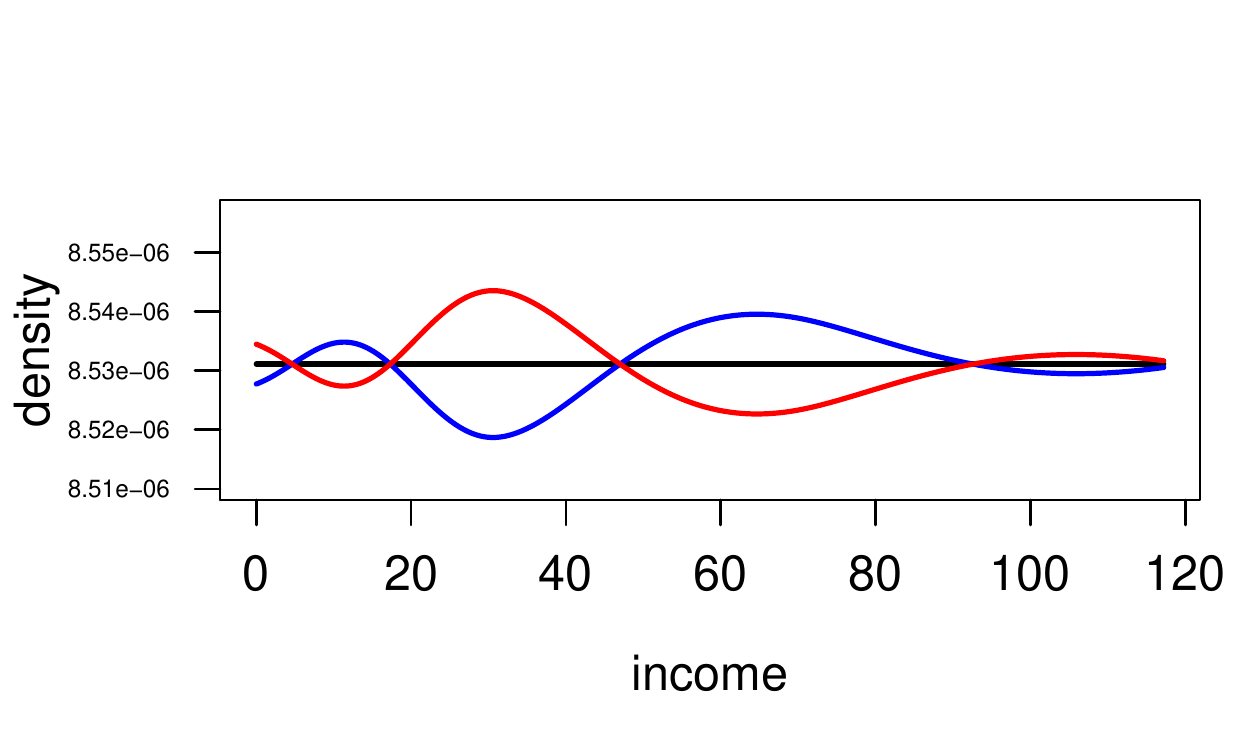}
		\caption{$\mathcal{B}^2$-unweighted version of $\bar{f}_{\mathsf{P}} \oplus_{\mathsf{P}} / \ominus_{\mathsf{P}} 2\sqrt{\rho_2} \odot_{\mathsf{P}} \text{wSFPC}_{2}$ in {$\mathcal{B}^2(\lambda)$}, with $\mathsf{P}=\lambda, \mathsf{P}^\delta, \mathsf{P}^m$.}				
	\end{subfigure}
	\begin{subfigure}[t]{1\textwidth}
        \includegraphics[width=0.33\textwidth]{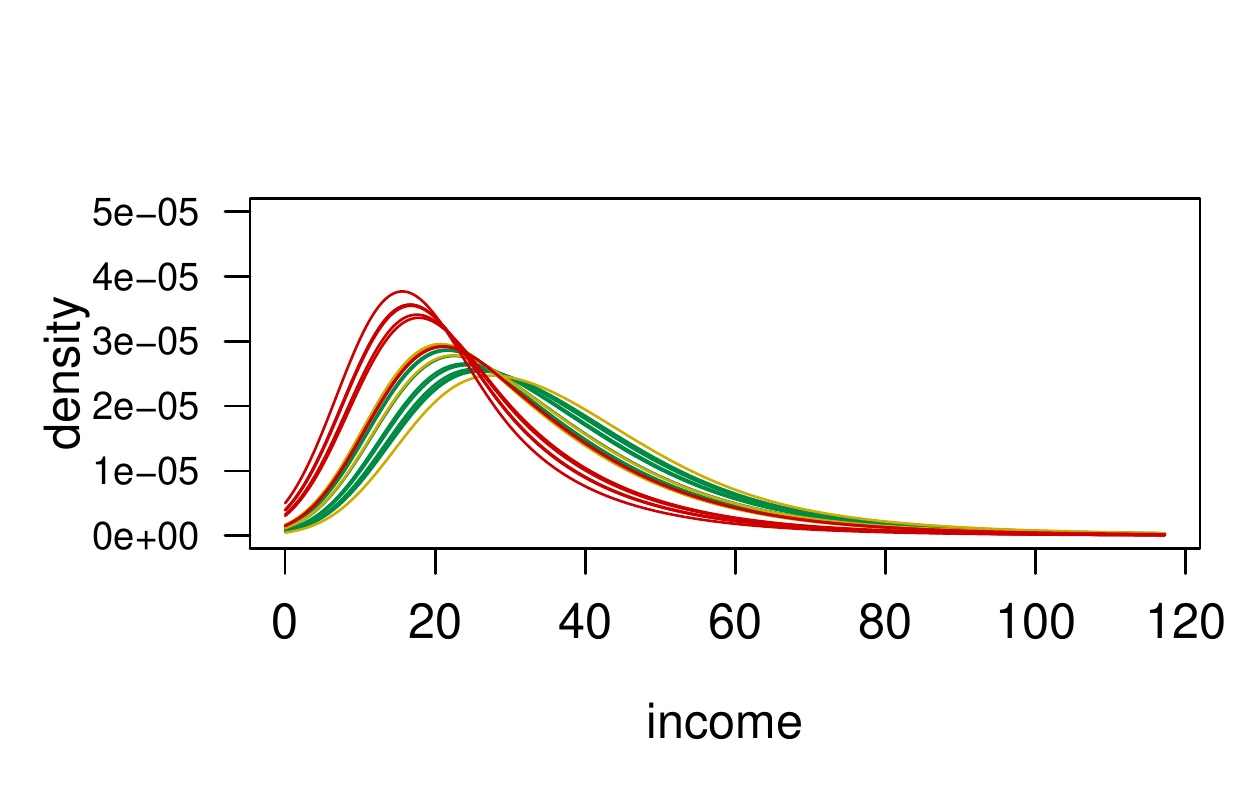}\includegraphics[width=0.33\textwidth]{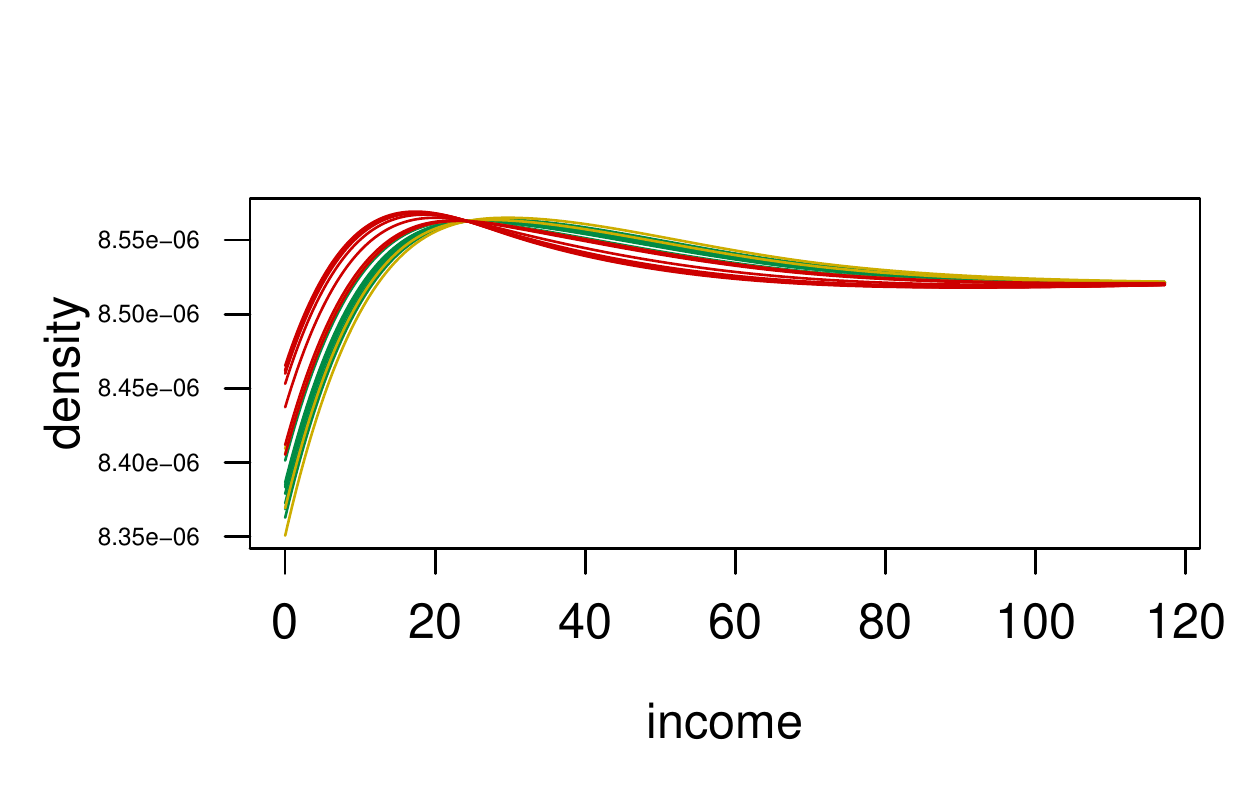}\includegraphics[width=0.33\textwidth]{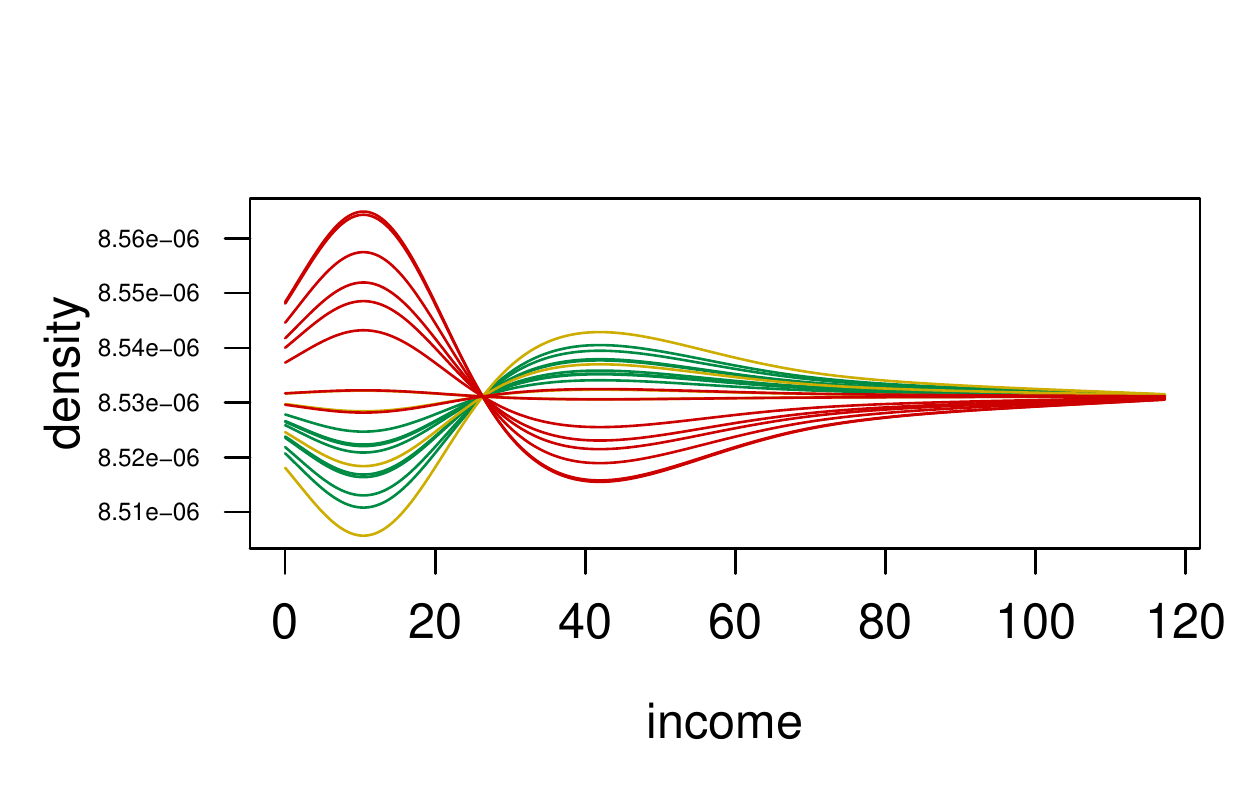}
		\caption{$\mathcal{B}^2$-unweighted version of the approximation of the $\mathsf{P}$-densities via SFPC$_1$.}
    \end{subfigure}
\caption{SFPCA results for income densities in Italian regions in case of reference measure set to (1) $\mathsf{P}=\lambda$ the Lebesgue measure (first column), (2) $\mathsf{P}=\mathsf{P}^\delta$ the exponential measure with $\delta = 3\times 10^{-5}$ (second column) and (3) $\mathsf{P}=\mathsf{P}^m$ the mean measure (third column). By $\mathcal{B}^2$-unweighted version of $f\in \mathcal{B}^2({\mathsf P})$ it is meant $\omega^{-1}(f_{\mathsf P}) \in \mathcal{B}^2(\lambda)$.
}
\label{fig:sfpca-income-leb-exp2-mean}
\end{figure}

\section{Conclusions}\label{sec:conclusion}

A novel weighting approach for probability density functions was proposed, based on changing the reference measure of Bayes Hilbert spaces. This new weighting scheme was provided within an original mathematical framework, where weighted Bayes spaces were linked to unweighted $\mathcal{B}^2$ and $L^2$ spaces. The key advantage of representing weighted densities in an unweighted space is the possibility of (i) making comparisons between densities derived from different weighting criteria, and (ii) visually interpret the results through ordinary `unweighted eyes'.
In fact, the proposed framework allows to perform varied kinds of analyses in weighted Bayes spaces, by using equivalent unweighted methods, which are typically developed for functional data analysis. Even though this strategy has been here demonstrated by extending a dimensionality reduction method (SFPCA) to the weighted case, other methods could be considered as well, such as regression \cite{talska18} or spatial prediction techniques \cite{menafoglio14}.
We finally stress that considering different weighting schemes can be particularly relevant in statistical applications, as they allow to account for different degrees of uncertainty across the domain of the data and for prior knowledge about the phenomenon to be encoded through a reference measure.

\section*{Acknowledgements}

The authors gratefully acknowledge both the support by the Czech Science Foundation (GACR), GA 19-01768S and the Internal Grant Agency of Palack\'y University, IGA PrF\_2018\_024.
The authors are also grateful to
Peter Filzmoser and Piercesare Secchi for comments that greatly improved
the manuscript.



\end{document}